\documentclass{amsart}
\usepackage[foot]{amsaddr}

\usepackage[latin1]{inputenc}
\usepackage{amssymb,amsmath,amsthm}
\usepackage{amsfonts}
\usepackage{enumitem}
\usepackage{booktabs}
\usepackage{ifthen}
\usepackage{tikz}
\usetikzlibrary{trees,matrix,calc,decorations.markings,decorations.pathreplacing,positioning}
\usepackage{forest}
\usepackage{url}
\usepackage{hyphenat}
\usepackage{hyperref}
\usepackage{supertabular}
\usepackage{multicol}
\usepackage{caption}
\usepackage[figuresright]{rotating}

\theoremstyle{plain}
\newtheorem{theorem}{Theorem}[section]
\newtheorem{proposition}[theorem]{Proposition}
\newtheorem{lemma}[theorem]{Lemma}
\newtheorem{corollary}[theorem]{Corollary}

\numberwithin{equation}{section}
\numberwithin{table}{section}
\numberwithin{figure}{section}
\makeatletter\let\c@figure\c@table\makeatother

\theoremstyle{definition}
\newtheorem{definition}[theorem]{Definition}
\newtheorem{example}[theorem]{Example}
\newtheorem{problem}[theorem]{Problem}

\newtheorem{remark}[theorem]{Remark}

\newcommand{\IN}{\ensuremath{\mathbb{N}}}
\newcommand{\ZZ}{\ensuremath{\mathbb{Z}}}
\newcommand{\RR}{\ensuremath{\mathbb{R}}}
\newcommand{\CC}{\ensuremath{\mathbb{C}}}

\newcommand{\nset}[1]{\ensuremath{[{#1}]}}

\newcommand{\divides}{\ensuremath{\mathrel{\vert}}}
\newcommand{\card}[1]{\ensuremath{\lvert{#1}\rvert}}

\newcommand{\gen}[1]{\ensuremath{\langle{#1}\rangle}}

\newcommand{\alg}[1]{\ensuremath{\mathbf{#1}}}

\newcommand{\satisfies}{\models}

\newcommand{\eqD}[1]{\ensuremath{\sim^\mathrm{d}_{#1}}}
\newcommand{\eqL}[1]{\ensuremath{\sim^\mathrm{L}_{#1}}}
\newcommand{\eqR}[1]{\ensuremath{\sim^\mathrm{R}_{#1}}}
\newcommand{\eqLR}[2]{\ensuremath{\sim^\mathrm{LR}_{{#1},{#2}}}}
\newcommand{\eqabm}[3]{\ensuremath{\sim^\mathrm{lin}_{{#1},{#2},{#3}}}}
\newcommand{\eqabG}[3]{\ensuremath{\sim^{#3}_{{#1},{#2}}}}

\newcommand{\NeqD}[2]{\ensuremath{T^\mathrm{d}_{{#1},{#2}}}}
\newcommand{\NeqL}[2]{\ensuremath{T^\mathrm{L}_{{#1},{#2}}}}
\newcommand{\NeqR}[2]{\ensuremath{T^\mathrm{R}_{{#1},{#2}}}}
\newcommand{\NeqLR}[3]{\ensuremath{T^\mathrm{LR}_{{#1},{#2},{#3}}}}
\newcommand{\Neqabm}[4]{\ensuremath{T^\mathrm{lin}_{{#1},{#2},{#3},{#4}}}}
\newcommand{\NeqabG}[4]{\ensuremath{T^{#3}_{{#1},{#2},{#4}}}}

\newcommand{\FS}{\alg{FS}_{\rm linqgr}}

\DeclareMathOperator{\id}{id}
\DeclareMathOperator{\lcm}{lcm}
\DeclareMathOperator{\var}{var}

\DeclareMathOperator{\Aut}{Aut}

\DeclareMathOperator{\addr}{\alpha}
\DeclareMathOperator{\ord}{ord}
\DeclareMathOperator{\Sub}{Sub}

\newcommand{\comb}{\ensuremath{\mathrm{comb}}}

\newcommand{\ld}{\ensuremath{\lambda}}
\newcommand{\rd}{\ensuremath{\rho}}

\tikzstyle{puu}=[dot/.style={inner sep=0.0625cm, circle, draw}, every node/.style={dot,fill},
  level distance=0.5cm,
  sibling distance=0.5cm
]

\tikzstyle{puup}=[dot/.style={inner sep=0.0625cm, circle, draw}, every node/.style={dot,fill},
  level distance=0.5cm,
  level3/.style={level distance=1cm},
  sibling distance=0.5cm,
  piste/.style={edge from parent/.style={dashed,draw}},
  viiva/.style={edge from parent/.style={solid,draw}}
]

\newcommand{\pairs}[3]{\ensuremath{\Lambda_{#1}(#2,#3)}}

\begin{document}
\title[Associativity conditions for linear quasigroups]{Associativity conditions for linear quasigroups and equivalence relations on binary trees}

\author{Erkko Lehtonen}

\address[E. Lehtonen]%
   {Institute for Algebra \\
    Johannes Kepler University Linz \\
    Altenberger Stra\ss{}e 69 \\
    4040 Linz \\
    Austria}

\author{Tam\'{a}s Waldhauser}

\address[T. Waldhauser]%
   {Bolyai Institute \\
    University of Szeged \\
    Aradi v\'{e}rtan\'{u}k tere 1 \\
    H6720 Szeged \\
    Hungary}

\thanks{This research was partially supported by the National Research, Development and Innovation Office of Hungary under grants no. K128042 and K138892, and by project TKP2021-NVA-09, implemented with the support provided by the Ministry of Innovation and Technology of Hungary from the National Research, Development and Innovation Fund, financed under the TKP2021-NVA funding scheme.}
\date{November 20, 2026}

\begin{abstract}
We characterise the bracketing identities satisfied by linear quasigroups with the help of certain equivalence relations on binary trees that are based on the left and right depths of the leaves modulo some integers. The numbers of equivalence classes of $n$\hyp{}leaf binary trees are variants of the Catalan numbers, and they form the associative spectrum (a kind of measure of non\hyp{}associativity) of a quasigroup.
Every group-based equivalence relation arising here can be replaced by one defined by an abelian group generated by at most two elements. The fine associative spectra form a lattice isomorphic to the lattice of rank-two subgroups of $\ZZ^2$ with an additional least element. This classifies the equivalence relations completely; closed formulas for all their numbers of classes remain open.
\end{abstract}

\maketitle


\section{Introduction}

If a binary operation $\circ$ is associative, then, by the generalised associative law, all bracketings of the formal product $x_1 \circ \dots \circ x_n$ give the same result.
Here, by a \emph{bracketing} we mean any expression that is obtained from $x_1 \circ \dots \circ x_n$ by inserting round brackets (parentheses) in a syntactically correct way to determine its value unambiguously.
However, even if our operation is not associative, it may still satisfy some nontrivial \emph{bracketing identities} in more than three variables, such as the identity $(x_1-(x_2-x_3))-x_4 \approx x_1-(x_2-(x_3-x_4))$ satisfied by subtraction.
On the other hand, it may also happen that an operation satisfies no such generalised associativity conditions at all; for example, exponentiation of positive integers is \emph{antiassociative} in this sense \cite[Example 3]{Lord-1987} and so is the logical implication operation \cite[4.3]{CsaWal-2000}, just to mention two noteworthy examples.
(Actually, it seems plausible that almost all binary operations on a given finite set are antiassociative \cite{CsaWal-2000}.)

In this paper we study associativity conditions of this kind for quasigroups.
Belousov \cite[Theorem 4]{Belousov-1966} proved that if a quasigroup $\alg{A}=(A, {\circ})$ satisfies a nontrivial bracketing identity, then we have $x \circ y = \varphi_0(x)+c+\varphi_1(y)$, where $(A, {+})$ is a group, $\varphi_0, \varphi_1$ are automorphisms of this group and $c \in A$ is an arbitrary constant.
In this case we say that $\alg{A}=(A, {\circ})$ is an \emph{affine quasigroup} over the group $(A, {+})$.
(Let us note that this result of Belousov's considers not bracketing identities, but so-called irreducible balanced identities of the first kind. 
Balanced identities of the first kind are the same as bracketing identities, and it is straightforward to verify that any quasigroup satisfying a nontrivial bracketing identity also satisfies an irreducible balanced identity of the first kind.)
It was also proved by Belousov that the automorphisms $\varphi_0$ and $\varphi_1$ have to be of finite order in order to have a nontrivial associativity condition.
We will also obtain this finite-order criterion for linear quasigroups as a corollary of our main theorem (see Corollary~\ref{cor:antiassociative}).

In light of Belousov's theorem, in order to understand associativity conditions of quasigroups, it suffices to focus on affine quasigroups.
We deal with the case $c=0$ here, i.e., with \emph{linear quasigroups} over groups; thus we consider operations of the form $x \circ y = \varphi_0(x)+\varphi_1(y)$.
Subtraction and arithmetic mean of numbers are simple examples of linear quasigroup operations, and, more generally, $(\CC, ax+by)$ is a linear quasigroup for any nonzero complex numbers $a$ and $b$.
Our main result is a complete description of bracketing identities satisfied by such operations (Theorem~\ref{thm:spectra}).
The description is given in terms of certain equivalence relations on binary trees, as one can naturally identify each bracketing of $x_1 \circ \dots \circ x_n$ with a binary tree with $n$ leaves.
We define an equivalence relation $\eqabG{a}{b}{\alg{G}}$ on binary trees for any group $\alg{G}$ and $a,b \in G$, and we prove that the bracketing identities satisfied by a linear quasigroup can always be given by such a relation.
Then we show that every such ``modulo $\alg{G}$" equivalence coincides with one where $\alg{G}$ is a two-generated abelian group, and finally we conclude that the bracketing identities satisfied by a linear quasigroup can be described by at most two linear congruences involving left and right depths of leaves in binary trees.
Theorem~\ref{thm:FS} further identifies the lattice of fine associative spectra with the rank-two subgroups of $\ZZ^2$, together with an additional least element. Our classification concerns the actual equivalence relations, not just their sequences of numbers of classes.

The number of equivalence classes on the set of binary trees with $n$ leaves gives the number of term operations induced by different bracketings of $x_1 \circ \dots \circ x_n$.
This sequence of natural numbers is called the \emph{associative spectrum} of the operation $\circ$, and it measures -- in some sense -- how far the operation is from being associative \cite{CsaWal-2000}:
the faster the sequence grows, the less associative the operation is considered.
The spectrum of an associative operation is constant $1$, whereas the spectrum of an antiassociative operation is the sequence of Catalan numbers.
Since associative spectra of linear quasigroups are determined by congruences for left and right depths of leaves in binary trees, these sequences can be regarded as interesting ``modular" variants of the Catalan numbers.
In fact, Hein and Huang defined $k$\hyp{}modular Catalan numbers as the sequence arising from the equivalence relation where left depths are required to agree modulo $k$ \cite{HeiHua-2017}.
We obtain a much larger class of equivalence relations, where left and right depths occur together in linear congruences with possibly different moduli. Many of the corresponding sequences were not identified in the OEIS.
We believe they may be of interest in their own right, as they are based on simple and fundamental relationships between binary trees.
We have computed the first few terms of the sequences, but unfortunately we were not able to explicitly describe the entire sequences.
Finding explicit formulas for the $n$\hyp{}th terms of such sequences remains an intriguing open problem.


\section{Preliminaries}

\subsection{Groupoids, groups, quasigroups}

We will use the following notation for familiar sets of natural numbers.
Let $\IN := \{0, 1, 2, \dots\}$ and $\IN_{+} := \IN \setminus \{0\}$.
For $n \in \IN$, let $\nset{n} := \{1,2,\dots,n\}$.

Recall that a \emph{groupoid} is an algebra $\alg{A} = (A, {\circ})$ with a single binary operation $\circ\colon A \times A \to A$, often referred to as \emph{multiplication} and usually written simply as juxtaposition.
A \emph{semigroup} is a groupoid with associative multiplication.
A \emph{monoid} is a semigroup with a neutral element.
A \emph{group} is a monoid in which every element is invertible.
A \emph{quasigroup} is a groupoid $(A, {\circ})$ such that for all $a, b \in A$, there exist unique elements $x, y \in A$ such that $a \circ x = b$ and $y \circ a = b$.
In other words, the multiplication table of $(A, {\circ})$ (for a finite set $A$) is a Latin square.
An associative quasigroup is a group.

\subsection{Bracketings, associative spectrum}

In this paper, we consider terms in the language of groupoids over the set $X := \{x_i \mid i \in \IN_{+}\}$, the so\hyp{}called 
standard set of variables.
Such \emph{terms} can be defined by the following recursion:
every variable $x_i \in X$ is a term, and
if $t_1$ and $t_2$ are terms, then $(t_1 t_2)$ is a term;
every term is obtained by a finite number of applications of these rules.

Let $\alg{A} = (A, {\circ})$ be a groupoid, and let $t$ be a term in which only (some of) the variables $x_1,\dots,x_n$ occur.
The \emph{term operation} induced by $t$ on $\alg{A}$ is the $n$-ary operation $t^{\alg{A}} \colon A^n \to A,\  (a_1,\dots,a_n) \mapsto t^{\alg{A}}(a_1,\dots,a_n)$, where $t^{\alg{A}}(a_1,\dots,a_n)$ is obtained by substituting $a_i$ for each $x_i$ and evaluating the expression over $\alg{A}$.

An \emph{identity} is a pair $(s,t)$ of terms, usually written as $s \approx t$; this identity is considered \emph{trivial} if $s = t$.
A groupoid $\alg{A} = (A, {\circ})$ \emph{satisfies} an identity $s \approx t$, in symbols, $\alg{A} \satisfies s \approx t$, if $s^{\alg{A}} = t^{\alg{A}}$, 
i.e., if $s^{\alg{A}}(a_1,\dots,a_n) = t^{\alg{A}}(a_1,\dots,a_n)$ for all $a_1,\dots,a_n \in A$ (assuming that only $x_1,\dots,x_n$ occur in $s$ and $t$).

A \emph{bracketing} of size $n$ is a term in the language of groupoids obtained by inserting pairs of parentheses in the string $x_1 x_2 \cdots x_n$ appropriately.
The number of distinct bracketings of size $n$ equals the $(n-1)$\hyp{}st Catalan number $C_{n-1}$.
We denote by $B_n$ the set of all bracketings of size $n$.
A \emph{bracketing identity} of size $n$ is an identity $t \approx t'$ where $t, t' \in B_n$.

Let $\alg{A} = (A, {\circ})$ be a groupoid.
For each $n \in \IN_+$, we define the equivalence relation $\sigma_n(\alg{A})$ on $B_n$ by the rule that $(t,t') \in \sigma_n(\alg{A})$ if and only if $\alg{A}$ satisfies the identity $t \approx t'$.
We call the sequence $\sigma(\alg{A})=(\sigma_n(\alg{A}))_{n \in \IN_{+}}$ the \emph{fine associative spectrum} of $\alg{A}$.
Observe that $\sigma(\alg{A})$ can be regarded as a single equivalence relation on the set of all bracketings with the property that each $B_n$ is a union of equivalence classes.
The \emph{associative spectrum} of $\alg{A}$ is the sequence $(s_n(\alg{A}))_{n \in \IN_+}$, where $s_n(\alg{A}) := \card{B_n / \sigma_n(\alg{A})}$.
Equivalently, $s_n(\alg{A})$ is the number of distinct term operations induced by the bracketings of size $n$ on $\alg{A}$.
We clearly have $1 \leq s_n(\alg{A}) \leq C_{n-1}$.
Since there is only one bracketing of size $1$, namely $x_1$, and only one bracketing of size $2$, namely $(x_1 x_2)$, it is obvious that $s_1(\alg{A}) = s_2(\alg{A}) = 1$ for every groupoid $\alg{A}$.
Therefore, we may always assume that $n \geq 3$ when we consider the $n$\hyp{}th component of an associative spectrum.

The associative spectrum can be seen as a measure of how far the groupoid operation is from being associative.
Intuitively, the faster the associative spectrum grows, the less associative the operation is considered.
If the operation is associative, then $s_n(\alg{A}) = 1$ for all $n \in \IN_+$.
At the other extreme, we have groupoids with ``Catalan spectrum", i.e., $s_n(\alg{A}) = C_{n-1}$ for $n \geq 2$; we call such groupoids \emph{antiassociative.}\footnote{This is not to be confused with the following property that is also often called \emph{antiassociativity:} for all $a, b, c \in A$, $a \circ (b \circ c) \neq (a \circ b) \circ c$.}
The notion of associative spectrum was introduced by Cs\'{a}k\'{a}ny and Waldhauser \cite{CsaWal-2000},
and it appears in the literature under different names, such as
``subassociativity type'' (Braitt, Silberger \cite{BraSil-2006}),
and
``the number of $*$\hyp{}equivalence classes of parenthesisations of $x_0 * x_1 * \dots * x_n$'' (Hein, Huang \cite{HeiHua-2017}, Huang, Mickey, Xu \cite{HuaMicXu-2017}).

\subsection{Binary trees}
\label{subs:binary-trees}

A \emph{tree} is a directed graph $T$ with a designated vertex called the \emph{root}, such that there is a unique directed walk from the root to every vertex $v$, including the root itself; the length\hyp{}zero walk is allowed.
Hence a tree is acyclic, and the edges are directed away from the root.
In this paper all trees are assumed to be finite and we draw trees in such a way that the root is on the top and edges are directed downwards; with this convention there is no need to indicate the direction of edges.
In a tree, the outneighbours of a vertex $v$ are called its \emph{children,} and $v$ is called the \emph{parent} of its children.
The vertices reachable from $v$ are called its \emph{descendants,} and $v$ is an \emph{ancestor} of any of its descendants.
A childless vertex is called a \emph{leaf;} non\hyp{}leaves are called \emph{internal vertices.}
A subgraph of a tree induced by a vertex $v$ and all its descendants is called the \emph{subtree} rooted at $v$.

An \emph{ordered tree} or \emph{plane tree} is a tree in which a linear ordering is specified for the children of each vertex.
We think of ordering the children from left to right, so that if $v$ has outdegree $k$ and its children are ordered as $u_0 < u_1 < \dots < u_{k-1}$, then $u_0$ is the leftmost child and $u_{k-1}$ is the rightmost child of $v$.
Diagrams presenting plane trees shall be drawn in such a way that the children of a vertex are drawn left\hyp{}to\hyp{}right; such a drawing uniquely specifies the ordering of children.

A \emph{binary tree} is a plane tree in which every internal vertex has exactly two children; the two children are referred to as the \emph{left child} and the \emph{right child.}
We denote by $\mathcal{T}_n$ the set of all (isomorphism classes of) binary trees with $n$ leaves, and we let $\mathcal{T} = \bigcup_{n \in \IN_{+}} \mathcal{T}_n$ stand for the set of all binary trees.
The subtree rooted at the left child (right child) of a vertex $v$ is referred to as the \emph{left} (\emph{right}) \emph{subtree} of $v$.

Let $T$ be a plane tree.
The \emph{address} of a vertex $v$ in $T$, denoted by $\addr_T(v)$, is a word over $\IN$ defined by the following recursion.
The address of the root is the empty word $\varepsilon$.
If $v$ is an internal node with address $w$, and the children of $v$ are $u_0 < u_1 < \dots < u_{k-1}$, then the address of the child $u_i$ is $w i$.
Thus, the address of a vertex conveys the sequence of choices of children made along the unique path from the root to the given vertex.

The length of the unique path from the root to a vertex $v$ in $T$ is called the (\emph{total}) \emph{depth} of $v$ in $T$ and is denoted by $d_T(v)$.
In a binary tree $T$, we also define
the \emph{left depth} of a vertex $v$ in $T$, denoted by $\ld_T(v)$, as the number of left steps on the unique path from the root of $T$ to $v$, i.e., the number of $0$'s in $\addr_T(v)$.
The \emph{right depth} of $v$ in $T$ is defined analogously and is denoted by $\rd_T(v)$.

\begin{example}
There are fourteen binary trees with five leaves; they are presented in Figure~\ref{fig:5-trees}.
Below each tree, we provide its left, right, and total depth sequences.
In tree $T_3$, the addresses of the leaves are $0$, $100$, $101$, $110$, and $111$.
\end{example}

\begin{figure}
\begin{center}
\makebox[0pt]{%
\begin{tabular}{ccccccc}
$T_1$ & $T_2$ & $T_3$ & $T_4$ & $T_5$ & $T_6$ & $T_7$ \\[2ex]
\begin{tikzpicture}[puu]
  \node {}
    child {node {}}
    child {node {}
      child {node {}}
      child {node {}
        child {node {}}
        child {node {}
          child {node {}}
          child {node {}}
        }
      }
    }
  ;
\end{tikzpicture}
&
\begin{tikzpicture}[puu]
  \node {}
    child {node {}}
    child {node {}
      child {node {}}
      child {node {}
        child {node {}
          child {node {}}
          child {node {}}
        }
        child {node {}}
      }
    }
  ;
\end{tikzpicture}
&
\begin{tikzpicture}[puu,level 1/.style={sibling distance=1cm},level 3/.style={sibling distance=0.5cm}]
  \node {}
    child {node {}}
    child {node {}
      child {node {}
        child {node {}}
        child {node {}}
      }
      child {node {}
        child {node {}}
        child {node {}}
      }
    }
  ;
\end{tikzpicture}
&
\begin{tikzpicture}[puu]
  \node {}
    child {node {}}
    child {node {}
      child {node {}
        child {node {}}
        child {node {}
          child {node {}}
          child {node {}}
        }
      }
      child {node {}}
    }
  ;
\end{tikzpicture}
&
\begin{tikzpicture}[puu]
  \node {}
    child {node {}}
    child {node {}
      child {node {}
        child {node {}
          child {node {}}
          child {node {}}
        }
        child {node {}}
      }
      child {node {}}
    }
  ;
\end{tikzpicture}
&
\begin{tikzpicture}[puu,level 1/.style={sibling distance=1cm},level 2/.style={sibling distance=0.5cm}]
  \node {}
    child {node {}
      child {node {}}
      child {node {}}
    }
    child {node {}
      child {node {}}
      child {node {}
        child {node {}}
        child {node {}}
      }
    }
  ;
\end{tikzpicture}
&
\begin{tikzpicture}[puu,level 1/.style={sibling distance=1cm},level 2/.style={sibling distance=0.5cm}]
  \node {}
    child {node {}
      child {node {}}
      child {node {}}
    }
    child {node {}
      child {node {}
        child {node {}}
        child {node {}}
      }
      child {node {}}
    }
  ;
\end{tikzpicture}
\\[2ex]
11110 & 11210 & 12110 & 12210 & 13210 & 21110 & 21210 \\
01234 & 01233 & 01223 & 01232 & 01222 & 01123 & 01122 \\
12344 & 12443 & 13333 & 13442 & 14432 & 22233 & 22332 \\[4ex]
$T_8$ & $T_9$ & $T_{10}$ & $T_{11}$ & $T_{12}$ & $T_{13}$ & $T_{14}$ \\[2ex]
\begin{tikzpicture}[puu]
  \node {}
    child {node {}
      child {node {}
        child {node {}
          child {node {}}
          child {node {}}
        }
        child {node {}}
      }
      child {node {}}
    }
    child {node {}}
  ;
\end{tikzpicture}
&
\begin{tikzpicture}[puu]
  \node {}
    child {node {}
      child {node {}
        child {node {}}
        child {node {}
          child {node {}}
          child {node {}}
        }
      }
      child {node {}}
    }
    child {node {}}
  ;
\end{tikzpicture}
&
\begin{tikzpicture}[puu,level 1/.style={sibling distance=1cm},level 3/.style={sibling distance=0.5cm}]
  \node {}
    child {node {}
      child {node {}
        child {node {}}
        child {node {}}
      }
      child {node {}
        child {node {}}
        child {node {}}
      }
    }
    child {node {}}
  ;
\end{tikzpicture}
&
\begin{tikzpicture}[puu]
  \node {}
    child {node {}
      child {node {}}
      child {node {}
        child {node {}
          child {node {}}
          child {node {}}
        }
        child {node {}}
      }
    }
    child {node {}}
  ;
\end{tikzpicture}
&
\begin{tikzpicture}[puu]
  \node {}
    child {node {}
      child {node {}}
      child {node {}
        child {node {}}
        child {node {}
          child {node {}}
          child {node {}}
        }
      }
    }
    child {node {}}
  ;
\end{tikzpicture}
&
\begin{tikzpicture}[puu,level 1/.style={sibling distance=1cm},level 2/.style={sibling distance=0.5cm}]
  \node {}
    child {node {}
      child {node {}
        child {node {}}
        child {node {}}
      }
      child {node {}}
    }
    child {node {}
      child {node {}}
      child {node {}}
    }
  ;
\end{tikzpicture}
&
\begin{tikzpicture}[puu,level 1/.style={sibling distance=1cm},level 2/.style={sibling distance=0.5cm}]
  \node {}
    child {node {}
      child {node {}}
      child {node {}
        child {node {}}
        child {node {}}
      }
    }
    child {node {}
      child {node {}}
      child {node {}}
    }
  ;
\end{tikzpicture}
\\[2ex]
43210 & 33210 & 32210 & 23210 & 22210 & 32110 & 22110 \\
01111 & 01211 & 01121 & 01221 & 01231 & 01112 & 01212 \\
44321 & 34421 & 33331 & 24431 & 23441 & 33222 & 23322 \\
\end{tabular}
}
\end{center}

\bigskip
\caption{The binary trees with five leaves.}
\label{fig:5-trees}
\end{figure}

The addresses of two consecutive leaves of a binary tree are related in the following way.

\begin{lemma}
\label{lem:addresses}
Let $T$ be a binary tree with leaves $1, 2, \dots, n$ in the left\hyp{}to\hyp{}right order.
Then for all $i \in \nset{n-1}$, $\addr_T(i) = u 0 1^p$ and $\addr_T(i+1) = u 1 0^q$ for some $p, q \in \IN$, where $u$ is the address of the deepest common ancestor of the leaves $i$ and $i + 1$.
\end{lemma}

\begin{proof}
Obvious, as the leaves $i$ and $i + 1$ are the rightmost leaf of the left subtree and the leftmost leaf of the right subtree, respectively, of the deepest common ancestor of $i$ and $i+1$.
\end{proof}

A binary tree is uniquely determined by the left (or right) depths of the leaves, as the following lemma reveals.

\begin{lemma}
\label{lem:da}
For any $T \in \mathcal{T}_n$ and $i \in \nset{n}$,
\begin{enumerate}[label={\upshape (\roman*)}]
\item\label{lem:da:l} the left depths $\ld_T(1), \dots, \ld_T(i)$ uniquely determine the addresses $\addr_T(1), \linebreak \dots, \addr_T(i)$, and
\item\label{lem:da:r} the right depths $\rd_T(i), \dots, \rd_T(n)$ uniquely determine the addresses $\addr_T(i), \linebreak \dots, \addr_T(n)$.
\end{enumerate}
\end{lemma}

\begin{proof}
We prove statement \ref{lem:da:l} by induction on $i$.
The proof of \ref{lem:da:r} is analogous.
The case $i = 1$ is clear, as $\addr_T(1) = 0^{\ld_T(1)}$.
Now let $i \in \nset{n-1}$, and assume that we are given the data $\ld_T(1), \dots, \ld_T(i+1)$, and, by the induction hypothesis, $\addr_T(1), \dots, \addr_T(i)$ are also known.
By Lemma~\ref{lem:addresses}, we have $\addr_T(i) = u 0 1^p$ and $\addr_T(i+1) = u 1 0^q$ for some word $u$ and $p,q \in \IN$.
Since we know $\addr_T(i)$, we are in possession of the word $u$, and we can also compute the number $q$ as $q = \ld_T(i+1) - \ld_T(i) + 1$. 
In this way we have determined $\addr_T(i+1) = u 1 0^q$.
\end{proof}

The following fact will also be useful.

\begin{lemma}
\label{lem:dr}
For any $T, T' \in \mathcal{T}_n$ with $T \neq T'$,
\begin{enumerate}[label={\upshape (\roman*)}]
\item\label{lem:dr:l}
the least $i \in \nset{n}$ such that $(\ld_T(i) - \ld_{T'}(i), \rd_T(i) - \rd_{T'}(i)) \neq (0,0)$ \textup{(}equivalently, the least $i \in \nset{n}$ such that $\ld_T(i) \neq \ld_{T'}(i)$\textup{)} satisfies $\rd_T(i) = \rd_{T'}(i)$, and
\item\label{lem:dr:r}
the greatest $i \in \nset{n}$ such that $(\ld_T(i) - \ld_{T'}(i), \rd_T(i) - \rd_{T'}(i)) \neq (0,0)$ \textup{(}equivalently, the greatest $i \in \nset{n}$ such that $\rd_T(i) \neq \rd_{T'}(i)$\textup{)} satisfies $\ld_T(i) = \ld_{T'}(i)$.
\end{enumerate}
\end{lemma}

\begin{proof}
We prove only statement \ref{lem:dr:l}; the proof of \ref{lem:dr:r} is analogous.
Assume we have chosen the least $i$ with $(\ld_T(i) - \ld_{T'}(i), \rd_T(i) - \rd_{T'}(i)) \neq (0,0)$.
If $i = 1$, then $\rd_T(1) = \rd_{T'}(1) = 0$, so the conclusion is immediate.
We may therefore assume $i > 1$.
Then the left depths in $T$ and $T'$ coincide up to the $(i-1)$\hyp{}st leaf, and thus, by Lemma~\ref{lem:da}, the addresses also agree up to the $(i-1)$\hyp{}st leaf.
(This also shows that if the left depths coincide up to the $(i-1)$\hyp{}st leaf, then so do also the right depths, and we see that the two formulations of the choice of $i$ are indeed equivalent.)
We can write $\addr_T(i-1) = u 0 1^p$, $\addr_T(i) = u 1 0^q$ and $\addr_{T'}(i-1) = v 0 1^{p'}$, $\addr_{T'}(i) = v 1 0^{q'}$, according to Lemma~\ref{lem:addresses}.
Since $\addr_T(i-1) = \addr_{T'}(i-1)$, we have $u = v$ and $p = p'$.
From $u = v$ it follows that the number of ones in $u 1 0^q$ is the same as the number of ones in $v 1 0^{q'}$; thus $\rd_T(i) = \rd_{T'}(i)$.
\end{proof}

New binary trees can be built from given ones by joining two trees under a new root vertex.
Let $T_1$ and $T_2$ be binary trees.
We denote by $T_1 \wedge T_2$ the binary tree that is obtained by taking the disjoint union of $T_1$ and $T_2$, adding a new vertex $u$ and designating it as the root of $T_1 \wedge T_2$, and setting the root of $T_1$ as the left child of $u$ and the root of $T_2$ as the right child of $u$.

It is well known that binary trees with $n$ leaves are in a one\hyp{}to\hyp{}one correspondence with bracketings of size $n$; hence the number of binary trees with $n$ leaves is $C_{n-1}$.
A canonical bijection between $B_n$ and $\mathcal{T}_n$ is given by the restriction to $B_n$ of the map $\tau$ from the set of all groupoid terms to the set of all binary trees defined recursively as follows.
For any variable $x_i$, let $\tau(x_i)$ be the binary tree with one vertex.
For a term $t = (t_1 \cdot t_2)$, let $\tau(t) := \tau(t_1) \wedge \tau(t_2)$.
We often identify a bracketing $t \in B_n$ with $\tau(t)$
without further mention.

The \emph{opposite} of a bracketing $t\in B_n$, denoted $t^{\mathrm{op}}$, is obtained by reversing the order of its symbols, interchanging opening and closing parentheses, and replacing $x_i$ by $x_{n-i+1}$ for $i\in\nset n$.
Here each variable is treated as a single symbol.
If $T = \tau(t)$ for some bracketing $t \in B_n$, then the tree $\tau(t^\mathrm{op})$ is called the \emph{opposite tree} of $T$ and is denoted by $T^\mathrm{op}$.
The opposite tree of $T$ can be thought of as obtained from $T$ by reflection over a vertical line.

\begin{example}
We list below the bracketings $t_i$ corresponding to the binary trees in Figure~\ref{fig:5-trees} (i.e., $T_i=\tau(t_i)$ for $i=1,\dots,14$).
The two bracketings appearing in each row are opposites of each other.
\medskip
\begin{center}
\begin{tabular}{r@{\,\,}lr@{\,\,}l}
\toprule
$t_1$ = & $x_1(x_2(x_3(x_4x_5)))$ & $t_8$ = & $(((x_1x_2)x_3)x_4)x_5$\\
$t_2$ = & $x_1(x_2((x_3x_4)x_5))$ & $t_9$ = & $((x_1(x_2x_3))x_4)x_5$\\
$t_3$ = & $x_1((x_2x_3)(x_4x_5))$ & $t_{10}$ = & $((x_1x_2)(x_3x_4))x_5$\\
$t_4$ = & $x_1((x_2(x_3x_4))x_5)$ & $t_{11}$ = & $(x_1((x_2x_3)x_4))x_5$\\
$t_5$ = & $x_1(((x_2x_3)x_4)x_5)$ & $t_{12}$ = & $(x_1(x_2(x_3x_4)))x_5$\\
$t_6$ = & $(x_1x_2)(x_3(x_4x_5))$ & $t_{13}$ = & $((x_1x_2)x_3)(x_4x_5)$\\
$t_7$ = & $(x_1x_2)((x_3x_4)x_5)$ & $t_{14}$ = & $(x_1(x_2x_3))(x_4x_5)$\\
\bottomrule
\end{tabular}
\end{center}

\end{example}

\subsection{Modular (left, right) depth sequences}

Let $T$ be a binary tree with $n$ leaves, and assume its leaves are $1, 2, \dots, n$ in the left\hyp{}to\hyp{}right order.
The \emph{depth sequence} of $T$ is the tuple $d_T := (d_T(1), d_T(2), \dots, d_T(n))$.
Similarly, the \emph{left depth sequence} of $T$ is the tuple $\ld_T := (\ld_T(1), \ld_T(2), \dots, \ld_T(n))$,
and the \emph{right depth sequence} of $T$ is $\rd_T := (\rd_T(1), \rd_T(2), \dots, \rd_T(n))$.
A binary tree is uniquely determined by its depth sequence,
and it is also uniquely determined by its left (or right) depth sequence (see Cs\'{a}k\'{a}ny, Waldhauser \cite[Statements 2.7, 2.8]{CsaWal-2000}, also cf.\ Lemma~\ref{lem:da}, Proposition~\ref{prop:random walk}, and Remark~\ref{remark:random walk}).

We may also consider (left, right) depth sequences modulo some $k \in \IN$.
Let $d^k_T$, $\ld^k_T$, $\rd^k_T$ be the sequences obtained from $d_T$, $\ld_T$, $\rd_T$, respectively, by taking componentwise remainders modulo $k$.
These are called the \emph{\textup{(}left, right\textup{)} depth sequences} of $T$ \emph{modulo $k$,} or \emph{modular \textup{(}left, right\textup{)} depth sequences} of $T$.
(Note that congruence modulo $0$ means equality and the (left, right) depth sequences modulo $0$ are just the ordinary (left, right) depth sequences.)
As the following example demonstrates, binary trees are not uniquely determined by their modular (left, right) depth sequences.

\begin{figure}
\hfill
\begin{tikzpicture}[puup]
  \node {}
    child {node(1) {} }
    child {node {}
      child {node {}}
      child {node {}
        child {node {}}
        child[piste, level distance=1cm, sibling distance=1cm] {node {}
          child[viiva, level distance=0.5cm, sibling distance=0.5cm] {node {}}
          child[viiva, level distance=0.5cm, sibling distance=0.5cm] {node {}
            child {node(k) {}}
            child {node {}
              child {node {}
                child {node {}
                  child[piste, level distance=1cm, sibling distance=1cm] {node {}
                    child[viiva, level distance=0.5cm, sibling distance=0.5cm] {node {}
                      child {node {}}
                      child {node(2k) {}}
                    }
                    child[viiva, level distance=0.5cm, sibling distance=0.5cm] {node {}}
                  }
                  child {node {}}
                }
                child {node {}}
              }
              child {node(k1) {}}
            }
          }
        }
      }
    }
  ;
\draw[decoration={brace,mirror,raise=5pt},decorate] (1.west) -- node[xshift=-12pt,yshift=-6pt,draw=none,fill=none] {$k$} (k.south);
\draw[decoration={brace,raise=5pt},decorate] (k1.east) -- node[xshift=12pt,yshift=-6pt,draw=none,fill=none] {$k$} (2k.south);
\end{tikzpicture}
\hfill
\begin{tikzpicture}[puup]
  \node {}
    child {node {}
      child {node {}
        child[piste, level distance=1cm, sibling distance=1cm] {node {}
          child[viiva, level distance=0.5cm, sibling distance=0.5cm] {node {}
            child {node {}
              child {node(k1) {}}
              child {node {}
                child {node {}}
                child {node {}
                  child {node {}}
                  child[piste, level distance=1cm, sibling distance=1cm] {node {}
                    child[viiva, level distance=0.5cm, sibling distance=0.5cm] {node {}}
                    child[viiva, level distance=0.5cm, sibling distance=0.5cm] {node {}
                      child {node(2k) {}}
                      child {node {}}
                    }
                  }
                }
              }
            }
            child {node(k) {}}
          }
          child[viiva, level distance=0.5cm, sibling distance=0.5cm] {node {}}
        }
        child {node {}}
      }
      child {node {}}
    }
    child {node(1) {} }
  ;
\draw[decoration={brace,raise=5pt},decorate] (1.east) -- node[xshift=12pt,yshift=-6pt,draw=none,fill=none] {$k$} (k.south);
\draw[decoration={brace,mirror,raise=5pt},decorate] (k1.west) -- node[xshift=-12pt,yshift=-6pt,draw=none,fill=none] {$k$} (2k.south);
\end{tikzpicture}
\hfill\mbox{}

\caption{Two binary trees with the same depth sequence modulo $k$.}
\label{fig:depth-mod-k}
\end{figure}

\begin{example}
For any $k\ge2$, the two binary trees with $2k+1$ leaves shown in Figure~\ref{fig:depth-mod-k} have the same depth sequence modulo $k$, namely $(1, 2, \dots, k-1, 0, 0, 0, k-1, \dots, 2, 1)$.
Similarly, the two binary trees with $k + 2$ leaves shown in Figure~\ref{fig:left-depth-mod-k} have the same left depth sequence modulo $k$, namely $(1, 0, k-1, k-2, \dots, 1, 0)$,
and their opposite trees have the same right depth sequence modulo $k$.
\end{example}

\begin{figure}
\hfill
\begin{tikzpicture}[puup]
  \node {}
    child {node {}}
    child {node {}
      child {node {}
        child {node {}
          child[piste, level distance=1cm, sibling distance=1cm] {node {}
            child[viiva, level distance=0.5cm, sibling distance=0.5cm] {node {}
              child {node {}}
              child {node(k) {}}
            }
            child[viiva, level distance=0.5cm, sibling distance=0.5cm] {node {}}
          }
          child {node {}}
        }
        child {node {}}
      }
      child {node(1) {} }
    }
  ;
\draw[decoration={brace,raise=5pt},decorate] (1.east) -- node[xshift=12pt,yshift=-6pt,draw=none,fill=none] {$k$} (k.south);
\end{tikzpicture}
\hfill
\begin{tikzpicture}[puup]
  \node {}
    child {node {}
      child {node {}
        child {node {}
          child[piste, level distance=1cm, sibling distance=1cm] {node {}
            child[viiva, level distance=0.5cm, sibling distance=0.5cm] {node {}
              child {node {}}
              child {node(k) {}}
            }
            child[viiva, level distance=0.5cm, sibling distance=0.5cm] {node {}}
          }
          child {node {}}
        }
        child {node {}}
      }
      child {node(1) {} }
    }
    child {node {}}
  ;
\draw[decoration={brace,raise=5pt},decorate] (1.east) -- node[xshift=12pt,yshift=-6pt,draw=none,fill=none] {$k$} (k.south);
\end{tikzpicture}
\hfill\mbox{}

\caption{Two binary trees with the same left depth sequence modulo $k$.}
\label{fig:left-depth-mod-k}
\end{figure}

It is easy to decide whether a given $n$\hyp{}tuple of natural numbers is the left or right depth sequence (modulo $k$) of some binary tree.
Namely, it is known that the right depth sequences of binary trees are precisely the so\hyp{}called zag sequences (see Cs\'ak\'any, Waldhauser \cite[Statement~2.8]{CsaWal-2000}) and the left depth sequences are the reversed zag sequences.
Such sequences are easy to recognise, also when reduced modulo $k$ (see Lehtonen, Waldhauser \cite[Proposition~5.4]{LehWal-2021}).
Total depth sequences are also easily recognisable, because, as mentioned above, a binary tree is uniquely determined by, and can be efficiently reconstructed from, its depth sequence.
As for total depth sequences modulo $k$, a rather simple necessary and sufficient condition in the case when $k = 2$ was provided by Huang, Mickey, and Xu \cite[Lemma~6]{HuaMicXu-2017}, but we are not aware of any similar results for moduli greater than $2$.

\subsection{Equivalence relations on binary trees based on modular (left, right) depth sequences}
\label{subsec:eqrel}

In this subsection
we are going to define several equivalence relations on the set $\mathcal{T}_n$ of binary trees with $n$ leaves ($n \in \IN_+$).
Using the one\hyp{}to\hyp{}one correspondence between binary trees with $n$ leaves and bracketings of $n$ variables, we may equivalently view these as equivalence relations on bracketings: if $\sim$ is any one of the equivalence relations defined on binary trees and $t, t' \in B_n$, we let $t \sim t'$ if and only if $\tau(t) \sim \tau(t')$.

\begin{definition}
\label{def:equivalences}
For $k, \ell \in \IN$ and $T, T' \in \mathcal T_n$, we let
\begin{itemize}
\item $T \eqD{k} T'$ if and only if $d^k_T = d^k_{T'}$, that is, $d_T(i) \equiv d_{T'}(i) \pmod{k}$ for all $i \in \nset{n}$ (\emph{$k$\hyp{}depth\hyp{}equivalence});
\item $T \eqL{k} T'$ if and only if $\ld^k_T = \ld^k_{T'}$, that is, $\ld_T(i) \equiv \ld_{T'}(i) \pmod{k}$ for all $i \in \nset{n}$ (\emph{$k$\hyp{}left\hyp{}depth\hyp{}equivalence});
\item $T \eqR{k} T'$ if and only if $\rd^k_T = \rd^k_{T'}$, that is, $\rd_T(i) \equiv \rd_{T'}(i) \pmod{k}$ for all $i \in \nset{n}$ (\emph{$k$\hyp{}right\hyp{}depth\hyp{}equivalence});
\item $T \eqLR{k}{\ell} T'$ if and only if $T \eqL{k} T'$ and $T \eqR{\ell} T'$ (\emph{$(k,\ell)$\hyp{}depth\hyp{}equivalence}).
\end{itemize}
We introduce the following notation for the number of equivalence classes of $\mathcal{T}_n$ with respect to the above equivalence relations:
\[
\NeqD{k}{n} := \card{\mathcal{T}_n / {\eqD{k}}},
\quad
\NeqL{k}{n} := \card{\mathcal{T}_n / {\eqL{k}}},
\quad
\NeqR{k}{n} := \card{\mathcal{T}_n / {\eqR{k}}},
\quad
\NeqLR{k}{\ell}{n} := \card{\mathcal{T}_n / {\eqLR{k}{\ell}}}.
\]
\end{definition}
\noindent It is clear that $T \eqL{k} T'$ if and only if $T^\mathrm{op} \eqR{k} T'{}^\mathrm{op}$, and consequently $\NeqL{k}{n} = \NeqR{k}{n}$.

\begin{example}
\label{ex:5-trees-2}
Let us consider the binary trees with five leaves, as shown in Figure~\ref{fig:5-trees}.
It is easy to verify that the only nontrivial $\eqL{3}$\hyp{}equivalence class is $\{T_5, T_8\}$, the only nontrivial $\eqR{3}$\hyp{}equivalence class is $\{T_1, T_{12}\}$, and the only nontrivial $\eqLR{2}{2}$\hyp{}equivalence class is $\{T_2, T_9\}$.
Moreover, the only nontrivial $\eqD{2}$\hyp{}equivalence classes are $\{T_2, T_9\}$, $\{T_3, T_{10}\}$, $\{T_4, T_{13}\}$, and $\{T_6, T_{11}\}$.
\end{example}

Regarding the number of equivalence classes of the $k$\hyp{}depth\hyp{}equivalence relation $\eqD{k}$, only a few particular cases are well understood.
The $0$\hyp{}depth\hyp{}equivalence relation $\eqD{0}$ is just the equality relation, so the numbers $\NeqD{0}{n}$ coincide with the Catalan numbers: $\NeqD{0}{n} = C_{n-1}$ for all $n \geq 1$.
The $1$\hyp{}depth\hyp{}equivalence relation $\eqD{1}$ is entirely trivial; all binary trees with $n$ leaves are $1$\hyp{}depth\hyp{}equivalent, so $\NeqD{1}{n} = 1$ for all $n \geq 1$.
The $2$\hyp{}depth\hyp{}equivalence was investigated by Huang, Mickey and Xu \cite{HuaMicXu-2017}, and the numbers $\NeqD{2}{n}$ were shown to be given by the sequence \href{http://oeis.org/A000975}{A000975} in \textit{The On\hyp{}Line Encyclopedia of Integer Sequences} (OEIS) \cite{oeis}, which is known to have several characterisations, for example,
for $n \geq 2$,
\[
\NeqD{2}{n}
= \left\lfloor \frac{2^n}{3} \right\rfloor
= \frac{2^{n+1} - 3 - (-1)^{n+1}}{6}
= \begin{cases}
\frac{2^n - 1}{3}, & \text{if $n$ is even,} \\
\frac{2^n - 2}{3}, & \text{if $n$ is odd.}
\end{cases}
\]
We are not aware of any results concerning moduli greater than $2$.
We have computed the values of $\NeqD{k}{n}$ for small $n$ and $k$ with the help of the GAP computer algebra system \cite{GAP} and present them in Table~\ref{tab:values}.
Apart from the first two rows, these sequences do not seem to match any entry in the OEIS as of November 2026.

In contrast, the number $\NeqL{k}{n}$ of $\eqL{k}$\hyp{}equivalence classes of $\mathcal{T}_n$ is well understood for any $k, n \in \IN_+$;
these numbers are given by the so\hyp{}called $k$\hyp{}modular Catalan numbers $C_{k,n}$ defined by Hein and Huang \cite{HeiHua-2017}:
$\NeqL{k}{n} = C_{k,n-1}$.
Closed formulas for modular Catalan numbers are known \cite[Theorem~1.1]{HeiHua-2017}; for $k, n \geq 1$,
\[
C_{k,n}
=
\sum_{\substack{\lambda \subseteq (k-1)^n \\ \card{\lambda} < n}} \frac{n - \card{\lambda}}{n} m_\lambda(1^n)
=
\sum_{0 \leq j \leq (n-1)/k} \frac{(-1)^j}{n} \binom{n}{j} \binom{2n - jk}{n + 1}.
\]
(For explanation of the symbols used in the first summation formula, please refer to \cite{HeiHua-2017}.)
In particular, $C_{2,n} = 2^{n-1}$ for $n \geq 1$; equivalently, $\NeqL{2}{n} = 2^{n-2}$ for $n \geq 2$.
It also holds that $C_{1,n} \leq C_{2,n} \leq C_{3,n} \leq \cdots$.
The numbers $\NeqL{k}{n} = C_{k,n-1}$ for $k, n \leq 15$ are evaluated in Table~\ref{table:L-eq}.

As for the number $\NeqLR{k}{\ell}{n}$ of $\eqLR{k}{\ell}$\hyp{}equivalence classes of $\mathcal{T}_n$,
Hein and Huang~\cite{HeiHua-2018}, \cite[Section~1.3, last paragraph]{HeiHua-2022} conjectured, based on computational evidence, that
$\NeqLR{k}{\ell}{n} = \NeqL{k + \ell - 1}{n}$,
for all $k, \ell, n \geq 1$.
We have verified this with the help of a computer for $k, \ell, n \leq 14$.
Note that, even though the above equality might hold, the equivalence relations $\eqLR{k}{\ell}$ and $\eqL{k+\ell-1}$ themselves do not coincide, as illustrated by Example~\ref{ex:5-trees-2}.
Thus equality of the counting sequences in this conjecture would not identify the corresponding fine spectra.

\section{Equivalence of binary trees modulo a group}

Next we define a class of equivalence relations on binary trees that are determined by groups. 
As we shall see in Section~\ref{sect:grids}, these relations are of key importance in describing associative spectra of linear quasigroups.

\begin{definition}
\label{def:group-product-word}
Let $\alg{G} = (G,{\cdot})$ be a group with neutral element $1$.
For a family $(\gamma_i)_{i \in I}$ of elements of $G$ and a word $w \in I^*$, define the group element $\gamma_w$ by the following recursion: $\gamma_\varepsilon := 1$, and if $w = i w'$ for some $i \in I$ and $w' \in I^*$, then $\gamma_w := \gamma_i \cdot \gamma_{w'}$.
In other words, $\gamma_w$ is the image of $w$ under the unique monoid homomorphism from the free monoid $I^*$ to $\alg{G}$ that assigns $\gamma_i$ to $i$ for each $i \in I$.

Let $a, b \in G$, and let $T$ and $T'$ be binary trees with $n$ leaves.
Let $\gamma_0 := a$ and $\gamma_1 := b$.
We say that $T$ and $T'$ are \emph{$(a,b)$\hyp{}equivalent modulo $\alg{G}$,} and we write $T \eqabG{a}{b}{\alg{G}} T'$, if for all $i \in \nset{n}$, $\gamma_{\addr_T(i)} = \gamma_{\addr_{T'}(i)}$.
We denote by $\NeqabG{a}{b}{\alg{G}}{n}$ the number of $\eqabG{a}{b}{\alg{G}}$\hyp{}equivalence classes of binary trees with $n$ leaves.
\end{definition}

\begin{remark}
We can assume without loss of generality that $\alg{G}$ is generated by $a$ and $b$, as the elements outside the subgroup generated by $a$ and $b$ play no role in Definition~\ref{def:group-product-word}. 
Let us then consider the Cayley graph $\Gamma(G,\{a,b\})$ of $G$ corresponding to the generating set $\{a,b\}$.
Let us label the left and right edges of a binary tree $T$ by $a$ and $b$, respectively. Then there is a unique graph homomorphism $\psi_T\colon T \to \Gamma(G,\{a,b\})$ that preserves edge labels and maps the root of $T$ to the neutral element of $G$. 
The element $\gamma_{\addr_T(i)}$ defined above is just $\psi_T(i)$ for each leaf $i$ of $T$.
Thus $T \eqabG{a}{b}{\alg{G}} T'$ holds for given $n$-leaf binary trees $T$ and $T'$ if and only if $\psi_T(i) = \psi_{T'}(i)$ for all $i \in \nset{n}$.
\end{remark}

\begin{remark}
Of course some of the relations $\eqabG{a}{b}{\alg{G}}$ might coincide for different groups $\alg{G}$ and elements $a,b \in G$.
In fact, we will prove in Proposition~\ref{prop:group vs lin} that each of these relations coincides with a relation of the form $\eqabG{a}{b}{\alg{G}}$, where $\alg{G}$ is a two-generated abelian group (see also Remark~\ref{remark:quotient group}). 
\end{remark}

\begin{example}
\label{ex:group equivalence}
The various equivalence relations on binary trees that we have seen in the previous section are special instances of $(a,b)$\hyp{}equivalence modulo some group $\alg{G}$.

\begin{enumerate}[label={(\roman*)}]
\item
Let $\alg{G} = (\ZZ_k, {+})$ for $k \in \IN$, and consider $\eqabG{a}{b}{\alg{G}}$.
With $a = 1$, $b = 1$, we get $k$\hyp{}depth\hyp{}equivalence $\eqD{k}$;
with $a = 1$, $b = 0$, we get $k$\hyp{}left\hyp{}depth\hyp{}equivalence $\eqL{k}$;
and
with $a = 0$, $b = 1$, we get $k$\hyp{}right\hyp{}depth\hyp{}equivalence $\eqR{k}$.
(Here $1$ denotes the residue class of the integer $1$, whereas the neutral element of the additive group $\ZZ_k$ is $0$.)
\item
For $k, \ell \in \IN$, taking $\alg{G} = (\ZZ_k, {+}) \times (\ZZ_\ell, {+})$, $a = (1,0)$, $b = (0,1)$, we get $(k,\ell)$\hyp{}depth\hyp{}equivalence $\eqLR{k}{\ell}$.
\end{enumerate}
\end{example}

Generalising the first item of Example~\ref{ex:group equivalence}, let us consider the group $\alg{G} = (\ZZ_k, {+})$ with arbitrary $a, b \in \ZZ_k$. 
We denote the corresponding equivalence relation by $\eqabm{a}{b}{k}$, and we give an explicit definition of this relation below.

\begin{definition}
For $a, b, m \in \ZZ$, define the equivalence relation $\eqabm{a}{b}{m}$ on 
$\mathcal{T}_n$ by the following rule.
Assume that $T$ and $T'$ are binary trees in $\mathcal{T}_n$ and their leaves are $1, 2, \dots, n$ in the left\hyp{}to\hyp{}right order.
Set
\[
T \eqabm{a}{b}{m} T'
\mathrel{{:}{\Longleftrightarrow}}
\forall i \in \nset{n} \colon a \ld_T(i) + b \rd_T(i) \equiv a \ld_{T'}(i) + b \rd_{T'}(i) \pmod{m}.
\]
Let $\Neqabm{a}{b}{m}{n} := \card{\mathcal{T}_n / {\eqabm{a}{b}{m}}}$.
\end{definition}

We have enumerated the numbers $\Neqabm{a}{b}{m}{n}$ for small values of the parameters $a$, $b$, $m$, $n$ and present them in Appendix~\ref{app:Tlin} (see Table~\ref{tab:Nabmn}).
It remains an open problem to determine these numbers for arbitrary $a$, $b$, $m$, $n$.
The following two lemmata show some relationships between numbers of this form, as well as to the other variants of Catalan numbers of Definition~\ref{def:equivalences}.

\begin{lemma}
\label{lem:abm}
Let $a, b \in \ZZ$, $m \in \IN_{+}$.
\begin{enumerate}[label={\upshape (\roman*)}]
\item\label{lem:abm:LandR}
${\eqabm{1}{0}{m}} = {\eqL{m}}$, ${\eqabm{0}{1}{m}} = {\eqR{m}}$, and ${\eqabm{1}{1}{m}} = {\eqD{m}}$;
consequently, $\Neqabm{1}{0}{m}{n} = \NeqL{m}{n}$, $\Neqabm{0}{1}{m}{n} = \NeqR{m}{n}$, and $\Neqabm{1}{1}{m}{n} = \NeqD{m}{n}$.
\item\label{lem:abm:bam}
$\Neqabm{a}{b}{m}{n} = \Neqabm{b}{a}{m}{n}$.
\item\label{lem:abm:scale}
For any $\ell \in \ZZ \setminus \{0\}$,
we have ${\eqabm{a}{b}{m}} = {\eqabm{\ell a}{\ell b}{\ell m}}$;
consequently, $\Neqabm{a}{b}{m}{n} = \Neqabm{\ell a}{\ell b}{\ell m}{n}$.
\item\label{lem:abm:unit}
If $\ell$ is a unit modulo $m$, then ${\eqabm{a}{b}{m}} = {\eqabm{\ell a}{\ell b}{m}}$;
consequently, $\Neqabm{a}{b}{m}{n} = \Neqabm{\ell a}{\ell b}{m}{n}$.
\item\label{lem:abm:LR}
If $a,b >0$, $\gcd(a,b) = 1$ and $m = ab$, then ${\eqabm{a}{b}{m}} = {\eqLR{b}{a}}$;
consequently, $\Neqabm{a}{b}{m}{n} = \NeqLR{b}{a}{n} = \NeqLR{a}{b}{n}$.
\end{enumerate}
\end{lemma}

\begin{proof}
\ref{lem:abm:LandR}
This is Example~\ref{ex:group equivalence}.

\ref{lem:abm:bam}
Since for all $T, T' \in \mathcal{T}_n$,
\begin{align*}
T \eqabm{a}{b}{m} T'
& \iff
\forall i \in \nset{n} \colon a \ld_T(i) + b \rd_T(i) \equiv a \ld_{T'}(i) + b \rd_{T'}(i) \pmod{m}
\\ & \iff
\forall i \in \nset{n} \colon a \rd_{T^\mathrm{op}}(i) + b \ld_{T^\mathrm{op}}(i) \equiv a \rd_{T'{}^\mathrm{op}}(i) + b \ld_{T'{}^\mathrm{op}}(i) \pmod{m}
\\ & \iff
T^\mathrm{op} \eqabm{b}{a}{m} T'{}^\mathrm{op},
\end{align*}
we conclude that the map $T \mapsto T^\mathrm{op}$ induces a bijection between $\mathcal{T}_n / {\eqabm{a}{b}{m}}$ and $\mathcal{T}_n / {\eqabm{b}{a}{m}}$ for each $n \in \IN_{+}$.

\ref{lem:abm:scale}
Clear because the congruences $a x + b y \equiv 0 \pmod{m}$ and $\ell a x + \ell b y \equiv 0 \pmod{\ell m}$ are equivalent.

\ref{lem:abm:unit}
Since $\ell$ is a unit modulo $m$, the congruences $a x + b y \equiv 0 \pmod{m}$ and $\ell a x + \ell b y \equiv 0 \pmod{m}$ are equivalent.

\ref{lem:abm:LR}
Since $\gcd(a,b) = 1$, it follows from the Chinese remainder theorem that $ax + by \equiv 0 \pmod{ab}$ is equivalent to $ax + by \equiv 0 \pmod{a}$ and $ax + by \equiv 0 \pmod{b}$, which in turn is equivalent to $ax \equiv 0 \pmod{b}$ and $by \equiv 0 \pmod{a}$.
From $\gcd(a,b) = 1$, it follows that $a$ is a unit modulo $b$ and $b$ is a unit modulo $a$; hence the last pair of congruences is equivalent to $x \equiv 0 \pmod{b}$ and $y \equiv 0 \pmod{a}$.
We conclude that $T \eqabm{a}{b}{m} T'$ if and only if $T \eqLR{b}{a} T'$, as claimed.
\end{proof}

\begin{lemma}
\label{lem:cirmi}
For all $a, b \in \ZZ$ and $m, n \in \IN_+$ with $\gcd(a,m) = 1$, we have $\Neqabm{a}{b}{m}{n} \geq \NeqR{m}{n}$.
\end{lemma}

\begin{proof}
Since $\gcd(a,m) = 1$, $a$ has a multiplicative inverse $a^{-1}$ modulo $m$.
By Lemma~\ref{lem:abm}\ref{lem:abm:unit}, $\Neqabm{1}{a^{-1} b}{m}{n} = \Neqabm{a}{b}{m}{n}$, so we may assume that $a = 1$.
It is known that $\NeqR{k}{n} = \NeqL{k}{n} = C_{k,n-1}$ and that the $k$-modular Catalan number $C_{k,n}$ counts the number of binary trees with $n + 1$ leaves that avoid (i.e., do not contain as a subgraph) the binary tree $\comb^1_k$
(see Hein, Huang \cite[Proposition~3.3]{HeiHua-2017}).
Here, $\comb^1_k$ is the tree shown on the left in Figure~\ref{fig:left-depth-mod-k}.
The subgraph embedding is required to preserve the left/right order of children.
In particular, $\comb_m^1$ occurs whenever a vertex address contains a right step followed by $m$ consecutive left steps.
(Note that we do not require that $\comb^1_k$ is not contained as a \emph{subtree}; see Subsection~\ref{subs:binary-trees}.)
Therefore, in order to prove the inequality $\Neqabm{1}{b}{m}{n} \geq \NeqR{m}{n}$, it suffices to show that if $T$ and $T'$ are distinct binary trees with $n$ leaves that avoid $\comb^1_m$, then $T \not\eqabm{1}{b}{m} T'$.
To this end, we are going to find a leaf $\ell$ satisfying $(\ld_T(\ell) - \ld_{T'}(\ell), \rd_T(\ell) - \rd_{T'}(\ell)) = (d,0)$ with $d \not\equiv 0 \pmod{m}$; then $1 \cdot d + b \cdot 0 = d \not\equiv 0 \pmod{m}$, from which it follows that $T \not\eqabm{1}{b}{m} T'$.

Let $T, T' \in \mathcal{T}_n$ with $T \neq T'$, and assume that $T$ and $T'$ avoid $\comb^1_m$.
Since $T \neq T'$, there exists, by Lemma~\ref{lem:da}, a leaf $i$ such that $\ld_T(i) \neq \ld_{T'}(i)$; let $\ell$ be the leftmost such leaf.
By Lemma~\ref{lem:dr}, $(\ld_T(\ell) - \ld_{T'}(\ell), \rd_T(\ell) - \rd_{T'}(\ell)) = (d,0)$ for some nonzero $d$.
Consider first the case that $\ell > 1$.
By Lemmata~\ref{lem:addresses} and \ref{lem:da}, $\addr_T(\ell-1) = \addr_{T'}(\ell-1) = u 0 1^p$, $\addr_T(\ell) = u 1 0^q$, $\addr_{T'}(\ell) = u 1 0^{q'}$, where $u$ is the address of the deepest common ancestor of leaves $\ell - 1$ and $\ell$ in $T$, or, equivalently, in $T'$.
By the choice of $\ell$, we have $q \neq q'$.
Since $T$ and $T'$ avoid $\comb^1_m$, we have $0 \leq q \leq m-1$ and $0 \leq q' \leq m-1$.
Therefore, $0 < \lvert \ld_T(\ell) - \ld_{T'}(\ell) \rvert < m$, so $d \not\equiv 0 \pmod{m}$.
Thus, $\ell$ is the leaf we were looking for.

Assume now that $\ell = 1$.
If $\ld_T(1) \not\equiv \ld_{T'}(1) \pmod{m}$, then $1$ is the leaf we were looking for, as $\rd_T(1) = \rd_{T'}(1) = 0$.
Assume now that $\ld_T(1) \equiv \ld_{T'}(1) \pmod{m}$; without loss of generality, assume that $\ld_T(1) > \ld_{T'}(1)$, so $\ld_T(1) = \ld_{T'}(1) + tm$ for some $t \in \IN_{+}$.
Let $r$ be the vertex in $T$ at depth $tm$ along the path from the root to leaf $1$, and let $T^*$ be the subtree of $T$ rooted at $r$.
Suppose $T^*$ has $s$ leaves; they are the $s$ leftmost leaves of $T$.
We have $s<n$, since $T^*$ is a proper leftmost subtree.
Then for each $i \in \nset{s}$, we have $\ld_{T^*}(i) = \ld_T(i) - tm$ and $\rd_{T^*}(i) = \rd_T(i)$.
The left depth sequence of $T^*$ cannot coincide with the first $s$ terms of the left depth sequence of $T'$.
(Because only the rightmost leaf of a binary tree has left depth $0$, we must have $\ld_{T^*}(s) = 0 < \ld_{T'}(s)$.)
Let $\ell'$ be the leftmost leaf $i$ with $\ld_{T^*}(i) \neq \ld_{T'}(i)$.
The first left depths of $T^*$ and $T'$ agree, so $\ell' > 1$.
The common\hyp{}prefix argument of Lemmata~\ref{lem:addresses} and~\ref{lem:da} applies to their initial leaves even though the trees have different numbers of leaves.
A similar argument as in the previous paragraph shows that $0 < \lvert \ld_{T^*}(\ell') - \ld_{T'}(\ell') \rvert < m$ and $\rd_{T^*}(\ell') - \rd_{T'}(\ell') = 0$.
Therefore $\ld_{T}(\ell') - \ld_{T'}(\ell') \not\equiv 0 \pmod{m}$ and $\rd_{T}(\ell') - \rd_{T'}(\ell') = 0$, so $\ell'$ is the leaf we were looking for.
\end{proof}

We saw in Example~\ref{ex:group equivalence} several groups that give rise to nontrivial equivalence relations on binary trees.
We conclude this section by a proposition relating random walks on binary trees to equivalence modulo the multiplicative group of nonzero real numbers that illustrates that the equivalence relations corresponding to groups can often be trivial.

\begin{proposition}
\label{prop:random walk}
Let $\RR^\ast$ denote the multiplicative group of nonzero real numbers, and let $0<p<1$.
For arbitrary binary trees $T, T' \in \mathcal{T}_n$, we have $T \eqabG{p}{1-p}{\RR^\ast} T'$ if and only if $T=T'$.
\end{proposition}

\begin{proof}
Let us perform a random walk on a binary tree $T \in \mathcal{T}_n$, starting at the root, and moving to the left or to the right child of the current vertex with probabilities $p$ and $1-p$, respectively. 
Of course, the walk ends when we reach a leaf.
The probability of ending at leaf $i$ is $p^{\ld_T(i)} \cdot (1-p)^{\rd_T(i)}$, which is nothing else but $\gamma_{\addr_T(i)}$ from Definition~\ref{def:group-product-word}, when applied to the case $\alg{G}=\RR^\ast$, $a=p$, $b=1-p$.
Therefore, it suffices to prove that the structure of the tree $T$ can be recovered from the ``leaf probabilities" $\gamma_{\addr_T(i)} = p^{\ld_T(i)} \cdot (1-p)^{\rd_T(i)}\ (i=1,\dots,n)$. 

We prove this by induction on the number of leaves. 
In the cases $n=1$ and $n=2$ we have nothing to prove, as there is only one binary tree with one, respectively two leaves.
Now let $n \geq 3$, and assume that trees with fewer than $n$ leaves are uniquely determined by their leaf probabilities.
Let $T \in \mathcal{T}_n$ have left and right subtrees $T_0$ and $T_1$ of sizes $k$ and $n-k$, respectively, i.e., $T = T_0 \wedge T_1$ with $T_0 \in \mathcal{T}_k$ and $T_1 \in \mathcal{T}_{n-k}$. 
The random walk starting at the root of $T$ ends at one of the leaves of $T_0$ if and only if the first step was taken to the left child of the root, and this event has probability $p$. 
Therefore, $k$ is the unique integer such that $\sum_{i=1}^k \gamma_{\addr_T(i)} = p$.
This way we have computed $k$ from the leaf probabilities of $T$, and then we can calculate the leaf probabilities of $T_0$ and $T_1$ as follows:
\begin{align*}
\gamma_{\addr_{T_0}(i)} &= \gamma_{\addr_T(i)} / p &&(i=1,\dots,k); \\ 
\gamma_{\addr_{T_1}(i)} &= \gamma_{\addr_T(k+i)} / (1-p) &&(i=1,\dots,n-k). 
\end{align*}
Now, applying the induction hypothesis to $T_0$ and $T_1$, we see that $T = T_0 \wedge T_1$ is indeed uniquely determined by the leaf probabilities $\gamma_{\addr_T(i)}$.
\end{proof}

\begin{remark}
\label{remark:random walk}
Since the random walk used in the proof above will eventually end at some leaf, we have $\sum_{i=1}^n p^{\ld_T(i)} \cdot (1-p)^{\rd_T(i)} = 1$.
In the case $p=1/2$, this becomes $\sum_{i=1}^n 2^{-d_T(i)} = 1$, and Proposition~\ref{prop:random walk} shows that binary trees are determined by their depth sequences, both of which are well-known facts (cf. also Kraft's inequality for prefix-free codes).
\end{remark}

\section{Linear quasigroups}

\subsection{Affine quasigroups}

A quasigroup $\alg{A} = (A, {\circ})$ is \emph{affine} over a (possibly nonabelian) group $(A, {+})$ if there exist automorphisms $\varphi_0, \varphi_1 \in \Aut(A, {+})$, and a constant $c \in A$ such that $x \circ y = \varphi_0(x) + c + \varphi_1(y)$.
If $c = 0$ in the above, then $\alg{A}$ is \emph{linear} over $(A, {+})$.
(Note that some authors use the term ``linear quasigroup" for the general case $x \circ y = \varphi_0(x) + c + \varphi_1(y)$ with arbitrary $c$.)
The quintuple $(A, {+}, \varphi_0, \varphi_1, c)$ is called an \emph{arithmetic form} of $\alg{A}$.
It is well known (see, e.g., \cite[Lemmata 2.6 and 2.7]{JedStaVoj-2017}) that
\begin{itemize}
\item an affine quasigroup with arithmetic form $(A, {+}, \varphi_0, \varphi_1, c)$ is idempotent if and only if $c = 0$ and $\varphi_0 + \varphi_1 = \id_A$ (pointwise addition of functions on the left side);
\item every medial quasigroup (i.e., a quasigroup that satisfies the identity \linebreak $(xy)(uv) \approx (xu)(yv)$) is affine; moreover, a quasigroup is medial if and only if it is affine with arithmetic form $(A, {+}, \varphi_0, \varphi_1, c)$, where $(A, {+})$ is an abelian group and $\varphi_0 \varphi_1 = \varphi_1 \varphi_0$ (this was proved independently by Bruck~\cite{Bruck-1944}, Murdoch~\cite{Murdoch-1941} and Toyoda~\cite{Toyoda-1941}).
\end{itemize}

\subsection{Bracketings over linear quasigroups}
\label{subsec:linear}

Let $S_A$ denote the symmetric group on the set $A$, and let $\varphi_i \in S_A\ (i \in I)$ be a family of permutations of $A$.
We use the notational convention $fg(x) = f(g(x))$ for composing maps.
Definition~\ref{def:group-product-word} then gives permutations $\varphi_w$ for each string $w \in I^*$: $\varphi_\varepsilon := \id_A$, and if $w = i w'$ for some $i \in I$ and $w' \in I^*$, then $\varphi_w := \varphi_i \varphi_{w'}$.
In the following we will use this definition for $I=\{0,1\}$ with $\varphi_0$ and $\varphi_1$ being the automorphisms occurring in the arithmetic form of the linear quasigroup under consideration (of course, in this case we can take the group $\Aut(A, {+})$ instead of $S_A$).
This notation allows us to conveniently describe when a given linear quasigroup satisfies a given bracketing identity.

\begin{proposition}
\label{prop:linear-gen}
Let $\alg{A} = (A, {\circ})$ be a linear quasigroup over a group $(A, {+})$ with arithmetic form $(A, {+}, \varphi_0, \varphi_1, 0)$.
Let $t, t' \in B_n$, and let $T, T' \in \mathcal{T}_n$ be the corresponding binary trees.
\begin{enumerate}[label={\upshape{(\roman*)}}]
\item\label{prop:linear-gen:term-operation}
$t^\alg{A}(a_1, \dots, a_n) = \varphi_{\addr_T(1)}(a_1) + \varphi_{\addr_T(2)}(a_2) + \dots + \varphi_{\addr_T(n)}(a_n)$.
\item\label{prop:linear-gen:satisfy}
$\alg{A}$ satisfies $t \approx t'$ if and only if for all $i \in \nset{n}$, $\varphi_{\addr_T(i)} = \varphi_{\addr_{T'}(i)}$.
\end{enumerate}
\end{proposition}

\begin{proof}
\ref{prop:linear-gen:term-operation}
We proceed by induction on $n$.
The claim holds for $n = 1$ because in this case we have $t = x_1$ and $t^\alg{A}(a_1) = \id_A(a_1) = \varphi_\varepsilon(a_1) = \varphi_{\addr_T(1)}(a_1)$.

Assume that the claim holds for $n \leq k$ for some $k \geq 1$, and let $t \in B_{k+1}$.
Then $t = (t_1 \cdot t_2)$ for some subterms $t_1$ and $t_2$, say, with $\var(t_1) = \{x_1, \dots, x_\ell\}$ and $\var(t_2) = \{x_{\ell + 1}, \dots, x_{k+1}\}$.
Put $T_j := \tau(t_j)$ for $j = 1, 2$ and $\mathbf{a} := (a_1, \dots, a_{k+1})$.
The leaves of $T_2$ retain their labels $\ell+1, \dots, k+1$ inherited from $T$.
By the induction hypothesis, we have
\begin{align*}
t_1^\alg{A}(\mathbf{a}) &= \varphi_{\addr_{T_1}(1)}(a_1) + \varphi_{\addr_{T_1}(2)}(a_2) + \dots + \varphi_{\addr_{T_1}(\ell)}(a_\ell), \\
t_2^\alg{A}(\mathbf{a}) &= \varphi_{\addr_{T_2}(\ell+1)}(a_{\ell+1}) + \varphi_{\addr_{T_2}(\ell+2)}(a_{\ell+2}) + \dots + \varphi_{\addr_{T_2}(k+1)}(a_{k+1}).
\end{align*}
Using the fact that $\varphi_0$ and $\varphi_1$ are automorphisms of $(A, {+})$, it follows that
\begin{align*}
t^\alg{A}(\mathbf{a})
&= \varphi_0(t_1^\alg{A}(\mathbf{a})) + \varphi_1(t_2^\alg{A}(\mathbf{a}))
\\ & = \varphi_0 \varphi_{\addr_{T_1}(1)}(a_1) + \dots + \varphi_0 \varphi_{\addr_{T_1}(\ell)}(a_\ell) + {}
\\ & \phantom{{} = {}} \varphi_1 \varphi_{\addr_{T_2}(\ell+1)}(a_{\ell+1}) + \dots + \varphi_1 \varphi_{\addr_{T_2}(k+1)}(a_{k+1})
\\ & = \varphi_{\addr_T(1)}(a_1) + \dots + \varphi_{\addr_T(k+1)}(a_{k+1}).
\end{align*}

\ref{prop:linear-gen:satisfy}
Assume first that $\alg{A}$ satisfies $t \approx t'$.
By applying part~\ref{prop:linear-gen:term-operation}, by assigning the neutral element $0$ of $(A, {+})$ to all variables but $x_i$, and by observing that any automorphism of $(A, {+})$ maps $0$ to itself, we get
\begin{align*}
t^{\alg{A}}(0, \dots, 0, a_i, 0, \dots, 0) &= \varphi_{\addr_T(i)}(a_i), \\
t'^{\alg{A}}(0, \dots, 0, a_i, 0, \dots, 0) &= \varphi_{\addr_{T'}(i)}(a_i),
\end{align*}
which implies $\varphi_{\addr_T(i)} = \varphi_{\addr_{T'}(i)}$ for all $i \in \nset{n}$.

Assume now that $\varphi_{\addr_T(i)} = \varphi_{\addr_{T'}(i)}$ for all $i \in \nset{n}$.
Then we have, by part~\ref{prop:linear-gen:term-operation}, that
\begin{align*}
&
t^\alg{A}(a_1, \dots, a_n)
= \varphi_{\addr_T(1)}(a_1) + \dots + \varphi_{\addr_T(n)}(a_n)
\\ &
= \varphi_{\addr_{T'}(1)}(a_1) + \dots + \varphi_{\addr_{T'}(n)}(a_n)
= t'^{\alg{A}}(a_1, \dots, a_n),
\end{align*}
that is, $t^\alg{A} = t'^{\alg{A}}$, so $\alg{A}$ satisfies the identity $t \approx t'$.
\end{proof}

\begin{remark}
\label{remark:linear-gen}
Item \ref{prop:linear-gen:satisfy} of Proposition~\ref{prop:linear-gen} can be reformulated in terms of Definition~\ref{def:group-product-word} as follows: $\alg{A}$ satisfies $t \approx t'$ if and only if $T \eqabG{\varphi_0}{\varphi_1}{\Aut(A, {+})} T'$.
Thus the fine associative spectrum of a linear quasigroup is always of the form $\eqabG{a}{b}{\alg{G}}$ for a suitable group $\alg{G}$ and elements $a,b \in G$.
\end{remark}

\begin{example}
\label{ex:rings}
Let $\alg{R}$ be a unital ring, and let $\alg{R}^\ast$ denote the multiplicative group of units of $\alg{R}$.
For any $a \in \alg{R}^\ast$, the map $x \mapsto ax$ is an automorphism of the additive group of $\alg{R}$, thus $x \circ y = ax+by$ is a linear quasigroup operation over $(R, {+})$ for all $a,b \in \alg{R}^\ast$.
In this case item \ref{prop:linear-gen:satisfy} of Proposition~\ref{prop:linear-gen} reads as follows: $(R, {\circ})$ satisfies $t \approx t'$ if and only if $T \eqabG{a}{b}{\alg{R}^\ast} T'$.
If $\alg{R}$ is commutative, then this can be written as $a^{\ld_T(i)} \cdot b^{\rd_T(i)} = a^{\ld_{T'}(i)} \cdot b^{\rd_{T'}(i)}$ for all $i \in \nset{n}$.

In this context, Proposition~\ref{prop:random walk} means that the operation $px+(1-p)y$ on real numbers is antiassociative.
The special case $p=1/2$ gives that the arithmetic mean as a binary operation is antiassociative, which was already noted in \cite{CsaWal-2000}.
These observations can be generalised as follows: the operation $ax+by$ is antiassociative whenever $a$ and $b$ are nonzero complex numbers and at least one of them is not a root of unity (see Example~\ref{example:antiassociative}).
This follows from Belousov's results on balanced identities \cite{Belousov-1966}, and the case $a=1$ was also considered by Lord \cite{Lord-1987}. 
\end{example}

\begin{remark}
In the proof of Proposition~\ref{prop:linear-gen} we did not use additive inverses; hence it holds even if $(A,{+})$ is just a monoid, and Example~\ref{ex:rings} is valid for semirings.
As an example, let us consider the tropical semiring $(\RR\cup\{\infty\},{\oplus},{\odot})$, where $x \oplus y = \min(x,y)$ and $x \odot y = x+y$.
In this case, our ``linear operation" takes the form $x \circ y = \min(a+x,b+y)$ with $a,b \in \RR$, and the bracketing identity $t \approx t'$ is satisfied if and only if $a\ld_T(i) + b\rd_T(i) = a\ld_{T'}(i) + b\rd_{T'}(i)$ for all $i \in \nset{n}$.
We will prove in Lemma~\ref{lem:treealisable-1} that if at least one of $a$ and $b$ is not zero, then this holds only if $T=T'$, i.e., the operation $\min(a+x,b+y)$ is antiassociative.
\end{remark}

\subsection{Special cases of linear quasigroups}

We apply Proposition~\ref{prop:linear-gen} to linear quasigroups with special assumptions on the automorphisms $\varphi_0,\varphi_1$.
The arising fine associative spectra correspond to the previously introduced equivalence relations on binary trees defined by linear congruences for left and right depths.

\begin{proposition}
\label{prop:linear}
Let $\alg{A} = (A, {\circ})$ be a linear quasigroup over a group $(A, {+})$ with arithmetic form $(A, {+}, \varphi_0, \varphi_1, 0)$.
Let $t, t' \in B_n$, and let $T, T' \in \mathcal{T}_n$ be the corresponding binary trees.
\begin{enumerate}[label={\upshape (\roman*)}]
\item\label{prop:linear:phi0=phi1}
If $\varphi_0 = \varphi_1$ and $\varphi_0$ has order $k$, then $\alg{A}$ satisfies $t \approx t'$ if and only if $T \eqD{k} T'$.
Consequently, $s_n(\alg{A}) = \NeqD{k}{n}$.

\item\label{prop:linear:phi1=id}
If $\varphi_1 = \id_A$ and $\varphi_0$ has order $k$, then $\alg{A}$ satisfies $t \approx t'$ if and only if $T \eqL{k} T'$.
Consequently, $s_n(\alg{A}) = \NeqL{k}{n} = C_{k,n-1}$.

\item\label{prop:linear:phi0=id}
If $\varphi_0 = \id_A$ and $\varphi_1$ has order $k$, then $\alg{A}$ satisfies $t \approx t'$ if and only if $T \eqR{k} T'$.
Consequently, $s_n(\alg{A}) = \NeqR{k}{n} = C_{k,n-1}$.

\end{enumerate}
\end{proposition}

\begin{proof}
\ref{prop:linear:phi0=phi1}
Since $\varphi_0 = \varphi_1$, we have $\varphi_{\addr_T(i)} = \varphi_0^{d_T(i)}$.
Since $\varphi_0$ has order $k$, it follows that $\varphi_{\addr_T(i)} = \varphi_{\addr_{T'}(i)}$ if and only if $d_T(i) \equiv d_{T'}(i) \pmod{k}$.
By Proposition~\ref{prop:linear-gen}, $\alg{A}$ satisfies $t \approx t'$ if and only if $T \eqD{k} T'$.
The last claim is clear because $\NeqD{k}{n} = \card{B_n / {\eqD{k}}}$.

\ref{prop:linear:phi1=id}
Since $\varphi_1 = \id_A$, we have $\varphi_{\addr_T(i)} = \varphi_0^{\ld_T(i)}$.
Since $\varphi_0$ has order $k$, it follows that $\varphi_{\addr_T(i)} = \varphi_{\addr_{T'}(i)}$ if and only if $\ld_T(i) \equiv \ld_{T'}(i) \pmod{k}$.
By Proposition~\ref{prop:linear-gen}, $\alg{A}$ satisfies $t \approx t'$ if and only if $T \eqL{k} T'$.
The last claim is clear because $C_{k,n-1} = \NeqL{k}{n} = \card{B_n / {\eqL{k}}}$.

\ref{prop:linear:phi0=id}
The proof is similar to part \ref{prop:linear:phi1=id}.
\end{proof}

\begin{proposition}
\label{prop:phi0-phi1-commute}
Let $\alg{A} = (A, {\circ})$ be a linear quasigroup over a group $(A, {+})$ with arithmetic form $(A, {+}, \varphi_0, \varphi_1, 0)$, and assume that $\varphi_0$ and $\varphi_1$ have orders $k$ and $\ell$, respectively, and $\varphi_0 \varphi_1 = \varphi_1 \varphi_0$.
Let $t, t' \in B_n$, and let $T, T' \in \mathcal{T}_n$ be the corresponding binary trees.
\begin{enumerate}[label={\upshape (\roman*)}]
\item\label{prop:phi0-phi1-commute:ineq}
If $T \eqLR{k}{\ell} T'$, then $\alg{A}$ satisfies $t \approx t'$.
Consequently, $\sigma_n(\alg{A})$ is a coarsening of $\eqLR{k}{\ell}$ and hence $s_n(\alg{A}) \leq \NeqLR{k}{\ell}{n}$.
\item\label{prop:phi0-phi1-commute:pqrs}
If for all $p, q, r, s \in \IN$, $\varphi_0^p \varphi_1^q = \varphi_0^r \varphi_1^s$ implies $p \equiv r \pmod{k}$ and $q \equiv s \pmod{\ell}$, then $\alg{A}$ satisfies $t \approx t'$ if and only if $T \eqLR{k}{\ell}T'$.
Consequently, $s_n(\alg{A}) = \NeqLR{k}{\ell}{n}$.
\end{enumerate}
\end{proposition}

\begin{proof}
\ref{prop:phi0-phi1-commute:ineq}
Assume $T \eqLR{k}{\ell} T'$.
Then $\ld_T(i) \equiv \ld_{T'}(i) \pmod{k}$ and $\rd_T(i) \equiv \rd_{T'}(i) \pmod{\ell}$ for all $i \in \nset{n}$.
Since $\varphi_0 \varphi_1 = \varphi_1 \varphi_0$ and $\varphi_0$ and $\varphi_1$ have orders $k$ and $\ell$, respectively, we have
$\varphi_{\addr_T(i)} = \varphi_0^{\ld_T(i)} \varphi_1^{\rd_T(i)} = \varphi_0^{\ld_{T'}(i)} \varphi_1^{\rd_{T'}(i)} = \varphi_{\addr_{T'}(i)}$ for all $i \in \nset{n}$.
By Proposition~\ref{prop:linear-gen}, $\alg{A}$ satisfies $t \approx t'$.

\ref{prop:phi0-phi1-commute:pqrs}
By part \ref{prop:phi0-phi1-commute:ineq}, it suffices to show that $\alg{A} \satisfies t \approx t'$ implies $T \eqLR{k}{\ell} T'$.
So, assume that $\alg{A} \satisfies t \approx t'$.
Then for all $i \in \nset{n}$, $\varphi_{\addr_T(i)} = \varphi_{\addr_{T'}(i)}$, i.e., $\varphi_0^{\ld_T(i)} \varphi_1^{\rd_T(i)} = \varphi_0^{\ld_{T'}(i)} \varphi_1^{\rd_{T'}(i)}$ because $\varphi_0 \varphi_1 = \varphi_1 \varphi_0$.
By our hypothesis, this implies that for all $i \in \nset{n}$, $\ld_T(i) \equiv \ld_{T'}(i) \pmod{k}$ and $\rd_T(i) \equiv \rd_{T'}(i) \pmod{\ell}$,
in other words, $T \eqLR{k}{\ell} T'$.
\end{proof}

\begin{proposition}
\label{prop:phi0-phi1-powers-of-pi}
Let $\alg{A} = (A, {\circ})$ be a linear quasigroup over a group $(A, {+})$ with arithmetic form $(A, {+}, \varphi_0, \varphi_1, 0)$, and assume that $\varphi_0 = \pi^a$ and $\varphi_1 = \pi^b$ for some
permutation $\pi$ of $A$
and
$a, b \in \IN$. Assume that $\pi$ has order $m$.
Let $t, t' \in B_n$, and let $T, T' \in \mathcal{T}_n$ be the corresponding binary trees.
Then $\mathbf{A} \satisfies t \approx t'$ if and only if $T \eqabm{a}{b}{m} T'$.
Consequently, $s_n(\alg{A}) = \Neqabm{a}{b}{m}{n}$.
\end{proposition}

\begin{proof}
We have $\varphi_{\addr_T(i)} = (\pi^a)^{\ld_T(i)} (\pi^b)^{\rd_T(i)} = \pi^{a \ld_T(i) + b \rd_T(i)}$ and, similarly, \linebreak $\varphi_{\addr_{T'}(i)} = \pi^{a \ld_{T'}(i) + b \rd_{T'}(i)}$.
Since $\pi$ has order $m$, it follows that $\varphi_{\addr_T(i)} = \varphi_{\addr_{T'}(i)}$ if and only if $a \ld_T(i) + b \rd_T(i) \equiv a \ld_{T'}(i) + b \rd_{T'}(i) \pmod{m}$.
The claim then follows from Proposition~\ref{prop:linear-gen}.
\end{proof}

\begin{remark}
The condition in the proof of Proposition~\ref{prop:phi0-phi1-powers-of-pi} is equivalent to the condition that
$(x,y) = (\ld_T(i) - \ld_{T'}(i), \rd_T(i) - \rd_{T'}(i))$ is a solution of the congruence $a x + b y \equiv 0 \pmod{m}$.
For $m \geq 1$, this congruence has exactly $m \gcd(a,b,m)$ solutions in $\ZZ_m^2$.
A method for determining the solutions was described by Lehmer~\cite[p.~155]{Lehmer-1913}.
\end{remark}

\begin{example}
\label{example:roots of unity}
In order to illustrate the above results,
let $\omega$ and $\zeta$ be primitive $k$-th and $\ell$-th roots of unity in $\CC$, respectively, say $\omega = e^{2 \pi i / k}$ and $\zeta = e^{2 \pi i / \ell}$.
Then $x \mapsto \omega x$ and $x \mapsto \zeta x$ are automorphisms of $(\CC,{+})$ of order $k$ and $\ell$, respectively, and they commute.
Let $t, t' \in B_n$, and let $T$ and $T'$ be the corresponding binary trees.
Using Propositions~\ref{prop:linear}, \ref{prop:phi0-phi1-commute} and \ref{prop:phi0-phi1-powers-of-pi}, we conclude the following:
\begin{enumerate}[label={(\alph*)}]
\item
The groupoid $\alg{A}_1 = (\CC, {\circ})$ with $x \circ y = \omega x + \omega y$ satisfies $t \approx t'$ if and only if $T \eqD{k} T'$, and its associative spectrum is $s_n(\alg{A}_1) = \NeqD{k}{n}$.

\item
The groupoid $\alg{A}_2 = (\CC, {\circ})$ with $x \circ y = \omega x + y$ satisfies $t \approx t'$ if and only if $T \eqL{k} T'$, and its associative spectrum is $s_n(\alg{A}_2) = \NeqL{k}{n} = C_{k,n-1}$.

\item
The groupoid $\alg{A}_3 = (\CC, {\circ})$ with $x \circ y = x + \zeta y$ satisfies $t \approx t'$ if and only if $T \eqR{\ell} T'$, and its associative spectrum is $s_n(\alg{A}_3) = \NeqR{\ell}{n} = C_{\ell,n-1}$.

\item
The groupoid $\alg{A}_4 = (\CC, {\circ})$ with $x \circ y = \omega^a x + \omega^b y$, where $a, b \in \IN$, satisfies $t \approx t'$ if and only if $T \eqabm{a}{b}{k} T'$, and its associative spectrum is $s_n(\alg{A}_4) = \Neqabm{a}{b}{k}{n}$.

\item
The groupoid $\alg{A}_5 := \alg{A}_2 \times \alg{A}_3$ satisfies $t \approx t'$ if and only if $T \eqLR{k}{\ell} T'$, and its associative spectrum is $s_n(\alg{A}_5) = \NeqLR{k}{\ell}{n}$.

\item
The groupoid $\alg{A}_6 := \alg{A}_3 \times \alg{A}_4$ satisfies $t \approx t'$ if and only if $T \sim T'$, where $\mathord{\sim} = \mathord{\eqR{\ell} \cap \eqabm{a}{b}{k}}$.
\end{enumerate}
\end{example}

\section{Grids, groups and trees}
\label{sect:grids}
We associate a subgroup $\Lambda_{\alg G}(a,b)$ of $\ZZ^2$ with a group $G$ and two designated elements $a$ and $b$, and a subgroup $\Lambda_{T,T'}$ with a pair $T, T'$ of trees. 
We first prove that $T\eqabG{a}{b}{\alg G}T'$ holds exactly when $\Lambda_{T,T'}\subseteq\Lambda_{\alg G}(a,b)$; we then show that every rank-two subgroup is realised by a pair of trees.

\subsection{Parallelogram grids}
\label{subsec:grids}

If $(u,v)$ and $(w,z)$ are two linearly independent vectors in $\RR \times \RR$, then the set of their integral linear combinations $\Lambda := \ZZ(u,v) + \ZZ(w,z) = \{ x(u,v) + y(w,z) : x,y \in \ZZ\}$ is a subgroup of $\RR \times \RR$, which can be regarded geometrically as a parallelogram grid.
Let us note that such subgroups are commonly called lattices, but the word ``lattice'' is also used to refer to partially ordered sets with infima and suprema.
To avoid ambiguity, we use the term ``lattice'' only in this latter context (such as the lattice of subgroups of a given group).

We will work with subgroups of $\ZZ\times\ZZ$, so let us give a complete overview of these subgroups here.
We have of course the trivial ``zero-dimensional" subgroup $\{(0,0)\}$, and ``one-dimensional" subgroups $\ZZ(u,v) = \{ x(u,v) : x \in \ZZ\}$ spanned by nonzero vectors $(u,v) \in \ZZ\times\ZZ \setminus \{(0,0)\}$.
The remaining subgroups are parallelogram grids in the sense explained above. 
It is not hard to verify that if $\Lambda$ is such a ``two-dimensional" subgroup, then we have $\Lambda = \ZZ(u,v) + \ZZ(w,0)$ for suitable integers $u,v,w$.
We can assume that $0 \leq u < w$ and $v>0$, and under this assumption the numbers $u,v,w$ are uniquely determined by the parallelogram grid $\Lambda$. 
It will be useful to state explicitly the condition for a pair of integers to belong to $\Lambda$:
\begin{equation}
\label{eq:grid condition}
(r,s) \in \ZZ(u,v) + \ZZ(w,0) \iff v \divides s \text{ and } vw \divides vr-us.
\end{equation}

We will denote the set of all subgroups of $\ZZ\times\ZZ$ by $\Sub(\ZZ\times\ZZ)$.
This is a lattice with least element $\{(0,0)\}$, greatest element $\ZZ\times\ZZ$, and the meet and join operations are given by $\Lambda_1 \cap \Lambda_2$ and $\Lambda_1 + \Lambda_2$, respectively.
Let $\Sub_2(\ZZ\times\ZZ)$ denote the set of two-dimensional subgroups of $\ZZ\times\ZZ$.

\begin{remark}
\label{remark:Sub2}
Clearly, $\Sub_2(\ZZ\times\ZZ)$ is closed under joins, and we will verify that it is also closed under meets.
If $\Lambda = \ZZ(u,v) + \ZZ(w,0) \in \Sub_2(\ZZ\times\ZZ)$, then $w(u,v)-u(w,0)=(0,vw) \in \Lambda$, i.e., $\Lambda$ contains a nonzero point on the $y$ axis.
(Actually, $(0,vw/\gcd(u,w))$ is the point of $\Lambda$ on the $y$ axis that lies immediately above the origin.)
Thus we see that $\Lambda$ contains a rectangular grid $\ZZ(k,0) + \ZZ(0,\ell)$, say, with $k=w$ and $\ell=vw$.
Now, if $\Lambda'$ is another two-dimensional subgroup, then $\Lambda' \supseteq \ZZ(k',0) + \ZZ(0,\ell')$ for suitable nonzero (in fact, positive) integers $k'$ and $\ell'$.
Therefore, $\Lambda \cap \Lambda' \supseteq 
\ZZ(\lcm(k,k'),0) + \ZZ(0,\lcm(\ell,\ell')) \in \Sub_2(\ZZ\times\ZZ)$, hence $\Lambda \cap \Lambda'$ is indeed a two-dimensional subgroup.
Thus we see that $\Sub_2(\ZZ\times\ZZ)$ is a sublattice of $\Sub(\ZZ\times\ZZ)$.
Let us note that $\Sub(\ZZ\times\ZZ) \setminus \Sub_2(\ZZ\times\ZZ)$ is not a sublattice: it is obviously closed under meets, but not closed under joins.
\end{remark}

\begin{example}
\label{ex:par-grid-1}
The parallelogram grid $\Lambda = \ZZ(u,v) + \ZZ(w,0)$ with $u = 6$, $v = 3$, $w = 10$ is shown in Figure~\ref{fig:grid}.
By \eqref{eq:grid condition}, the points in $\Lambda$ are the pairs $(r,s)$ satisfying $3 \divides s$ and $30 \divides 3r - 6s$, or, equivalently, $s \equiv 0~\pmod{3}$ and $r \equiv 2s~\pmod{10}$.
The $y$ coordinate of the point on the $y$ axis immediately above the origin is $vw / \gcd(u,w) = 3 \cdot 10 / 2 = 15$.
The grid $\Lambda$ contains the rectangular grid $\ZZ(10,0) + \ZZ(0,15)$.
\end{example}

\begin{figure}
  \centering
  \begin{tikzpicture}
    \coordinate (Origin)   at (0,0);
    \coordinate (XAxisMin) at (-2,0);
    \coordinate (XAxisMax) at (8,0);
    \coordinate (YAxisMin) at (0,-1.25);
    \coordinate (YAxisMax) at (0,7.5);
    \draw [thin, gray,-latex] (XAxisMin) -- (XAxisMax);
    \draw [thin, gray,-latex] (YAxisMin) -- (YAxisMax);

    \clip (-2.5,-1.5) rectangle (8.5cm,8.5cm); 
    \pgftransformcm{10/3}{0}{6/3}{3/3}{\pgfpoint{0cm}{0cm}}
    \coordinate (Bone) at (1,0);
    \coordinate (Btwo) at (0,1);
    \draw[style=help lines,dashed] (-14,-14) grid[step=1cm] (14,14);
    \foreach \x in {-10,...,10}{
      \foreach \y in {-10,...,10}{
        \node[draw,circle,inner sep=2pt,fill] at (\x,\y) {};
      }
    }
    \draw [ultra thick,-latex,red] (Origin)
        -- (Bone) node [below right] {$(10,0)$};
    \draw [ultra thick,-latex,red] (Origin)
        -- (Btwo) node [above left] {$(6,3)$};
    \node[draw=none,fill=none,above left] at (-3,5) {$(0,15)$};
    \node[draw=none,fill=none,above left] at (-2,5) {$(10,15)$};
    \node[draw=none,fill=none,above] at (-1,2) {$(2,6)$};
  \end{tikzpicture}
  \caption{Parallelogram grid $\ZZ(6,3) + \ZZ(10,0)$.}
  \label{fig:grid}
\end{figure}

\subsection{Groups and grids}

In this subsection we prove that each ``modulo group" equivalence relation (see Definition~\ref{def:group-product-word}) can be described by a parallelogram grid.
(More precisely, we only prove here that they correspond to subgroups of $\ZZ\times\ZZ$, but later we will see that nontrivial ``modulo group" equivalences give rise to two-dimensional subgroups, i.e., parallelogram grids.)

\begin{definition}\label{def:Lambda_G}
For a group $\alg{G}=(G,{\cdot})$ and $a,b \in G$, let $\pairs{\alg{G}}{a}{b}$ denote the following set of pairs of integers:
\[
\pairs{\alg{G}}{a}{b} = \{ (r,s) \in \ZZ\times\ZZ : b a^r b^{-1} = a b^{-s} a^{-1} \}.
\]
\end{definition}

\begin{lemma}
For any group $\alg{G}$ and $a,b \in G$, the set $\pairs{\alg{G}}{a}{b}$ is a subgroup of $\ZZ\times\ZZ$. 
\end{lemma}

\begin{proof}
The defining condition $b a^r b^{-1} = a b^{-s} a^{-1}$ of $\pairs{\alg{G}}{a}{b}$ is equivalent to $\phi(a^r)=b^{-s}$, where $\phi(x) = a^{-1}bxb^{-1}a$ is the conjugation by $b^{-1}a$.
If $(r,s), (r',s') \in \pairs{\alg{G}}{a}{b}$, then, using the fact that $\phi$ is an automorphism of $\alg{G}$, we have
\[
\phi(a^{r+r'})=\phi(a^ra^{r'})=\phi(a^r)\phi(a^{r'})=b^{-s}b^{-s'}=b^{-(s+s')},
\]
thus $(r+r',s+s') \in \pairs{\alg{G}}{a}{b}$.
Similarly, $(r,s) \in \pairs{\alg{G}}{a}{b}$ implies $(-r,-s) \in \pairs{\alg{G}}{a}{b}$:
\[
\phi(a^{-r})=\phi((a^r)^{-1})=\phi(a^r)^{-1}=(b^{-s})^{-1}=b^{-(-s)}.
\qedhere
\]
\end{proof}

The following proposition gives an explicit description of the grid $\pairs{\alg{G}}{a_0}{a_1}$ in the case when $a_0$ and $a_1$ both have finite order.

\begin{proposition}\label{prop:phi-hat}
Let $\alg{G}$ be a group and let $a,b \in G$ be elements of finite orders $m_0 = \ord(a)$ and $m_1 = \ord(b)$, respectively.
Let $g$ be an arbitrary generator of the cyclic subgroup $\gen{bab^{-1}} \cap \gen{aba^{-1}}$; then we have $g = (bab^{-1})^{e_0} = (aba^{-1})^{e_1}$ for suitable integers $e_0$ and $e_1$.
The grid $\pairs{\alg{G}}{a}{b}$ can be described in terms of these parameters as
\begin{equation}\label{eq:phi-hat-grid}
	\pairs{\alg{G}}{a}{b} = \ZZ(m_0,0) + \ZZ(0,m_1) + \ZZ(e_0,-e_1). 
\end{equation}
Consequently, $(r,s) \in \pairs{\alg{G}}{a}{b}$ if and only if
\begin{equation}\label{eq:phi-hat-condition}
	\gcd(e_0,m_0) \divides r, \;\; \gcd(e_1,m_1) \divides s \;\; \text{ and } \;\; m \divides re_1 + se_0, 
\end{equation}
where $m = \ord(g) \cdot \gcd(e_0,m_0) \cdot \gcd(e_1,m_1)$.
\end{proposition}
\begin{proof}
Let us introduce some notation to facilitate the proof: let $d_0 = \gcd(e_0,m_0)$, $d_1 = \gcd(e_1,m_1)$ and $\widehat{a} = bab^{-1}$, $\widehat{b} = aba^{-1}$.
Clearly, we have $\ord(\widehat{a}) = m_0$, $\ord(\widehat{b}) = m_1$ and $\ord(g) = {m_0}/d_0 = m_1/d_1$.
The latter equalities imply that $m = m_0d_1 = m_1d_0$.

By Definition~\ref{def:Lambda_G}, $(r,s) \in \pairs{\alg{G}}{a}{b}$ means that $\widehat{a}^r = \widehat{b}^{-s}$.
This element belongs to the subgroup $\gen{\widehat{a}} \cap \gen{\widehat{b}} = \gen{g}$, thus we can express it as $g^x$ for some integer $x$.
Taking $g = \widehat{a}^{e_0}$ into account, the equality $\widehat{a}^r = g^x$ is equivalent to $xe_0 \equiv r~\pmod{m_0}$.
Similarly, $\widehat{b}^{-s} = g^x$ is equivalent to $xe_1 \equiv -s~\pmod{m_1}$.
Now we can prove \eqref{eq:phi-hat-grid} by simply rewriting the two congruences:
\begin{align*}
(r,s) \in \pairs{\alg{G}}{a}{b} &\iff \exists x \in \ZZ\colon xe_0 \equiv r \pmod{m_0} \text{ and } xe_1 \equiv -s \pmod{m_1}\\
&\iff \exists x,y,z \in \ZZ\colon r = xe_0 + ym_0 \text{ and } s = -xe_1 + zm_1\\
&\iff \exists x,y,z \in \ZZ\colon (r,s) = x \cdot (e_0,-e_1) + y \cdot (m_0,0) + z \cdot (0,m_1).
\end{align*}

From the above calculations we see that if $(r,s) \in \pairs{\alg{G}}{a}{b}$, then \eqref{eq:phi-hat-condition} holds.
Indeed, $d_0=\gcd(e_0,m_0)$ divides $r = xe_0 + ym_0$ and $d_1=\gcd(e_1,m_1)$ divides $s = -xe_1 + zm_1$; furthermore, $re_1 + se_0 = ym_0e_1 + zm_1e_0$ is divisible by $m = m_0d_1 = m_1d_0$.

It remains to prove that \eqref{eq:phi-hat-condition} implies $(r,s) \in \pairs{\alg{G}}{a}{b}$, and for this we only need to verify that the congruences $xe_0 \equiv r \pmod{m_0}$ and $xe_1 \equiv -s \pmod{m_1}$ have a common solution.
The congruence $xe_0 \equiv r \pmod{m_0}$ has a solution, since $d_0 \divides r$, according to \eqref{eq:phi-hat-condition}.
Dividing by $d_0$, we obtain the congruence $xe_0/d_0 \equiv r/d_0 \pmod{m_0/d_0}$.
Here the modulus is equal to $m_1/d_1$, which is relatively prime to $e_1/d_1$.
Therefore, we may multiply both sides by $e_1/d_1$, and then we get that $xe_0e_1/(d_0d_1) \equiv re_1/(d_0d_1) \pmod{m_0/d_0}$.
Now we multiply by $d_0d_1$ on both sides: $xe_0e_1 \equiv re_1\pmod{m_0d_1}$.
Taking into account that $m=m_0d_1$, we can conclude that $xe_0 \equiv r \pmod{m_0}$ is equivalent to $xe_0e_1 \equiv re_1\pmod{m}$.
Similarly, the congruence $xe_1 \equiv -s \pmod{m_1}$ can be reformulated as $xe_0e_1 \equiv -se_0\pmod{m}$.
From \eqref{eq:phi-hat-condition} we know that $m \divides re_1 + se_0$, which means that $re_1 \equiv -se_0 \pmod{m}$.
Thus the two congruences are actually equivalent, so it is clear that they have a common solution. 
\end{proof}

\begin{remark}
\label{rem:m}
For the grid $\pairs{\alg{G}}{a}{b}$ in the proposition above, the standard generators $(u,v)$ and $(w,0)$ in \eqref{eq:grid condition}  can be given as $(u,v)=(u,d_1)$ and $(w,0)=(m_0,0)$, where $u$ is the least nonnegative multiple of $d_0$ such that $ue_1 + d_1e_0$ is divisible by $m$.
Therefore, the area of the fundamental parallelogram spanned by $(u,v)$ and $(w,0)$ is $m_0d_1=m$; consequently, the quotient group $(\ZZ\times\ZZ)/\pairs{\alg{G}}{a}{b}$ is of size $m$.	
\end{remark}

\begin{corollary}\label{cor:phi-hat}
If $\alg{G}$ is a group and $a,b \in G$ are both of finite order, then $\pairs{\alg{G}}{a}{b}$ is a two-dimensional subgroup of $\ZZ\times\ZZ$.
\end{corollary}
\begin{proof}
	Proposition~\ref{prop:phi-hat} implies that $\pairs{\alg{G}}{a}{b}$ contains the rectangular grid $\ZZ(m_0,0)\allowbreak + \ZZ(0,m_1)$.
\end{proof}

Next we show how $\pairs{\alg{G}}{a}{b}$ determines the equivalence relation $\eqabG{a}{b}{\alg{G}}$.

\begin{proposition}
\label{prop:group vs lin}
Let $\alg{G}$ be a group, let $a,b \in G$, and let $T,T'$ be binary trees with leaves $1, 2, \dots, n$ \textup{(}in the left\hyp{}to\hyp{}right order\textup{)}.
Then $T\eqabG{a}{b}{\alg{G}}T'$ holds if and only if
$(\ld_T(i)-\ld_{T'}(i),\rd_T(i)-\rd_{T'}(i)) \in \pairs{\alg{G}}{a}{b}$ for all $i \in \nset{n}$.
\end{proposition}

\begin{proof}
First let us make some preliminary observations that we will use in the proof.
To simplify notation, we let $\ld_i = \ld_T(i)$, $\rd_i = \rd_T(i)$, $\ld'_i = \ld_{T'}(i)$ and $\rd'_i = \rd_{T'}(i)$ for $i \in \nset{n}$.
For a leaf $i \in \nset{n-1}$, let $z$ and $z'$ be the deepest common ancestors of $i$ and $i+1$ in $T$ and $T'$, respectively.
Setting $u = \addr_T(z)$ and $v = \addr_{T'}(z')$, we have
\begin{equation}
\label{eq:addresses}
\addr_T(i) = u 0 1^p, \; \addr_T(i+1) = u 1 0^q, \; \addr_{T'}(i) = v 0 1^{p'}, \; \addr_{T'}(i+1) = v 1 0^{q'}
\end{equation}
for some $p, q, p', q' \in \IN$ (see Lemma~\ref{lem:addresses}).
This implies the following relationships among the depths:
\begin{align*}
\ld_i &= \ld_T(z) +1, & \ld'_i &= \ld_{T'}(z') +1, &\ld_{i+1} &= \ld_T(z) +q, &\ld'_{i+1} &= \ld_{T'}(z') +q',\\ 
\rd_i &= \rd_T(z) +p, & \rd'_i &= \rd_{T'}(z') +p', &\rd_{i+1} &= \rd_T(z) +1, & \rd'_{i+1} &= \rd_{T'}(z') +1. 
\end{align*}
We can thus conclude that
\begin{equation}
\label{eq:depths}
(\ld_i-\ld'_i) - (\ld_{i+1} - \ld'_{i+1}) = q'-q
\text{ and }
(\rd_i-\rd'_i) - (\rd_{i+1} - \rd'_{i+1}) = p-p'.
\end{equation}

Let $\gamma_0 = a$ and $\gamma_1 = b$, and let us simply write $\Gamma_i = \gamma_{\addr_T(i)}$, $\Gamma'_i = \gamma_{\addr_{T'}(i)}$ (see Definition~\ref{def:group-product-word}).
We can compute these elements of $\alg{G}$ with the help of \eqref{eq:addresses}:
\begin{equation*}
\Gamma_i = \gamma_{u} a b^p, \; \Gamma_{i+1} = \gamma_{u} b a^q, \; \Gamma'_i = \gamma_{v} a b^{p'}, \; \Gamma'_{i+1} = \gamma_{v} b a^{q'}.
\end{equation*}
Therefore, we have
\begin{equation}
\label{eq:gammas}
\Gamma'_i \Gamma_i^{-1} = \gamma_{v} a b^{p'-p} a^{-1} \gamma_{u}^{-1}
\text{ and }
\Gamma'_{i+1} \Gamma_{i+1}^{-1} = \gamma_{v} b a^{q'-q} b^{-1} \gamma_{u}^{-1}.
\end{equation}

Now we are ready to begin the proof.
Assume first that $T \eqabG{a}{b}{\alg{G}} T'$, i.e., $\Gamma_i = \Gamma'_i$ for all $i \in \nset{n}$.
We are going to prove by induction on $i$ that $(\ld_i-\ld'_i,\rd_i-\rd'_i) \in \pairs{\alg{G}}{a}{b}$.
The base case $i=1$ is straightforward: $(\ld_1-\ld'_1,\rd_1-\rd'_1) = (\ld_1-\ld'_1,0) \in \pairs{\alg{G}}{a}{b}$ is equivalent to $a^{\ld_1-\ld'_1}=1$, and this is certainly true, as $a^{\ld_1} = \Gamma_1 = \Gamma'_1 = a^{\ld'_1}$. 
Now, for the induction step, let us assume that $(\ld_i-\ld'_i,\rd_i-\rd'_i) \in \pairs{\alg{G}}{a}{b}$ holds for some $i \in \nset{n-1}$.
Since $\pairs{\alg{G}}{a}{b}$ is a group, in order to prove that $(\ld_{i+1}-\ld'_{i+1},\rd_{i+1}-\rd'_{i+1}) \in \pairs{\alg{G}}{a}{b}$, it suffices to verify that $(q'-q,p-p') \in \pairs{\alg{G}}{a}{b}$, according to \eqref{eq:depths}.
We know that $\Gamma_i = \Gamma'_i$ and $\Gamma_{i+1} = \Gamma'_{i+1}$, and this implies that $\Gamma'_i \Gamma_i^{-1} = \Gamma'_{i+1}\Gamma_{i+1}^{-1} = 1$.
From \eqref{eq:gammas} it follows then that $a b^{p'-p} a^{-1} = b a^{q'-q} b^{-1}$, and this shows that we indeed have $(q'-q,p-p') \in \pairs{\alg{G}}{a}{b}$.

For the converse, assume that $(\ld_i-\ld'_i,\rd_i-\rd'_i) \in \pairs{\alg{G}}{a}{b}$ for all $i \in \nset{n}$.
We are going to prove by induction on $i$ that $\Gamma_i = \Gamma'_i$.
The base case $\Gamma_1 = \Gamma'_1$ is equivalent to $a^{\ld_1} = a^{\ld'_1}$, and this follows immediately, as $(\ld_1-\ld'_1,\rd_1-\rd'_1) = (\ld_1-\ld'_1,0) \in \pairs{\alg{G}}{a}{b}$.
For the induction step, let us assume that $\Gamma_i = \Gamma'_i$ for some $i \in \nset{n-1}$.
We assumed that $(\ld_i-\ld'_i,\rd_i-\rd'_i)$ and $(\ld_{i+1}-\ld'_{i+1},\rd_{i+1}-\rd'_{i+1})$ belong to the group $\pairs{\alg{G}}{a}{b}$, hence $(q'-q,p-p') \in \pairs{\alg{G}}{a}{b}$ by \eqref{eq:depths}.
This fact together with \eqref{eq:gammas} and the induction hypothesis $\Gamma_i = \Gamma'_i$ allow us to deduce that $\Gamma_{i+1} = \Gamma'_{i+1}$:
\[
1 = \Gamma'_i \Gamma_i^{-1} = \gamma_{v} a b^{p'-p} a^{-1} \gamma_{u}^{-1} = \gamma_{v} b a^{q'-q} b^{-1} \gamma_{u}^{-1} = \Gamma'_{i+1} \Gamma_{i+1}^{-1}.
\qedhere
\]
\end{proof}

\begin{remark}
\label{remark:quotient group}
Let us consider the quotient group $\alg{Q}:=(\ZZ\times\ZZ)/\pairs{\alg{G}}{a}{b}$ and the natural homomorphism $\nu\colon \ZZ\times\ZZ \to \alg{Q},\ (r,s) \mapsto (r,s)+\pairs{\alg{G}}{a}{b}$. 
For $\gamma_0=\nu(1,0)$ and $\gamma_1=\nu(0,1)$, Definition~\ref{def:group-product-word} gives $\gamma_{\addr_T(i)} = \nu(\ld_T(i),\rd_T(i))$; thus Proposition~\ref{prop:group vs lin} implies that $T\eqabG{a}{b}{\alg{G}}T'$ is equivalent to $T\eqabG{\gamma_0}{\gamma_1}{\alg{Q}}T'$.
Observe that $\alg{Q}$ is an abelian group generated by at most two elements, hence it is a direct product of (at most) two cyclic groups.
This shows that each relation $\eqabG{a}{b}{\alg{G}}$ coincides with an equivalence relation modulo a group of a very restricted structure.
This somewhat surprising fact allows us to completely describe the fine associative spectra of linear quasigroups (see Theorem~\ref{thm:spectra}).

Propositions \ref{prop:linear}, \ref{prop:phi0-phi1-commute} and \ref{prop:phi0-phi1-powers-of-pi} all follow as special cases of Proposition~\ref{prop:group vs lin}.
It might be instructive to write out explicitly the corresponding groups $\Lambda:=\pairs{\alg{G}}{a}{b}$ and $\alg{Q}:=(\ZZ\times\ZZ)/\Lambda$.
\begin{itemize}
    \item In item \ref{prop:linear:phi0=phi1} of Proposition~\ref{prop:linear}, we have $\Lambda = \{ (r,s) : k \divides r+s \} = \ZZ(k,0) + \ZZ(k-1,1)$ and $\alg{Q} \cong \ZZ_k$. We can assume (up to the choice of the isomorphism $\alg{Q} \cong \ZZ_k$) that $\gamma_0 = \gamma_1 = 1$. 
    (Cf.\ Example~\ref{ex:group equivalence}.)
    \item In item \ref{prop:linear:phi1=id} of Proposition~\ref{prop:linear}, we have $\Lambda = \{ (r,s) : k \divides r \} = \ZZ(k,0) + \ZZ(0,1)$ and $\alg{Q} \cong \ZZ_k$. We can assume (up to the choice of the isomorphism $\alg{Q} \cong \ZZ_k$) that $\gamma_0 = 1$ and  $\gamma_1 = 0$.
    \item In item \ref{prop:linear:phi0=id} of Proposition~\ref{prop:linear}, we have $\Lambda = \{ (r,s) : k \divides s \} = \ZZ(1,0) + \ZZ(0,k)$ and $\alg{Q} \cong \ZZ_k$. We can assume (up to the choice of the isomorphism $\alg{Q} \cong \ZZ_k$) that $\gamma_0 = 0$ and  $\gamma_1 = 1$.
    \item In item \ref{prop:phi0-phi1-commute:pqrs} of Proposition~\ref{prop:phi0-phi1-commute}, we have $\Lambda = \{ (r,s) : k \divides r \text{ and } \ell \divides s \} = \ZZ(k,0) + \ZZ(0,\ell)$ and $\alg{Q} \cong \ZZ_k\times\ZZ_\ell$ (which can be visualised as a $k\times\ell$ grid drawn on the surface of a torus). We can assume (up to the choice of the isomorphism $\alg{Q} \cong \ZZ_k\times\ZZ_\ell$) that $\gamma_0 = (1,0)$ and  $\gamma_1 = (0,1)$.
    \item In Proposition~\ref{prop:phi0-phi1-powers-of-pi}, let $d = \gcd(a,b,m)$.
    Then $\Lambda=\{(r,s) : m \divides ar+bs\}$ and $\alg{Q} \cong \gen{a,b} \leq \ZZ_m$, where $\gen{a,b}$ has order $m/d$.
    Equivalently, $\alg{Q} \cong \ZZ_{m/d}$, and under this isomorphism the distinguished generators are $a/d$ and $b/d$, respectively.
\end{itemize}
\end{remark}

\subsection{Trees and grids}

The following definition is motivated by Proposition~\ref{prop:group vs lin}; it is essentially just a notation that will be convenient when applying that proposition.

\begin{definition}\label{def:Lambda_T}
For a subgroup $\Lambda \subseteq \ZZ\times\ZZ$, let us define the equivalence relation $\sim_\Lambda$ on binary trees as follows: for $T,T' \in \mathcal{T}_n$, 
\[
T \sim_\Lambda T' \iff (\ld_T(i)-\ld_{T'}(i),\rd_T(i)- \rd_{T'}(i)) \in \Lambda \text{ for all } i \in \nset{n}.
\]
Let $T_{\Lambda,n} := \card{\mathcal{T}_n / {\sim_\Lambda}}$.
\end{definition}

We have enumerated the numbers $T_{\Lambda,n}$ with $\Lambda = \ZZ(u,v) + \ZZ(w,0)$ for small values of the parameters $u$, $v$, $w$ and present them in Appendix~\ref{app:Tlin} (see Table~\ref{tab:grid-eq}).

\begin{remark}
\label{rem:Lambda vs R lin}
Note that if $\Lambda = \ZZ(u,v) + \ZZ(w,0)$, then the condition \eqref{eq:grid condition} for membership in $\Lambda$ reveals that ${\sim_{\Lambda}} = {\eqR{v}} \cap {\eqabm{v}{-u}{vw}}$.
\end{remark}

\begin{example}\label{ex:par-grid-2}
For the parallelogram grid $\Lambda = \ZZ(6,3) + \ZZ(10,0)$ of Example~\ref{ex:par-grid-1} (see Figure~\ref{fig:grid}), it holds that
${\sim_\Lambda} = {\eqR{3}} \cap {\eqabm{3}{-6}{30}} = {\eqR{3}} \cap {\eqabm{1}{8}{10}}$.
The second equality holds by Lemma~\ref{lem:abm}\ref{lem:abm:scale}.
\end{example}

In the previous subsection we assigned a grid to every group (with two designated elements); now we assign a grid to any pair of binary trees with the same number of leaves. 
Again, the definition speaks about arbitrary subgroups of $\ZZ\times\ZZ$, but we will prove that if $T \neq T'$ then the corresponding subgroup is two-dimensional.

\begin{definition}
For trees $T,T' \in \mathcal{T}_n$, let $\Lambda_{T,T'}$ denote the subgroup of $\ZZ\times\ZZ$ spanned by $\{ (\ld_T(i)-\ld_{T'}(i),\rd_T(i)-\rd_{T'}(i)) : i \in \nset{n} \}$.
Let us say that a subgroup $\Lambda \subseteq \ZZ\times\ZZ$ is \emph{treealisable}, if there exist binary trees $T,T' \in \mathcal{T}_n$ for some $n \in \IN_{+}$ such that $\Lambda = \Lambda_{T,T'}$.
\end{definition}

\begin{remark}
\label{remark:G vs Lambda}
\leavevmode
\begin{enumerate}[label=(\roman*)]
\item\label{remark:G vs Lambda:1}
Note that $T \sim_\Lambda T'$ is equivalent to $\Lambda_{T,T'} \subseteq \Lambda$ for all $\Lambda \in \Sub(\ZZ\times\ZZ)$ and for all $T,T' \in \mathcal{T}_n$.
\item\label{remark:G vs Lambda:2}
Using the above notations, we can state Proposition~\ref{prop:group vs lin} in a very compact way:
\[
T\eqabG{a}{b}{\alg{G}}T' \iff T \sim_{\pairs{\alg{G}}{a}{b}} T' \iff \Lambda_{T,T'} \subseteq \pairs{\alg{G}}{a}{b}.
\]
(The first equivalence is the statement of Proposition~\ref{prop:group vs lin}, the second one is just a repetition of the observation made in item \ref{remark:G vs Lambda:1}.)
\item\label{remark:G vs Lambda:3}
By Proposition~\ref{prop:linear-gen} and Remark~\ref{remark:linear-gen}, we can verify whether a linear quasigroup satisfies a bracketing identity just by comparing the grids corresponding to the identity and to the quasigroup:
\[
\alg{A} \satisfies t \approx t' \iff \Lambda_{T,T'} \subseteq  \pairs{\Aut(A, {+})}{\varphi_0}{\varphi_1},
\]
where $(A, {+}, \varphi_0, \varphi_1, 0)$ is an arithmetic form of $\alg{A}$ and $T, T'$ are the binary trees corresponding to $t, t'$.
\end{enumerate}
\end{remark}

The goal of this section is to characterise the treealisable subgroups of $\ZZ \times \ZZ$.
As we will prove in the next two lemmata, all two\hyp{}dimensional subgroups of $\ZZ\times\ZZ$ are treealisable (Lemma~\ref{lem:treealisable-2}), but no one\hyp{}dimensional subgroup is treealisable (Lemma~\ref{lem:treealisable-1}).

\begin{lemma}
\label{lem:treealisable-2}
For every two\hyp{}dimensional parallelogram grid $\Lambda$, there exist binary trees $T$ and $T'$ such that $\Lambda_{T,T'} = \Lambda$.
\end{lemma}

\begin{proof}
Let $p$, $q$, $r$, and $s$ be positive integers, and let $T$ and $T'$ be the binary trees with $p + q + r + s + 1$ leaves shown in Figure~\ref{fig:grid-trees}.
We introduce some notation to describe these trees more precisely.
For a tree $S$, define two extensions by
\[
L_0(S) = R_0(S) = S, \quad
L_{j+1}(S) = L_j(S) \wedge \bullet, \quad
R_{j+1}(S) = \bullet \wedge R_j(S),
\]
where $\bullet$ is the one\hyp{}vertex tree.
The trees in Figure~\ref{fig:grid-trees} are
\[
\begin{aligned}
T  &= R_p \bigl( L_{r-1}(R_q(\bullet)) \bigr) \wedge L_s(\bullet), \\
T' &= R_p(\bullet) \wedge L_s \bigl( R_{q-1}(L_r(\bullet)) \bigr).
\end{aligned}
\]
The left and right depths of each leaf in $T$ and $T'$ and the differences thereof are presented in Table~\ref{table:grid-trees}.
Every leaf difference belongs to the subgroup generated by $(r,0)$, $(-s,p)$, and $(0,-q)$, and these three vectors occur at the leaves $x, y, z$, respectively.
Hence
\[
\Lambda_{T,T'} = \ZZ(r,0) + \ZZ(-s,p) + \ZZ(0,-q)
= \ZZ(r,0) + \ZZ(r-s,p) + \ZZ(0,-q).
\]
Let $\Lambda$ be the parallelogram grid generated by $\{(u,v), (w,0)\}$, with $0 \leq u < w$ and $v > 0$, and let us construct the trees $T$ and $T'$ as above with $p := v$, $q := vw$, $r := w$, $s := w - u$.
Then we obtain the parallelogram grid generated by $\{(u,v), (w,0), (0,-vw)\}$, which is clearly the same as $\Lambda$; thus we have $\Lambda = \Lambda_{T,T'}$.
\end{proof}

\begin{figure}
\hfill
\begin{tikzpicture}[puup,level 1/.style={sibling distance=3.5cm},level 2/.style={sibling distance=0.5cm}]
  \node(root) {}
    child {node {}
      child {node(btop) {}}
      child {node {}
        child {node {}}
        child[piste, level distance=1cm, sibling distance=1cm] {node {}
          child[viiva, level distance=0.5cm, sibling distance=0.5cm] {node(bbot) {}}
          child[viiva, level distance=0.5cm, sibling distance=0.5cm] {node {}
            child {node {}
              child[piste, level distance=1cm, sibling distance=1cm] {node {}
                child[viiva, level distance=0.5cm, sibling distance=0.5cm] {node {}
                  child {node(dtop) {}}
                  child {node {}
                    child {node {}}
                    child[piste, level distance=1cm, sibling distance=1cm] {node {}
                      child[viiva, level distance=0.5cm, sibling distance=0.5cm] {node(dbot) {}}
                      child[viiva, level distance=0.5cm, sibling distance=0.5cm] {node(yy) {}}
                    }
                  }
                }
                child[viiva, level distance=0.5cm, sibling distance=0.5cm] {node(cbot) {}}
              }
              child {node {}}
            }
            child {node(ctop) {}}
          }
        }
      }
    }
    child {node {}
      child {node {}
        child[piste, level distance=1cm, sibling distance=1cm] {node {}
          child[viiva, level distance=0.5cm, sibling distance=0.5cm] {node(zz) {}}
          child[viiva, level distance=0.5cm, sibling distance=0.5cm] {node(abot) {}}
        }
        child {node {}}
      }
      child {node(atop) {}}
    }
  ;
\draw[decoration={brace,raise=5pt},decorate] (atop.east) -- node[xshift=12pt,yshift=-6pt,draw=none,fill=none] {$s$} (abot.south);
\draw[decoration={brace,mirror,raise=5pt},decorate] (btop.west) -- node[xshift=-12pt,yshift=-6pt,draw=none,fill=none] {$p$} (bbot.south);
\draw[decoration={brace,raise=5pt},decorate] (ctop.east) -- node[xshift=12pt,yshift=-6pt,draw=none,fill=none] {$\,\,\,\,\,\,\,\,\,r - 1$} (cbot.south);
\draw[decoration={brace,mirror,raise=5pt},decorate] (dtop.west) -- node[xshift=-12pt,yshift=-6pt,draw=none,fill=none] {$q$} (dbot.south);
\draw (dtop.west) node[xshift=-5pt,yshift=5pt,draw=none,fill=none] {$x$};
\draw (yy.south) node[yshift=-6pt,draw=none,fill=none] {$y$};
\draw (zz.south) node[yshift=-6pt,draw=none,fill=none] {$z$};
\draw (root.north) node[yshift=20pt,draw=none,fill=none] {$T$};
\end{tikzpicture}
\hfill
\begin{tikzpicture}[puup,level 1/.style={sibling distance=3.5cm},level 2/.style={sibling distance=0.5cm}]
  \node(root) {}
    child {node {}
      child {node(btop) {}}
      child {node {}
        child {node {}}
        child[piste, level distance=1cm, sibling distance=1cm] {node {}
          child[viiva, level distance=0.5cm, sibling distance=0.5cm] {node(bbot) {}}
          child[viiva, level distance=0.5cm, sibling distance=0.5cm] {node(xx) {}
          }
        }
      }
    }
    child {node {}
      child {node {}
        child[piste, level distance=1cm, sibling distance=1cm] {node {}
          child[viiva, level distance=0.5cm, sibling distance=0.5cm] {node {}
            child {node(dtop) {}}
            child {node {}
              child {node {}}
              child[piste, level distance=1cm, sibling distance=1cm] {node {}
                child[viiva, level distance=0.5cm, sibling distance=0.5cm] {node(dbot) {}}
                child[viiva, level distance=0.5cm, sibling distance=0.5cm] {node {}
                  child {node {}
                    child[piste, level distance=1cm, sibling distance=1cm] {node {}
                      child[viiva, level distance=0.5cm, sibling distance=0.5cm] {node(yy) {}}
                      child[viiva, level distance=0.5cm, sibling distance=0.5cm] {node(cbot) {}}
                    }
                    child {node {}}
                  }
                  child {node(ctop) {}}
                }
              }
            }
          }
          child[viiva, level distance=0.5cm, sibling distance=0.5cm] {node(abot) {}}
        }
        child {node {}}
      }
      child {node(atop) {}}
    }
  ;
\draw[decoration={brace,raise=5pt},decorate] (atop.east) -- node[xshift=12pt,yshift=-6pt,draw=none,fill=none] {$s$} (abot.south);
\draw[decoration={brace,mirror,raise=5pt},decorate] (btop.west) -- node[xshift=-12pt,yshift=-6pt,draw=none,fill=none] {$p$} (bbot.south);
\draw[decoration={brace,raise=5pt},decorate] (ctop.east) -- node[xshift=12pt,yshift=-6pt,draw=none,fill=none] {$r$} (cbot.south);
\draw[decoration={brace,mirror,raise=5pt},decorate] (dtop.west) -- node[xshift=-12pt,yshift=-6pt,draw=none,fill=none] {$q - 1\,\,\,\,\,\,\,\,\,$} (dbot.south);
\draw (xx.south) node[yshift=-6pt,draw=none,fill=none] {$x$};
\draw (yy.south) node[yshift=-6pt,draw=none,fill=none] {$y$};
\draw (ctop.east) node[xshift=5pt,yshift=5pt,draw=none,fill=none] {$z$};
\draw (root.north) node[yshift=20pt,draw=none,fill=none] {$T'$};
\end{tikzpicture}
\hfill\mbox{}

\caption{Binary trees treealising a parallelogram grid.}
\label{fig:grid-trees}
\end{figure}

\begin{table}
\begin{center}
\footnotesize
\begin{tabular}{@{\quad\,\,}ccccccc}
\toprule
$i$                 & $\ld_T(i)$ & $\rd_T(i)$ & $\ld_{T'}(i)$ & $\rd_{T'}(i)$ & $\ld_T(i)-\ld_{T'}(i)$ & $\rd_T(i)-\rd_{T'}(i)$ \\
\midrule 
$1$                 & $2$           & $0$         & $2$              & $0$            & $0$       & $0$          \\
$2$                 & $2$           & $1$         & $2$              & $1$            & $0$       & $0$          \\
$\vdots$            & $\vdots$      & $\vdots$    & $\vdots$         & $\vdots$       & $\vdots$  & $\vdots$     \\
$p$                 & $2$           & $p - 1$     & $2$              & $p - 1$        & $0$       & $0$          \\
$\llap{$x = {}$}
 p + 1$             & $r + 1$       & $p$         & $1$              & $p$            & $r$       & $0$          \\
$p + 2$             & $r + 1$       & $p + 1$     & $s + 1$          & $1$            & $r - s$   & $p$          \\
$p + 3$             & $r + 1$       & $p + 2$     & $s + 1$          & $2$            & $r - s$   & $p$          \\
$\vdots$            & $\vdots$      & $\vdots$    & $\vdots$         & $\vdots$       & $\vdots$  & $\vdots$     \\
$p + q$             & $r + 1$       & $p + q - 1$ & $s + 1$          & $q - 1$        & $r - s$   & $p$          \\
$\llap{$y = {}$}
 p + q + 1$         & $r$           & $p + q$     & $r + s$          & $q$            & $-s$      & $p$          \\
$p + q + 2$         & $r - 1$       & $p + 1$     & $r + s - 1$      & $q + 1$        & $-s$      & $p - q$      \\
$p + q + 3$         & $r - 2$       & $p + 1$     & $r + s - 2$      & $q + 1$        & $-s$      & $p - q$      \\
$\vdots$            & $\vdots$      & $\vdots$    & $\vdots$         & $\vdots$       & $\vdots$  & $\vdots$     \\
$p + q + r$         & $1$           & $p + 1$     & $s + 1$          & $q + 1$        & $-s$      & $p - q$      \\
$\llap{$z = {}$}
 p + q + r + 1$     & $s$           & $1$         & $s$              & $q + 1$        & $0$       & $-q$         \\
$p + q + r + 2$     & $s - 1$       & $2$         & $s - 1$          & $2$            & $0$       & $0$          \\
$p + q + r + 3$     & $s - 2$       & $2$         & $s - 2$          & $2$            & $0$       & $0$          \\
$\vdots$            & $\vdots$      & $\vdots$    & $\vdots$         & $\vdots$       & $\vdots$  & $\vdots$     \\
$p + q + r + s + 1$ & $0$           & $2$         & $0$              & $2$            & $0$       & $0$          \\
\bottomrule
\end{tabular}
\end{center}

\bigskip
\caption{Left and right depth sequences of $T$ and $T'$ and their differences.}
\label{table:grid-trees}
\end{table}

\begin{example}
The proof of Lemma~\ref{lem:treealisable-2} reveals that the parallelogram grid $\Lambda = \ZZ(6,3) + \ZZ(10,0)$ of Example~\ref{ex:par-grid-1} (see Figure~\ref{fig:grid}) is treealised by (i.e., $\Lambda = \Lambda_{T,T'}$ holds for) the pair of trees $T$ and $T'$ shown in Figure~\ref{fig:grid-trees} with
$p = 3$,
$q = 30$,
$r = 10$,
$s = 4$.
In fact, we could take $q = 15$ instead of $30$.
\end{example}

\begin{lemma}
\label{lem:treealisable-1}
If $\Lambda$ is a one-dimensional subgroup of $\ZZ\times\ZZ$, then $\Lambda$ is not treealisable.
\end{lemma}

\begin{proof}
Assume that $T, T' \in \mathcal{T}_n$ and $T \neq T'$.
We will prove that $\Lambda_{T,T'}$ is a two\hyp{}dimensional subgroup of $\ZZ\times\ZZ$.
Since $T \neq T'$, there is some $i \in \nset{n}$ such that $(\ld_T(i) - \ld_{T'}(i), \rd_T(i) - \rd_{T'}(i)) \neq (0,0)$, as binary trees are determined by their left (right) depth sequences.
By Lemma~\ref{lem:dr}\ref{lem:dr:l}, the least such $i$ satisfies
$(\ld_T(i) - \ld_{T'}(i), \rd_T(i) - \rd_{T'}(i)) = (k,0) \in \Lambda_{T,T'}$, where $k$ is a nonzero integer.
Similarly, by Lemma~\ref{lem:dr}\ref{lem:dr:r}, the greatest such $i$ satisfies
$(\ld_T(i) - \ld_{T'}(i), \rd_T(i) - \rd_{T'}(i)) = (0, \ell) \in \Lambda_{T,T'}$, where $\ell$ is a nonzero integer.
Therefore, $\Lambda_{T,T'} \supseteq \ZZ(k,0) + \ZZ(0,\ell)$, and this proves that $\Lambda_{T,T'}$ is indeed a two\hyp{}dimensional subgroup of $\ZZ\times\ZZ$.
\end{proof}

\begin{corollary}\label{cor:sim_ab0}
If $a$ and $b$ are integers and $(a,b) \neq (0,0)$, then $\eqabm{a}{b}{0}$ is the equality relation, i.e., binary trees are uniquely determined by any nontrivial linear combination of their left and right depths.
\end{corollary}
\begin{proof}
Observe that ${\eqabm{a}{b}{0}} = {\sim_\Lambda}$, where
$\Lambda = \ZZ\big(b/\gcd(a,b),-a/\gcd(a,b)\big)$
is a one\hyp{}dimensional subgroup of $\ZZ \times \ZZ$.
By the proof of Lemma~\ref{lem:treealisable-1}, $\Lambda$ cannot contain any grid of the form $\Lambda_{T,T'}$ with $T \neq T'$.
\end{proof}

\begin{remark}
If $\Lambda = \ZZ(a,b)$ with $a<0$ and $b>0$, then we can provide an alternative proof for non-treealisability of $\Lambda$ using Proposition~\ref{prop:random walk}.
Let us consider the function $f(x)=x^{{-a}}+x^b$ on the real interval $[0,1]$.
Since $f$ is continuous and $0=f(0) < 1 < f(1)=2$, there is a number $\vartheta \in (0,1)$ such that $f(\vartheta)=1$.
If $\Lambda_{T,T'}=\Lambda$ for some $T,T' \in \mathcal{T}_n$, then we certainly have $T \sim_\Lambda T'$, and this implies $b\ld_T(i)-a\rd_T(i) = b\ld_{T'}(i)-a\rd_{T'}(i)$ for all $i \in \nset{n}$, as $\Lambda \subseteq \{ (r,s) : br-as=0 \}$ (see Definition~\ref{def:Lambda_T}).
We can rewrite this condition as $p^{\ld_T(i)} \cdot (1-p)^{\rd_T(i)} = p^{\ld_{T'}(i)} \cdot (1-p)^{\rd_{T'}(i)}$ with $p=\vartheta^b$ and $1-p=\vartheta^{-a}$.
Now Proposition~\ref{prop:random walk} shows that $T=T'$, i.e., $\Lambda = \{(0,0)\}$, which is clearly false.
\end{remark}

We have thus obtained a complete characterisation of treealisable subgroups of $\ZZ \times \ZZ$.

\begin{theorem}
\label{thm:treealisable}
    A subgroup of $\ZZ\times\ZZ$ is treealisable if and only if it is either a two-dimensional subgroup or the trivial subgroup.
\end{theorem}

\begin{proof}
This follows from Lemmata~\ref{lem:treealisable-2} and \ref{lem:treealisable-1} and from the trivial fact that the trivial subgroup $\{(0,0)\} \leq \ZZ \times \ZZ$ is treealisable (just take $T = T'$).
\end{proof}

\section{The lattice of associative spectra of linear quasigroups}

The following theorem gives a complete characterisation of fine associative spectra of linear quasigroups. 
Recall that we have identified bracketings with binary trees, hence equivalence relations on binary trees are also equivalences on bracketings, and vice versa.
(In particular, in the first item of the theorem $\sigma$ is given as an equivalence relation on bracketings, while in the other items $\sigma$ is given as an equivalence relation on binary trees.)

\begin{theorem}
\label{thm:spectra}
Let $\sigma_n$ be an equivalence relation on $B_n$ for each $n \in \IN_{+}$, and let $\sigma=(\sigma_n)_{n \in \IN_{+}}$.
Then the following four conditions are equivalent:
\begin{enumerate}[label={\upshape (\roman*)}]
\item\label{thm:spectra:sigma} $\sigma = \sigma(\alg{A})$ for some linear quasigroup $\alg{A}$;
\item\label{thm:spectra:G} $\sigma = {\eqabG{a}{b}{\alg{G}}}$ for some group $\alg{G}$ and elements $a,b \in G$;
\item\label{thm:spectra:Lambda} $\sigma = {\sim_\Lambda}$ for some $\Lambda \in \Sub(\ZZ\times\ZZ)$;
\item\label{thm:spectra:modulo} $\sigma = {\eqR{k}} \cap {\eqabm{a}{b}{m}}$ for suitable nonnegative integers $k,a,b,m$.
\end{enumerate}
\end{theorem}

\begin{proof}
Proposition~\ref{prop:linear-gen} shows that $\sigma(\alg{A}) = {\eqabG{\varphi_0}{\varphi_1}{\Aut(A, {+})}}$ (see Remark~\ref{remark:linear-gen}), and this verifies \ref{thm:spectra:sigma} $\implies$ \ref{thm:spectra:G}.
We have ${\eqabG{a}{b}{\alg{G}}} = {\sim_{\pairs{\alg{G}}{a}{b}}}$ by Proposition~\ref{prop:group vs lin} (see Remark~\ref{remark:G vs Lambda}\ref{remark:G vs Lambda:2}), thus \ref{thm:spectra:G} $\implies$ \ref{thm:spectra:Lambda}.

Let us now prove that \ref{thm:spectra:Lambda} $\implies$ \ref{thm:spectra:modulo}.
Assume first that $\Lambda \in \Sub(\ZZ\times\ZZ)$ is an at most one-dimensional subgroup.
We have seen in the proof of Lemma~\ref{lem:treealisable-1} that if $T \neq T'$, then $\Lambda_{T,T'}$ is a two-dimensional subgroup of $\ZZ\times\ZZ$; thus ${\sim_{\Lambda}}$ is the equality relation (cf. Remark~\ref{remark:G vs Lambda}\ref{remark:G vs Lambda:1}), which can be written, e.g., as ${\eqR{0}} \cap {\eqabm{0}{0}{0}}$ (note that the modulo $0$ congruence is just the equality relation).
On the other hand, if $\Lambda = \ZZ(u,v) + \ZZ(w,0)$ is a two-dimensional subgroup of $\ZZ\times\ZZ$, then \eqref{eq:grid condition} yields that ${\sim_{\Lambda}} = {\eqR{v}} \cap {\eqabm{v}{vw-u}{vw}}$ (see Remark~\ref{rem:Lambda vs R lin}).

Finally, \ref{thm:spectra:modulo} $\implies$ \ref{thm:spectra:sigma} follows from Example~\ref{example:roots of unity}(f) when both $k$ and $m$ are positive.
If $k=0$, equality of right depths forces equality of trees, and the equality relation is realised by the arithmetic mean on $\RR$ (Example~\ref{ex:rings}).
If $m=0$ and $(a,b)\ne(0,0)$, the same conclusion follows from Corollary~\ref{cor:sim_ab0}.
If $k>0$, $m=0$ and $a=b=0$, then we have the relation $\eqR{k}$, which is realised by Example~\ref{example:roots of unity}(c).
\end{proof}

\begin{remark}
Item \ref{thm:spectra:modulo} in the previous theorem seems a bit asymmetric. 
This is due to the fact that we agreed to choose the generators of our parallelogram grids in such a way that one of the vectors lies on the $x$ axis.
Had we chosen one of the generators on the $y$ axis, then we would have gotten the relation in the form  ${\eqL{k}} \cap {\eqabm{a}{b}{m}}$ (of course, with other integers $k,a,b,m$).
For instance, for the parallelogram grid $\Lambda = \ZZ(6,3) + \ZZ(10,0)$ of Example~\ref{ex:par-grid-1}, such a generating set is $\{(0,15), (2,6)\}$, as can be easily observed from Figure~\ref{fig:grid}.
Proposition~\ref{prop:phi-hat} offers a symmetric, but redundant description of grids and associative spectra.
For the grid $\Lambda = \ZZ(6,3) + \ZZ(10,0)$ of Example~\ref{ex:par-grid-1}, we can write this symmetric description as ${\sim_{\Lambda}} = {\eqR{3}} \cap {\eqabm{1}{8}{10}} = {\eqL{2}} \cap {\eqR{3}} \cap {\eqabm{1}{8}{10}}$ (cf. Example~\ref{ex:par-grid-2}).
\end{remark}

By Theorem~\ref{thm:spectra}, the fine associative spectra of linear quasigroups are exactly the relations of the form $\sim_\Lambda\ (\Lambda \leq \ZZ\times\ZZ)$.
We refine this theorem by describing when two subgroups of $\ZZ\times\ZZ$ yield the same equivalence relation, and we also determine the structure of the partially ordered set of fine spectra of linear quasigroups (the ordering is the containment of equivalence relations, or, equivalently, refinement of the corresponding partitions).

\begin{theorem}
\label{thm:FS}
Let $\FS$ denote the partially ordered set of fine spectra of linear quasigroups ordered by inclusion.
Then $\FS$ is a lattice, and the following map is an order isomorphism preserving meets:
\[
\Psi\colon \Sub_2(\ZZ\times\ZZ) \cup \bigl\{\{(0,0)\}\bigr\} \to \FS,\ \Lambda \to {\sim_\Lambda}.
\]
\end{theorem}

\begin{proof}
Theorem~\ref{thm:spectra} shows that $\Phi\colon \Sub(\ZZ\times\ZZ) \to \FS,\ \Lambda \to {\sim_\Lambda}$ is a surjective map.
However, this map is not injective: it follows from Lemma~\ref{lem:treealisable-1} that if $\Lambda \in \Sub(\ZZ\times\ZZ)$ is an at most one-dimensional subgroup, then $\Phi(\Lambda)$ is the equality relation (as explained in the proof of Theorem~\ref{thm:spectra}).
On the other hand, we can see from Lemma~\ref{lem:treealisable-2} that the restriction of $\Phi$ to $\Sub_2(\ZZ\times\ZZ)$ is injective.
Indeed, by Remark~\ref{remark:G vs Lambda}\ref{remark:G vs Lambda:1}, if $\Lambda_1,\Lambda_2 \in \Sub_2(\ZZ\times\ZZ)$, then $\Phi(\Lambda_1)=\Phi(\Lambda_2)$ if and only if $\Lambda_1$ and $\Lambda_2$ contain the same treealisable subgroups of $\ZZ\times\ZZ$. 
By Lemma~\ref{lem:treealisable-2}, both $\Lambda_1$ and $\Lambda_2$ are treealisable; hence $\Phi(\Lambda_1)=\Phi(\Lambda_2)$ can hold only if  $\Lambda_1$ and $\Lambda_2$ mutually contain each other.
This argument also shows that if $\Lambda_1 \in \Sub(\ZZ\times\ZZ)\setminus\Sub_2(\ZZ\times\ZZ)$ and $\Lambda_2 \in \Sub_2(\ZZ\times\ZZ)$, then $\Phi(\Lambda_1) \neq \Phi(\Lambda_2)$, as $\Phi(\Lambda_1)$ is the equality relation, while $\Phi(\Lambda_2)$ is not.
This proves that the map $\Psi$ given in the statement of the theorem is bijective.

For the proof that $\Psi$ preserves meets, the cases involving the zero subgroup are immediate.
It therefore suffices to verify that ${\sim_{\Lambda_1 \cap \Lambda_2}} = {\sim_{\Lambda_1}} \cap {\sim_{\Lambda_2}}$ for all $\Lambda_1,\Lambda_2 \in \Sub_2(\ZZ\times\ZZ)$.
This follows directly from Remark~\ref{remark:G vs Lambda}\ref{remark:G vs Lambda:1}:
\begin{align*}
T \sim_{\Lambda_1 \cap \Lambda_2} T' 
&\iff \Lambda_{T,T'} \subseteq \Lambda_1 \cap \Lambda_2\\ 
&\iff \Lambda_{T,T'} \subseteq \Lambda_1 \text{ and } \Lambda_{T,T'} \subseteq \Lambda_2\\
&\iff T \sim_{\Lambda_1} T' \text{ and } T \sim_{\Lambda_2} T'\\ 
&\iff  T \, ({\sim_{\Lambda_1}} \cap {\sim_{\Lambda_2}}) \, T'.
\end{align*}
To finish the proof, we just need to recall the well-known fact that every bijective meet-homomorphism is an order isomorphism.
\end{proof}

\begin{remark}
\label{remark:not a sublattice}
If we consider fine spectra of all groupoids (not just linear quasigroups), then we obtain an uncountable lattice $\alg{FS}$ in which the meet operation is the intersection and the join operation is the transitive closure of the union (just as in the lattice of all equivalence relations) \cite{LieWal-2009}.
By Theorem~\ref{thm:FS}, $\FS$ is a countable meet-subsemilattice of $\alg{FS}$.
However, it is not a sublattice, as $\FS$ is not closed under the join operation of $\alg{FS}$.
Indeed, by Example~\ref{ex:5-trees-2}, the transitive closure of the union of ${\eqL{3}}$ and ${\eqR{3}}$ is not the total relation on $\mathcal T_5$.
However, $\Psi^{-1}({\eqL{3}}) = \ZZ(3,0)+\ZZ(0,1)$ and $\Psi^{-1}({\eqR{3}}) = \ZZ(1,0)+\ZZ(0,3)$, and the join of these two subgroups in $\Sub_2(\ZZ\times\ZZ)$ is clearly $\ZZ\times\ZZ$.
Thus the join of $\eqL{3}$ and $\eqR{3}$ in the lattice $\FS$ is $\Psi(\ZZ\times\ZZ)$, which is the total relation on each $\mathcal T_n$, that is, $\bigcup_{n\ge1}(\mathcal T_n\times\mathcal T_n)$.
Thus, even though $\FS$ is a lattice in its own right, it is not a sublattice of $\alg{FS}$, and the join of two relations in $\FS$ can sometimes be larger than their join in $\alg{FS}$.
\end{remark}

\begin{remark}
\label{remark:coatoms}
We have seen in Remark~\ref{remark:Sub2} that $\Sub_2(\ZZ\times\ZZ)$ is a sublattice of $\Sub(\ZZ\times\ZZ)$, and Theorem~\ref{thm:FS} tells us that, up to isomorphism, the lattice $\FS$ can be obtained by attaching an extra bottom element to $\Sub_2(\ZZ\times\ZZ)$.
This lattice has no atoms (again, by Remark~\ref{remark:Sub2}), and its coatoms correspond to the maximal subgroups of $\ZZ\times\ZZ$, i.e., to those subgroups that have a prime index.
Thus the coatoms of $\FS$ are the following:
\begin{itemize}
    \item $\Psi(\ZZ(p,0)+\ZZ(0,1)) = {\eqL{p}}$ for any prime number $p$;
    \item $\Psi(\ZZ(1,0)+\ZZ(0,p)) = {\eqR{p}}$ for any prime number $p$;
    \item $\Psi(\ZZ(p,0)+\ZZ(u,1)) = {\eqabm{1}{-u}{p}} = {\eqabm{1}{p-u}{p}}$ for any prime number $p$ and $u \in \{1,\dots,p-1\}$.
\end{itemize}
\end{remark}

We conclude the paper by several corollaries of our description of fine associative spectra of linear quasigroups. 
First we prove that the associative spectrum of a linear quasigroup grows either exponentially, or it is constant $1$ (in the latter case, the quasigroup is a group, of course). 
A similar dichotomy holds for graph algebras \cite{LehWal-2022}, standing in stark contrast to arbitrary groupoids, for which polynomial spectra of arbitrary degrees do exist \cite{LieWal-2009}.
\begin{corollary}
    If $\alg{A}$ is a linear quasigroup, then either $s_n(\alg{A}) = 1$ for every $n \geq 1$, or $s_n(\alg{A}) \geq 2^{n-2}$ for every $n \geq 1$.
\end{corollary}

\begin{proof}
Both conditions obviously hold for $n = 1$, so assume $n \geq 2$.
If $\alg{A}$ is a nonassociative linear quasigroup, then its fine spectrum is contained in one of the coatoms of $\FS$; therefore, by Remark~\ref{remark:coatoms}, at least one of the following inequalities holds:
\begin{itemize}
    \item $s_n(\alg{A}) \geq \NeqL{p}{n}$ (where $p$ is a prime number),
    \item $s_n(\alg{A}) \geq \NeqR{p}{n}$ (where $p$ is a prime number),
    \item $s_n(\alg{A}) \geq \Neqabm{1}{-u}{p}{n}$ (where $p$ is a prime number and $u \in \{1,\dots,p-1\}$).
\end{itemize}
It was proved in \cite{HeiHua-2017} that $\NeqL{p}{n} \geq \NeqL{2}{n}$ for all $p \geq 2$, and we also know that $\NeqL{2}{n}= 2^{n-2}$. This settles the first case, and then the second case is also done, as $\NeqR{p}{n} = \NeqL{p}{n}$.
The third case reduces to the second one, as Lemma~\ref{lem:cirmi} gives $\Neqabm{1}{-u}{p}{n} \geq \NeqR{p}{n}$.
\end{proof}

\begin{remark}
Let us outline an alternative proof of the corollary above that does not use Lemma~\ref{lem:cirmi} and Remark~\ref{remark:coatoms}.
Assume $n \geq 2$.
Let us say that a binary tree $T \in \mathcal{T}_n$ is \emph{special} if $\ld_T(1),\dots,\ld_T(n-2) \in \{1,2\}$.
Right depth sequences of binary trees were described in \cite{CsaWal-2000,LehWal-2021} as so-called zag sequences or Sisyphus sequences, and dualising this result we see that we can arbitrarily choose $\ld_T(i)=1$ or $\ld_T(i)=2$ for $i=1,\dots,n-2$ in the left depth sequence of a special binary tree (of course, we must have $\ld_T(n-1)=1$ and $\ld_T(n)=0$).
Thus, there are exactly $2^{n-2}$ special binary trees in $\mathcal{T}_n$ (we invite the reader to imagine what these trees and the corresponding bracketings look like, although we will not need this for the proof).
Now if these special trees are pairwise inequivalent in the fine associative spectrum of $\alg{A}$, then we clearly have $s_n(\alg{A}) \geq 2^{n-2}$.

Assume then that $T \sim_\Lambda T'$ for two different special binary trees $T,T' \in \mathcal{T}_n$, where $\sim_\Lambda$ is the fine associative spectrum of our linear quasigroup $\alg{A}$.
For the least $i \in \nset{n}$ such that $\ld_T(i) \neq \ld_{T'}(i)$, we have $(\ld_T(i) - \ld_{T'}(i), \rd_T(i) - \rd_{T'}(i)) = (k,0)$ for some nonzero integer $k$ (see Lemma~\ref{lem:dr}).
Since the left depths in $T$ and in $T'$ are either $1$ or $2$ (except for the last two leaves), we have either $k=2-1=1$ or $k=1-2=-1$.
In both cases $(1,0) \in \Lambda_{T,T'} \subseteq \Lambda$, which implies that $\Lambda$ is of the form $\Lambda = \ZZ(1,0) + \ZZ(0,\ell)$ for some positive integer $\ell$.
If $\ell=1$ then $\Lambda = \ZZ\times\ZZ$, thus $s_n(\alg{A})=1$, while if $\ell \geq 2$, then ${\sim_\Lambda} = {\eqR{\ell}}$, thus $s_n(\alg{A}) = \NeqR{\ell}{n} \geq \NeqR{2}{n} = 2^{n-2}$ (the last inequality is again from \cite{HeiHua-2017}).
\end{remark}

The next corollary describes the two extremal cases: associative and antiassociative linear quasigroups.
The former is trivial, the latter follows from results of \cite{Belousov-1966}, yet we feel that it is worth including their proofs as illustrations of our main results.

\begin{corollary}
\label{cor:antiassociative}
Let $\alg{A} = (A, {\circ})$ be a linear quasigroup over a group $(A, {+})$ with arithmetic form $(A, {+}, \varphi_0, \varphi_1, 0)$.
\begin{enumerate}[label={\upshape (\roman*)}]
\item The operation $\circ$ is associative if and only if $\varphi_0=\varphi_1=\id_A$.
\item The operation $\circ$ is antiassociative if and only if at least one of $\varphi_0$ and $\varphi_1$ is of infinite order.
\end{enumerate}
\end{corollary}

\begin{proof}
Let us put $\Lambda = \pairs{\Aut(A,{+})}{\varphi_0}{\varphi_1}$; then, by Propositions~\ref{prop:linear-gen}, \ref{prop:group vs lin} and Remark~\ref{remark:G vs Lambda}\ref{remark:G vs Lambda:2}, the fine associative spectrum of $(A, {\circ})$ is $\sim_\Lambda$
(as explained in the proof of Theorem~\ref{thm:spectra}).
\begin{enumerate}[label={\upshape (\roman*)}]
\item This can be easily proved directly from the definition of associativity, but let us use a sledgehammer to crack a nut, just for fun: according to Theorem~\ref{thm:FS}, $\sim_\Lambda$ is the top element of the lattice $\FS$ if and only if $\Lambda=\ZZ\times\ZZ$.
The latter holds only if $(1,0), (0,1) \in \Lambda$, and these mean exactly that $\varphi_0=\varphi_1=\id_A$,
according to Definition~\ref{def:Lambda_G}.

\item This was proved by Belousov \cite{Belousov-1966}, but we can also easily derive it from our results as follows. By (the proof of) Theorem~\ref{thm:FS}, $\sim_\Lambda$ is the bottom element of the lattice $\FS$ if and only if $\Lambda \notin \Sub_2(\ZZ\times\ZZ)$.
If $\varphi_0$ and $\varphi_1$ have finite orders $m_0$ and $m_1$, respectively, then $(m_0,0), (0,m_1) \in \Lambda$ (see Corollary~\ref{cor:phi-hat}), hence $\Lambda \in \Sub_2(\ZZ\times\ZZ)$.
If, say, $\varphi_0$ is of infinite order, then $\Lambda$ contains no element of the form $(k,0)$ other than $(0,0)$, thus $\Lambda \notin \Sub_2(\ZZ\times\ZZ)$. \qedhere
\end{enumerate}
\end{proof}

\begin{example}
\label{example:antiassociative}
Consider a linear operation $x \circ y = ax+by$ over a unital ring $\alg{R}$ for units $a,b \in \alg{R}^\ast$, as in Example~\ref{ex:rings}.
Corollary~\ref{cor:antiassociative} gives that this operation is not antiassociative if and only if both $a$ and $b$ are roots of unity.
\end{example}

According to Corollary~\ref{cor:antiassociative}, all antiassociative linear quasigroups are infinite.
The following example shows that each element of $\FS$ except the equality relation (i.e., the bottom element) can be realised by finite linear quasigroups.

\begin{example}
\label{ex:finite}
Let $\Lambda$ be a two-dimensional subgroup of $\ZZ \times \ZZ$, and let $\alg{Q}:=(\ZZ\times\ZZ)/\Lambda$ be the corresponding quotient group.
We write $\overline{(r,s)}$ for the coset $(r,s)+\Lambda \in Q$.
Let $(A, {+})$ be the additive group of the vector space over $\ZZ_2$ generated by the elements of $Q$ as basis vectors (in other words, $(A, {+})$ is the additive group of the group ring $\ZZ_2[Q]$).
Thus the elements of $A$ are formal linear combinations of the form $\sum_{q \in Q}\mu_q q$ with coefficients $\mu_q \in \ZZ_2$.
Every permutation of the basis vectors gives rise to an automorphism of $(A, {+})$; in particular, we have the following automorphisms $\varphi_0$ and $\varphi_1$:
\[
	\varphi_0\Big(\sum_{q \in Q}\mu_q q\Big) = \sum_{q \in Q}\mu_q \big(q+\overline{(1,0)}\big), \qquad
	\varphi_1\Big(\sum_{q \in Q}\mu_q q\Big) = \sum_{q \in Q}\mu_q \big(q+\overline{(0,1)}\big).
\]
Since $\varphi_0$ and $\varphi_1$ commute, we have $\varphi_{\addr_T(i)} = \varphi_0^{\ld_T(i)} \varphi_1^{\rd_T(i)}$ for all $T \in \mathcal{T}_n$ and $i \in \nset{n}$, just as in the proof of Proposition~\ref{prop:phi0-phi1-commute}, and this automorphism can be given explicitly as
\[
	\varphi_{\addr_T(i)}\Big(\sum_{q \in Q}\mu_q q\Big) = \sum_{q \in Q}\mu_q \big(q+\overline{(\ld_T(i),\rd_T(i))}\big).
\]
This implies, using Proposition~\ref{prop:linear-gen}, that the linear quasigroup $\alg{A}$ given by the arithmetic form $(A, {+}, \varphi_0, \varphi_1, 0)$ satisfies a bracketing identity $t \approx t'$ if and only if $\overline{(\ld_T(i),\rd_T(i))} = \overline{(\ld_T'(i),\rd_T'(i))}$ holds for the corresponding binary trees $T$ and $T'$ for all $i \in \nset{n}$.
This is equivalent to $(\ld_T(i)-\ld_{T'}(i),\rd_T(i)-\rd_{T'}(i)) \in \Lambda$, thus the fine associative spectrum of $\alg{A}$ is indeed $\sim_\Lambda$.
\end{example}

The last corollary describes implications between bracketing identities within the class of linear quasigroups.
Note that this is different from the usual deduction rules in equational theories (cf. Remark~\ref{remark:not a sublattice}).

\begin{corollary}
\label{cor:consequence}
Let $t_1, t'_1 \in B_{n_1}$ and $t_2, t'_2 \in B_{n_2}$, and let $T_1, T'_1, T_2, T'_2$ be the corresponding binary trees.
Then the following two conditions are equivalent:
\begin{enumerate}[label={\upshape (\roman*)}]
\item\label{cor:consequence:1} $t_2 \approx t'_2$ is a consequence of $t_1 \approx t'_1$ in the class of linear quasigroups, i.e., for every linear quasigroup $\alg{A}$, we have $\alg{A} \satisfies t_1 \approx t'_1 \implies \alg{A} \satisfies t_2 \approx t'_2$;
\item\label{cor:consequence:2} $\Lambda_{T_2,T'_2} \subseteq \Lambda_{T_1,T'_1}$.
\end{enumerate}
\end{corollary}

\begin{proof}
According to Theorem~\ref{thm:FS}, fine spectra of linear quasigroups are exactly the equivalence relations of the form $\sim_\Lambda$ with $\Lambda \in \Sub_2(\ZZ\times\ZZ) \cup \bigl\{\{(0,0)\}\bigr\}$.
Thus, using Remark~\ref{remark:G vs Lambda}\ref{remark:G vs Lambda:3}, we can reformulate condition \ref{cor:consequence:1} of the corollary as follows:
\begin{equation}\label{eq:cor:conseqence:1}
\forall \Lambda \in \Sub_2(\ZZ\times\ZZ) \cup \bigl\{\{(0,0)\}\bigr\} \colon \Lambda_{T_1,T'_1} \subseteq \Lambda \implies \Lambda_{T_2,T'_2} \subseteq \Lambda.   
\end{equation}
It is clear that condition \ref{cor:consequence:2} implies \eqref{eq:cor:conseqence:1}.
Conversely, assume that \eqref{eq:cor:conseqence:1} holds, and let us put $\Lambda = \Lambda_{T_1,T'_1}$ (note that in this case $\Lambda$ indeed belongs to $\Sub_2(\ZZ\times\ZZ) \cup \bigl\{\{(0,0)\}\bigr\}$, by Theorem~\ref{thm:treealisable}).
Now \eqref{eq:cor:conseqence:1} yields $\Lambda_{T_2,T'_2} \subseteq \Lambda_{T_1,T'_1}$; hence condition \ref{cor:consequence:2} is satisfied.
\end{proof}

\begin{remark}
\label{rem:single-identity}
Fine associative spectra of linear quasigroups can be described by a single identity.
Indeed, let $\{t_i \approx t'_i : i \in I\}$ be a (possibly infinite) set of bracketing identities, and let $\Lambda_i$ be the grid corresponding to the identity $t_i \approx t'_i$.
Discarding trivial identities, we can assume that each $\Lambda_i$ is a two-dimensional subgroup of $\ZZ \times \ZZ$, hence $\Lambda := \sum_{i \in I}\Lambda_i$ is also a two-dimensional subgroup.
According to Lemma~\ref{lem:treealisable-2}, $\Lambda = \Lambda_{T,T'}$ for a suitable pair of binary trees.
Denoting the corresponding bracketings by $t$ and $t'$, the argument used in the proof of Corollary~\ref{cor:consequence} shows that a linear quasigroup satisfies the identities $t_i \approx t'_i\ (i \in I)$ if and only if it satisfies $t \approx t'$.
\end{remark}

\begin{example}
As an illustration of Corollary~\ref{cor:consequence}, let us show that if a linear quasigroup $\alg{A}$ satisfies the identity
\begin{equation}
x_1 ( x_2 ( x_3 ( \dots ( x_{m-1} x_m ) \dots ) ) ) \approx ( ( \dots ( ( x_1 x_2 ) x_3 ) \dots ) x_{m-1} ) x_m
\label{eq:ex:NieKep}
\end{equation}
for some $m \geq 3$, then $\alg{A}$ is associative and hence a group.
(This was proved in a slightly more general form for division groupoids by Niemenmaa and Kepka \cite[Theorem~4.1]{NieKep-1992}.)

The case $m = 3$ is trivial, so we assume that $m \geq 4$.
By Corollary~\ref{cor:consequence}, it suffices to show that $\Lambda_{T_2,T'_2} \subseteq \Lambda_{T_1,T'_1}$, where $T_1$ and $T'_1$ are the binary trees corresponding to the two terms in identity \eqref{eq:ex:NieKep} and $T_2$ and $T'_2$ correspond to the associative law, i.e., $\Lambda_{T_2,T'_2} = \ZZ \times \ZZ$.
The left and right depths of the leaves in $T_1$ and $T'_1$ and the differences thereof are shown in the following table.

\medskip
\begin{center}
\begin{tabular}{@{\quad\,\,}ccccccc}
\toprule
$i$                 & $\ld_{T_1}(i)$ & $\rd_{T_1}(i)$ & $\ld_{T'_1}(i)$ & $\rd_{T'_1}(i)$ & $\ld_{T_1}(i)-\ld_{T'_1}(i)$ & $\rd_{T_1}(i)-\rd_{T'_1}(i)$ \\
\midrule 
$1$                 & $1$           & $0$         & $m-1$            & $0$            & $2-m$     & $0$          \\
$2$                 & $1$           & $1$         & $m-2$            & $1$            & $3-m$     & $0$          \\
$\vdots$            & $\vdots$      & $\vdots$    & $\vdots$         & $\vdots$       & $\vdots$  & $\vdots$     \\
$m-1$               & $1$           & $m-2$       & $1$              & $1$            & $0$       & $m-3$        \\
$m$                 & $0$           & $m-1$       & $0$              & $1$            & $0$       & $m-2$        \\
\bottomrule
\end{tabular}
\end{center}

\medskip\noindent
Therefore,
$\{(2-m,0), (3-m,0), (0,m-3), (0,m-2)\} \subseteq \Lambda_{T_1,T'_1}$.
Consequently, $\Lambda_{T_1,T'_1}$ contains the pairs
$(3-m,0) - (2-m,0) = (1,0)$ and $(0,m-2) - (0,m-3) = (0,1)$, and it follows that
$\Lambda_{T_1,T'_1} = \ZZ \times \ZZ$, which concludes the proof.
\end{example}

\section{Open problems}

Equivalence relations arising as fine associative spectra of linear quasigroups have been classified above. The following enumerative problems concern their numbers of classes, which can coincide even when the relations themselves differ.

\begin{problem}
Find closed formulas for the number of equivalence classes of $n$\hyp{}leaf binary trees under the various equivalence relations considered in this paper: $\eqD{k}$ (for $k \geq 3$), $\eqLR{k}{\ell}$, $\eqabm{a}{b}{m}$, $\sim_\Lambda$.
\end{problem}

\begin{problem}
Catalan numbers and their modular variants are known to count many kinds of combinatorial objects.
In a similar way, find new interpretations for the numbers $\NeqD{k}{n}$, $\NeqLR{k}{\ell}{n}$, $\Neqabm{a}{b}{m}{n}$, $T_{\Lambda,n}$.
\end{problem}

In Example~\ref{ex:finite} we constructed a finite linear quasigroup with a given nontrivial fine associative spectrum $\sim_\Lambda$. Observe that this quasigroup is medial (since $\varphi_0$ and $\varphi_1$ commute), and its size is $2^m$, where $m = [\ZZ \times \ZZ : \Lambda]$.
If $m \geq 3$, then the subquasigroup $\big\{ \sum_{q \in Q}\mu_q q : \sum_{q \in Q}\mu_q = 0\big\}$ provides a linear quasigroup of size $2^{m-1}$ with the same fine associative spectrum. 
This motivates the following problem.
\begin{problem}
For any two-dimensional subgroup $\Lambda \leq \ZZ \times \ZZ$, find the smallest (medial) quasigroup whose fine associative spectrum is $\sim_\Lambda$.
\end{problem}

As mentioned in the introduction, a result of Belousov \cite[Theorem 4]{Belousov-1966} implies that a quasigroup satisfying a nontrivial bracketing identity must be affine.
Therefore, one could describe all possible associative spectra of quasigroups by considering affine quasigroups.

\begin{problem}
Extend the results of the current paper to affine quasigroups.
\end{problem}

\appendix
\section{Numerical data}
\label{app:Tlin}

Table~\ref{tab:values} shows the number of $\eqD{k}$\hyp{}equivalence classes of $\mathcal{T}_n$ for $1 \leq k \leq 8$ and $1 \leq n \leq 15$.

Table~\ref{table:L-eq} shows the number of $\eqL{k}$\hyp{}equivalence classes of $\mathcal{T}_n$ (modular Catalan numbers), $\NeqL{k}{n} = C_{k,n-1}$, for $k, n \leq 15$.

Table~\ref{tab:Nabmn} shows the number of $\eqabm{a}{b}{m}$\hyp{}equivalence classes of $\mathcal{T}_n$ for $1 \leq a \leq b \leq m \leq 14$, using $m$ to represent the zero residue.
Note that, by Lemma~\ref{lem:abm}, the triples $(a,b,m)$ and $(a',b',m')$ yield the same sequence if
\begin{itemize}
\item $a' = \ell a$, $b' = \ell b$, $m' = \ell m$ for some $\ell \in \IN_+$;
\item $a' = b$, $b' = a$, $m' = m$; or
\item $a' = \ell a$, $b' = \ell b$, $m' = m$ for some unit $\ell$ modulo $m$.
\end{itemize}
After dividing out common factors, we retain triples with $\gcd(a,b,m) = 1$.
Coefficients may be replaced by congruent coefficients.
Among pairs related by simultaneous multiplication by a unit modulo $m$ and interchange of the coordinates, we choose the lexicographically least representative with $1 \leq a \leq b \leq m$.
The table entries are arranged in the lexicographical order of the first 14 terms of the sequence $(\Neqabm{a}{b}{m}{n})_{n \in \IN_+}$.
The last column indicates cases where the sequence coincides with that of another equivalence relation based on (left, right) depth sequences, according to Lemma~\ref{lem:abm}.
A question mark means that no identification with one of these previously named counting sequences is supplied here; it does not assert that the sequence is new or absent from the OEIS.
Repeated finite rows alone do not prove equality of the full sequences.

Table~\ref{tab:grid-eq} shows the number $T_{\Lambda,n}$ of $\sim_\Lambda$\hyp{}equivalence classes of $\mathcal{T}_n$, where $\Lambda$ is the parallelogram grid $\ZZ(u,v) + \ZZ(w,0)$ with $0 \leq u < w \leq 5$ and $0 < v \leq 5$.

\setlength{\rotFPtop}{0pt plus 1fil}
\setlength{\rotFPbot}{0pt plus 1fil}

\begin{sidewaystable}[p]
\setlength{\tabcolsep}{3pt}
\begin{center}
\begin{tabular}{rrrrrrrrrrrrrrrrc}
\toprule
$k$ \textbackslash{} $n$ & 1 & 2 & 3 & 4 & 5 & 6 & 7 & 8 & 9 & 10 & 11 & 12 & 13 & 14 & 15 & OEIS\\
\midrule
1 & 1 & 1 & 1 & 1 & 1 & 1 & 1 & 1 & 1 & 1 & 1 & 1 & 1 & 1 & 1 & A000012 \\
2 & 1 & 1 & 2 & 5 & 10 & 21 & 42 & 85 & 170 & 341 & 682 & 1\,365 & 2\,730 & 5\,461 & 10\,922 & A000975 \\
3 & 1 & 1 & 2 & 5 & 14 & 42 & 129 & 398 & 1\,223 & 3\,752 & 11\,510 & 35\,305 & 108\,217 & 331\,434 & 1\,014\,304 & \\
4 & 1 & 1 & 2 & 5 & 14 & 42 & 132 & 429 & 1\,429 & 4\,849 & 16\,689 & 58\,074 & 203\,839 & 720\,429 & 2\,560\,520 & \\
5 & 1 & 1 & 2 & 5 & 14 & 42 & 132 & 429 & 1\,430 & 4\,862 & 16\,795 & 58\,773 & 207\,906 & 742\,203 & 2\,670\,389 & \\
6 & 1 & 1 & 2 & 5 & 14 & 42 & 132 & 429 & 1\,430 & 4\,862 & 16\,796 & 58\,786 & 208\,011 & 742\,885 & 2\,674\,303 & \\
7 & 1 & 1 & 2 & 5 & 14 & 42 & 132 & 429 & 1\,430 & 4\,862 & 16\,796 & 58\,786 & 208\,012 & 742\,900 & 2\,674\,439 & \\
8 & 1 & 1 & 2 & 5 & 14 & 42 & 132 & 429 & 1\,430 & 4\,862 & 16\,796 & 58\,786 & 208\,012 & 742\,900 & 2\,674\,440 & \\
\midrule
$C_{n-1}$ & 1 & 1 & 2 & 5 & 14 & 42 & 132 & 429 & 1\,430 & 4\,862 & 16\,796 & 58\,786 & 208\,012 & 742\,900 & 2\,674\,440 & A000108 \\
\bottomrule
\end{tabular}
\end{center}

\bigskip
\caption{The number $\NeqD{k}{n}$ of binary trees with $n$ leaves up to $k$\hyp{}depth\hyp{}equivalence and Catalan numbers $C_{n-1}$.}
\label{tab:values}
\end{sidewaystable}

\begin{sidewaystable}[p]
\setlength{\tabcolsep}{3pt}
\begin{center}
\begin{tabular}{r*{15}{r}c}
\toprule
$k$ \textbackslash{} $n$ & 1 & 2 & 3 & 4 & 5 & 6 & 7 & 8 & 9 & 10 & 11 & 12 & 13 & 14 & 15 & OEIS \\
\midrule
 1 & 1 & 1 & 1 & 1 &  1 &  1 &   1 &   1 &      1 &      1 &       1 &       1 &        1 &        1 &           1 & A000012 \\
 2 & 1 & 1 & 2 & 4 &  8 & 16 &  32 &  64 &    128 &    256 &     512 &  1\,024 &   2\,048 &   4\,096 &      8\,192 & A011782 \\
 3 & 1 & 1 & 2 & 5 & 13 & 35 &  96 & 267 &    750 & 2\,123 &  6\,046 & 17\,303 &  49\,721 & 143\,365 &    414\,584 & A005773 \\
 4 & 1 & 1 & 2 & 5 & 14 & 41 & 124 & 384 & 1\,210 & 3\,865 & 12\,482 & 40\,677 & 133\,572 & 441\,468 & 1\,467\,296 & A159772 \\
 5 & 1 & 1 & 2 & 5 & 14 & 42 & 131 & 420 & 1\,375 & 4\,576 & 15\,431 & 52\,603 & 180\,957 & 627\,340 & 2\,189\,430 & A261588 \\
 6 & 1 & 1 & 2 & 5 & 14 & 42 & 132 & 428 & 1\,420 & 4\,796 & 16\,432 & 56\,966 & 199\,444 & 704\,146 & 2\,504\,000 & A261589 \\
 7 & 1 & 1 & 2 & 5 & 14 & 42 & 132 & 429 & 1\,429 & 4\,851 & 16\,718 & 58\,331 & 205\,632 & 731\,272 & 2\,620\,176 & A261590 \\
 8 & 1 & 1 & 2 & 5 & 14 & 42 & 132 & 429 & 1\,430 & 4\,861 & 16\,784 & 58\,695 & 207\,452 & 739\,840 & 2\,658\,936 & A261591 \\
 9 & 1 & 1 & 2 & 5 & 14 & 42 & 132 & 429 & 1\,430 & 4\,862 & 16\,795 & 58\,773 & 207\,907 & 742\,220 & 2\,670\,564 & A261592 \\
10 & 1 & 1 & 2 & 5 & 14 & 42 & 132 & 429 & 1\,430 & 4\,862 & 16\,796 & 58\,785 & 207\,998 & 742\,780 & 2\,673\,624 & \\
11 & 1 & 1 & 2 & 5 & 14 & 42 & 132 & 429 & 1\,430 & 4\,862 & 16\,796 & 58\,786 & 208\,011 & 742\,885 & 2\,674\,304 & \\
12 & 1 & 1 & 2 & 5 & 14 & 42 & 132 & 429 & 1\,430 & 4\,862 & 16\,796 & 58\,786 & 208\,012 & 742\,899 & 2\,674\,424 & \\
13 & 1 & 1 & 2 & 5 & 14 & 42 & 132 & 429 & 1\,430 & 4\,862 & 16\,796 & 58\,786 & 208\,012 & 742\,900 & 2\,674\,439 & \\
14 & 1 & 1 & 2 & 5 & 14 & 42 & 132 & 429 & 1\,430 & 4\,862 & 16\,796 & 58\,786 & 208\,012 & 742\,900 & 2\,674\,440 & \\
15 & 1 & 1 & 2 & 5 & 14 & 42 & 132 & 429 & 1\,430 & 4\,862 & 16\,796 & 58\,786 & 208\,012 & 742\,900 & 2\,674\,440 & \\
\midrule
$C_{n-1}$
   & 1 & 1 & 2 & 5 & 14 & 42 & 132 & 429 & 1\,430 & 4\,862 & 16\,796 & 58\,786 & 208\,012 & 742\,900 & 2\,674\,440 & A000108 \\
\bottomrule
\end{tabular}
\end{center}

\medskip
\caption{The number $\NeqL{k}{n}$ of binary trees with $n$ leaves up to $k$\hyp{}left\hyp{}depth\hyp{}equivalence, i.e., modular Catalan number $C_{k,n-1}$.}
\label{table:L-eq}
\end{sidewaystable}

\begin{sidewaystable}[p]
\setlength{\tabcolsep}{3pt}
\begin{center}
\small
\begin{tabular}{rrr*{14}{r}l}
\toprule
$a$ & $b$ & $m$ & $n = {}$1 & 2 & 3 & 4 & 5 & 6 & 7 & 8 & 9 & 10 & 11 & 12 & 13 & 14 & Known identification \\
\midrule
1 & 1 & 1 & 1 & 1 & 1 & 1 & 1 & 1 & 1 & 1 & 1 & 1 & 1 & 1 & 1 & 1 & \\
1 & 2 & 2 & 1 & 1 & 2 & 4 & 8 & 16 & 32 & 64 & 128 & 256 & 512 & 1024 & 2048 & 4096 & $\NeqL{2}{n}$ \\
1 & 1 & 2 & 1 & 1 & 2 & 5 & 10 & 21 & 42 & 85 & 170 & 341 & 682 & 1365 & 2730 & 5461 & $\NeqD{2}{n}$ \\
1 & 3 & 3 & 1 & 1 & 2 & 5 & 13 & 35 & 96 & 267 & 750 & 2123 & 6046 & 17303 & 49721 & 143365 & $\NeqL{3}{n}$ \\
1 & 4 & 4 & 1 & 1 & 2 & 5 & 14 & 41 & 124 & 384 & 1210 & 3865 & 12482 & 40677 & 133572 & 441468 & $\NeqL{4}{n}$ \\
2 & 3 & 6 & 1 & 1 & 2 & 5 & 14 & 41 & 124 & 384 & 1210 & 3865 & 12482 & 40677 & 133572 & 441468 & $\NeqLR{3}{2}{n}$ \\
1 & 2 & 3 & 1 & 1 & 2 & 5 & 14 & 42 & 128 & 390 & 1185 & 3586 & 10862 & 32929 & 99883 & 303000 & ? \\
1 & 1 & 3 & 1 & 1 & 2 & 5 & 14 & 42 & 129 & 398 & 1223 & 3752 & 11510 & 35305 & 108217 & 331434 & $\NeqD{3}{n}$ \\
1 & 2 & 4 & 1 & 1 & 2 & 5 & 14 & 42 & 131 & 420 & 1374 & 4561 & 15306 & 51793 & 176404 & 603990 & ? \\
1 & 5 & 5 & 1 & 1 & 2 & 5 & 14 & 42 & 131 & 420 & 1375 & 4576 & 15431 & 52603 & 180957 & 627340 & $\NeqL{5}{n}$ \\
1 & 6 & 6 & 1 & 1 & 2 & 5 & 14 & 42 & 132 & 428 & 1420 & 4796 & 16432 & 56966 & 199444 & 704146 & $\NeqL{6}{n}$ \\
2 & 5 & 10 & 1 & 1 & 2 & 5 & 14 & 42 & 132 & 428 & 1420 & 4796 & 16432 & 56966 & 199444 & 704146 & $\NeqLR{5}{2}{n}$ \\
3 & 4 & 12 & 1 & 1 & 2 & 5 & 14 & 42 & 132 & 428 & 1420 & 4796 & 16432 & 56966 & 199444 & 704146 & $\NeqLR{4}{3}{n}$ \\
1 & 3 & 4 & 1 & 1 & 2 & 5 & 14 & 42 & 132 & 429 & 1425 & 4807 & 16402 & 56472 & 195860 & 683420 & ? \\
1 & 1 & 4 & 1 & 1 & 2 & 5 & 14 & 42 & 132 & 429 & 1429 & 4849 & 16689 & 58074 & 203839 & 720429 & $\NeqD{4}{n}$ \\
1 & 3 & 6 & 1 & 1 & 2 & 5 & 14 & 42 & 132 & 429 & 1429 & 4851 & 16718 & 58331 & 205631 & 731257 & ? \\
1 & 7 & 7 & 1 & 1 & 2 & 5 & 14 & 42 & 132 & 429 & 1429 & 4851 & 16718 & 58331 & 205632 & 731272 & $\NeqL{7}{n}$ \\
1 & 2 & 6 & 1 & 1 & 2 & 5 & 14 & 42 & 132 & 429 & 1430 & 4861 & 16784 & 58695 & 207450 & 739810 & ? \\
1 & 4 & 6 & 1 & 1 & 2 & 5 & 14 & 42 & 132 & 429 & 1430 & 4861 & 16784 & 58695 & 207452 & 739839 & ? \\
1 & 8 & 8 & 1 & 1 & 2 & 5 & 14 & 42 & 132 & 429 & 1430 & 4861 & 16784 & 58695 & 207452 & 739840 & $\NeqL{8}{n}$ \\
2 & 7 & 14 & 1 & 1 & 2 & 5 & 14 & 42 & 132 & 429 & 1430 & 4861 & 16784 & 58695 & 207452 & 739840 & $\NeqLR{7}{2}{n}$ \\
1 & 4 & 5 & 1 & 1 & 2 & 5 & 14 & 42 & 132 & 429 & 1430 & 4862 & 16790 & 58708 & 207382 & 738815 & ? \\
1 & 1 & 5 & 1 & 1 & 2 & 5 & 14 & 42 & 132 & 429 & 1430 & 4862 & 16795 & 58773 & 207906 & 742203 & $\NeqD{5}{n}$ \\
\midrule
\multicolumn{18}{r}{\textit{Continued on the next page.}} \\
\bottomrule
\end{tabular}
\end{center}
\caption{$\Neqabm{a}{b}{m}{n}$}
\label{tab:Nabmn}
\end{sidewaystable}

\begin{sidewaystable}[p]
\ContinuedFloat
\setlength{\tabcolsep}{3pt}
\begin{center}
\small
\begin{tabular}{rrr*{14}{r}l}
\toprule
\multicolumn{18}{l}{\textit{Continued from the previous page.}} \\
\midrule
$a$ & $b$ & $m$ & $n = {}$1 & 2 & 3 & 4 & 5 & 6 & 7 & 8 & 9 & 10 & 11 & 12 & 13 & 14 & Known identification \\
\midrule
1 & 2 & 5 & 1 & 1 & 2 & 5 & 14 & 42 & 132 & 429 & 1430 & 4862 & 16795 & 58773 & 207907 & 742219 & ? \\
1 & 4 & 8 & 1 & 1 & 2 & 5 & 14 & 42 & 132 & 429 & 1430 & 4862 & 16795 & 58773 & 207907 & 742220 & ? \\
1 & 9 & 9 & 1 & 1 & 2 & 5 & 14 & 42 & 132 & 429 & 1430 & 4862 & 16795 & 58773 & 207907 & 742220 & $\NeqL{9}{n}$ \\
2 & 3 & 12 & 1 & 1 & 2 & 5 & 14 & 42 & 132 & 429 & 1430 & 4862 & 16795 & 58773 & 207907 & 742220 & ? \\
1 & 10 & 10 & 1 & 1 & 2 & 5 & 14 & 42 & 132 & 429 & 1430 & 4862 & 16796 & 58785 & 207998 & 742780 & $\NeqL{10}{n}$ \\
1 & 5 & 6 & 1 & 1 & 2 & 5 & 14 & 42 & 132 & 429 & 1430 & 4862 & 16796 & 58786 & 208005 & 742795 & ? \\
1 & 1 & 6 & 1 & 1 & 2 & 5 & 14 & 42 & 132 & 429 & 1430 & 4862 & 16796 & 58786 & 208011 & 742885 & $\NeqD{6}{n}$ \\
1 & 2 & 8 & 1 & 1 & 2 & 5 & 14 & 42 & 132 & 429 & 1430 & 4862 & 16796 & 58786 & 208011 & 742885 & ? \\
1 & 3 & 9 & 1 & 1 & 2 & 5 & 14 & 42 & 132 & 429 & 1430 & 4862 & 16796 & 58786 & 208011 & 742885 & ? \\
1 & 5 & 10 & 1 & 1 & 2 & 5 & 14 & 42 & 132 & 429 & 1430 & 4862 & 16796 & 58786 & 208011 & 742885 & ? \\
1 & 6 & 8 & 1 & 1 & 2 & 5 & 14 & 42 & 132 & 429 & 1430 & 4862 & 16796 & 58786 & 208011 & 742885 & ? \\
1 & 6 & 9 & 1 & 1 & 2 & 5 & 14 & 42 & 132 & 429 & 1430 & 4862 & 16796 & 58786 & 208011 & 742885 & ? \\
1 & 11 & 11 & 1 & 1 & 2 & 5 & 14 & 42 & 132 & 429 & 1430 & 4862 & 16796 & 58786 & 208011 & 742885 & $\NeqL{11}{n}$ \\
1 & 12 & 12 & 1 & 1 & 2 & 5 & 14 & 42 & 132 & 429 & 1430 & 4862 & 16796 & 58786 & 208012 & 742899 & $\NeqL{12}{n}$ \\
1 & 1 & 7 & 1 & 1 & 2 & 5 & 14 & 42 & 132 & 429 & 1430 & 4862 & 16796 & 58786 & 208012 & 742900 & $\NeqD{7}{n}$ \\
1 & 1 & 8 & 1 & 1 & 2 & 5 & 14 & 42 & 132 & 429 & 1430 & 4862 & 16796 & 58786 & 208012 & 742900 & $\NeqD{8}{n}$ \\
1 & 1 & 9 & 1 & 1 & 2 & 5 & 14 & 42 & 132 & 429 & 1430 & 4862 & 16796 & 58786 & 208012 & 742900 & $\NeqD{9}{n}$ \\
1 & 1 & 10 & 1 & 1 & 2 & 5 & 14 & 42 & 132 & 429 & 1430 & 4862 & 16796 & 58786 & 208012 & 742900 & $\NeqD{10}{n}$ \\
1 & 1 & 11 & 1 & 1 & 2 & 5 & 14 & 42 & 132 & 429 & 1430 & 4862 & 16796 & 58786 & 208012 & 742900 & $\NeqD{11}{n}$ \\
1 & 1 & 12 & 1 & 1 & 2 & 5 & 14 & 42 & 132 & 429 & 1430 & 4862 & 16796 & 58786 & 208012 & 742900 & $\NeqD{12}{n}$ \\
1 & 1 & 13 & 1 & 1 & 2 & 5 & 14 & 42 & 132 & 429 & 1430 & 4862 & 16796 & 58786 & 208012 & 742900 & $\NeqD{13}{n}$ \\
1 & 1 & 14 & 1 & 1 & 2 & 5 & 14 & 42 & 132 & 429 & 1430 & 4862 & 16796 & 58786 & 208012 & 742900 & $\NeqD{14}{n}$ \\
1 & 2 & 7 & 1 & 1 & 2 & 5 & 14 & 42 & 132 & 429 & 1430 & 4862 & 16796 & 58786 & 208012 & 742900 & ? \\
\midrule
\multicolumn{18}{r}{\textit{Continued on the next page.}} \\
\bottomrule
\end{tabular}
\end{center}
\caption{$\Neqabm{a}{b}{m}{n}$}
\end{sidewaystable}

\begin{sidewaystable}[p]
\setlength{\tabcolsep}{3pt}
\ContinuedFloat
\begin{center}
\small
\begin{tabular}{rrr*{14}{r}l}
\toprule
\multicolumn{18}{l}{\textit{Continued from the previous page.}} \\
\midrule
$a$ & $b$ & $m$ & $n = {}$1 & 2 & 3 & 4 & 5 & 6 & 7 & 8 & 9 & 10 & 11 & 12 & 13 & 14 & Known identification \\
\midrule
1 & 2 & 9 & 1 & 1 & 2 & 5 & 14 & 42 & 132 & 429 & 1430 & 4862 & 16796 & 58786 & 208012 & 742900 & ? \\
1 & 2 & 10 & 1 & 1 & 2 & 5 & 14 & 42 & 132 & 429 & 1430 & 4862 & 16796 & 58786 & 208012 & 742900 & ? \\
1 & 2 & 11 & 1 & 1 & 2 & 5 & 14 & 42 & 132 & 429 & 1430 & 4862 & 16796 & 58786 & 208012 & 742900 & ? \\
1 & 2 & 12 & 1 & 1 & 2 & 5 & 14 & 42 & 132 & 429 & 1430 & 4862 & 16796 & 58786 & 208012 & 742900 & ? \\
1 & 2 & 13 & 1 & 1 & 2 & 5 & 14 & 42 & 132 & 429 & 1430 & 4862 & 16796 & 58786 & 208012 & 742900 & ? \\
1 & 2 & 14 & 1 & 1 & 2 & 5 & 14 & 42 & 132 & 429 & 1430 & 4862 & 16796 & 58786 & 208012 & 742900 & ? \\
1 & 3 & 7 & 1 & 1 & 2 & 5 & 14 & 42 & 132 & 429 & 1430 & 4862 & 16796 & 58786 & 208012 & 742900 & ? \\
1 & 3 & 8 & 1 & 1 & 2 & 5 & 14 & 42 & 132 & 429 & 1430 & 4862 & 16796 & 58786 & 208012 & 742900 & ? \\
1 & 3 & 10 & 1 & 1 & 2 & 5 & 14 & 42 & 132 & 429 & 1430 & 4862 & 16796 & 58786 & 208012 & 742900 & ? \\
1 & 3 & 11 & 1 & 1 & 2 & 5 & 14 & 42 & 132 & 429 & 1430 & 4862 & 16796 & 58786 & 208012 & 742900 & ? \\
1 & 3 & 12 & 1 & 1 & 2 & 5 & 14 & 42 & 132 & 429 & 1430 & 4862 & 16796 & 58786 & 208012 & 742900 & ? \\
1 & 3 & 13 & 1 & 1 & 2 & 5 & 14 & 42 & 132 & 429 & 1430 & 4862 & 16796 & 58786 & 208012 & 742900 & ? \\
1 & 3 & 14 & 1 & 1 & 2 & 5 & 14 & 42 & 132 & 429 & 1430 & 4862 & 16796 & 58786 & 208012 & 742900 & ? \\
1 & 4 & 9 & 1 & 1 & 2 & 5 & 14 & 42 & 132 & 429 & 1430 & 4862 & 16796 & 58786 & 208012 & 742900 & ? \\
1 & 4 & 10 & 1 & 1 & 2 & 5 & 14 & 42 & 132 & 429 & 1430 & 4862 & 16796 & 58786 & 208012 & 742900 & ? \\
1 & 4 & 12 & 1 & 1 & 2 & 5 & 14 & 42 & 132 & 429 & 1430 & 4862 & 16796 & 58786 & 208012 & 742900 & ? \\
1 & 4 & 13 & 1 & 1 & 2 & 5 & 14 & 42 & 132 & 429 & 1430 & 4862 & 16796 & 58786 & 208012 & 742900 & ? \\
1 & 4 & 14 & 1 & 1 & 2 & 5 & 14 & 42 & 132 & 429 & 1430 & 4862 & 16796 & 58786 & 208012 & 742900 & ? \\
1 & 5 & 8 & 1 & 1 & 2 & 5 & 14 & 42 & 132 & 429 & 1430 & 4862 & 16796 & 58786 & 208012 & 742900 & ? \\
1 & 5 & 11 & 1 & 1 & 2 & 5 & 14 & 42 & 132 & 429 & 1430 & 4862 & 16796 & 58786 & 208012 & 742900 & ? \\
1 & 5 & 12 & 1 & 1 & 2 & 5 & 14 & 42 & 132 & 429 & 1430 & 4862 & 16796 & 58786 & 208012 & 742900 & ? \\
1 & 5 & 13 & 1 & 1 & 2 & 5 & 14 & 42 & 132 & 429 & 1430 & 4862 & 16796 & 58786 & 208012 & 742900 & ? \\
1 & 6 & 7 & 1 & 1 & 2 & 5 & 14 & 42 & 132 & 429 & 1430 & 4862 & 16796 & 58786 & 208012 & 742900 & ? \\
1 & 6 & 10 & 1 & 1 & 2 & 5 & 14 & 42 & 132 & 429 & 1430 & 4862 & 16796 & 58786 & 208012 & 742900 & ? \\
\midrule
\multicolumn{18}{r}{\textit{Continued on the next page.}} \\
\bottomrule
\end{tabular}
\end{center}
\caption{$\Neqabm{a}{b}{m}{n}$}
\end{sidewaystable}

\begin{sidewaystable}[p]
\setlength{\tabcolsep}{3pt}
\ContinuedFloat
\begin{center}
\small
\begin{tabular}{rrr*{14}{r}l}
\toprule
\multicolumn{18}{l}{\textit{Continued from the previous page.}} \\
\midrule
$a$ & $b$ & $m$ & $n = {}$1 & 2 & 3 & 4 & 5 & 6 & 7 & 8 & 9 & 10 & 11 & 12 & 13 & 14 & Known identification \\
\midrule
1 & 6 & 12 & 1 & 1 & 2 & 5 & 14 & 42 & 132 & 429 & 1430 & 4862 & 16796 & 58786 & 208012 & 742900 & ? \\
1 & 6 & 13 & 1 & 1 & 2 & 5 & 14 & 42 & 132 & 429 & 1430 & 4862 & 16796 & 58786 & 208012 & 742900 & ? \\
1 & 6 & 14 & 1 & 1 & 2 & 5 & 14 & 42 & 132 & 429 & 1430 & 4862 & 16796 & 58786 & 208012 & 742900 & ? \\
1 & 7 & 8 & 1 & 1 & 2 & 5 & 14 & 42 & 132 & 429 & 1430 & 4862 & 16796 & 58786 & 208012 & 742900 & ? \\
1 & 7 & 11 & 1 & 1 & 2 & 5 & 14 & 42 & 132 & 429 & 1430 & 4862 & 16796 & 58786 & 208012 & 742900 & ? \\
1 & 7 & 12 & 1 & 1 & 2 & 5 & 14 & 42 & 132 & 429 & 1430 & 4862 & 16796 & 58786 & 208012 & 742900 & ? \\
1 & 7 & 14 & 1 & 1 & 2 & 5 & 14 & 42 & 132 & 429 & 1430 & 4862 & 16796 & 58786 & 208012 & 742900 & ? \\
1 & 8 & 9 & 1 & 1 & 2 & 5 & 14 & 42 & 132 & 429 & 1430 & 4862 & 16796 & 58786 & 208012 & 742900 & ? \\
1 & 8 & 10 & 1 & 1 & 2 & 5 & 14 & 42 & 132 & 429 & 1430 & 4862 & 16796 & 58786 & 208012 & 742900 & ? \\
1 & 8 & 12 & 1 & 1 & 2 & 5 & 14 & 42 & 132 & 429 & 1430 & 4862 & 16796 & 58786 & 208012 & 742900 & ? \\
1 & 8 & 14 & 1 & 1 & 2 & 5 & 14 & 42 & 132 & 429 & 1430 & 4862 & 16796 & 58786 & 208012 & 742900 & ? \\
1 & 9 & 10 & 1 & 1 & 2 & 5 & 14 & 42 & 132 & 429 & 1430 & 4862 & 16796 & 58786 & 208012 & 742900 & ? \\
1 & 9 & 12 & 1 & 1 & 2 & 5 & 14 & 42 & 132 & 429 & 1430 & 4862 & 16796 & 58786 & 208012 & 742900 & ? \\
1 & 9 & 14 & 1 & 1 & 2 & 5 & 14 & 42 & 132 & 429 & 1430 & 4862 & 16796 & 58786 & 208012 & 742900 & ? \\
1 & 10 & 11 & 1 & 1 & 2 & 5 & 14 & 42 & 132 & 429 & 1430 & 4862 & 16796 & 58786 & 208012 & 742900 & ? \\
1 & 10 & 12 & 1 & 1 & 2 & 5 & 14 & 42 & 132 & 429 & 1430 & 4862 & 16796 & 58786 & 208012 & 742900 & ? \\
1 & 10 & 14 & 1 & 1 & 2 & 5 & 14 & 42 & 132 & 429 & 1430 & 4862 & 16796 & 58786 & 208012 & 742900 & ? \\
1 & 11 & 12 & 1 & 1 & 2 & 5 & 14 & 42 & 132 & 429 & 1430 & 4862 & 16796 & 58786 & 208012 & 742900 & ? \\
1 & 12 & 13 & 1 & 1 & 2 & 5 & 14 & 42 & 132 & 429 & 1430 & 4862 & 16796 & 58786 & 208012 & 742900 & ? \\
1 & 12 & 14 & 1 & 1 & 2 & 5 & 14 & 42 & 132 & 429 & 1430 & 4862 & 16796 & 58786 & 208012 & 742900 & ? \\
1 & 13 & 13 & 1 & 1 & 2 & 5 & 14 & 42 & 132 & 429 & 1430 & 4862 & 16796 & 58786 & 208012 & 742900 & $\NeqL{13}{n}$ \\
1 & 13 & 14 & 1 & 1 & 2 & 5 & 14 & 42 & 132 & 429 & 1430 & 4862 & 16796 & 58786 & 208012 & 742900 & ? \\
1 & 14 & 14 & 1 & 1 & 2 & 5 & 14 & 42 & 132 & 429 & 1430 & 4862 & 16796 & 58786 & 208012 & 742900 & $\NeqL{14}{n}$ \\
\bottomrule
\end{tabular}
\end{center}
\caption{$\Neqabm{a}{b}{m}{n}$}
\end{sidewaystable}

\begin{sidewaystable}[p]
\setlength{\tabcolsep}{3pt}
\begin{center}
\small
\begin{tabular}{rrr*{14}{r}}
\toprule
$u$ & $v$ & $w$ & $n = {}$1 & 2 & 3 & 4 & 5 & 6 & 7 & 8 & 9 & 10 & 11 & 12 & 13 & 14 \\
\midrule
0 & 1 & 1 & 1 & 1 & 1 & 1 & 1 & 1 & 1 & 1 & 1 & 1 & 1 & 1 & 1 & 1 \\
0 & 2 & 1 & 1 & 1 & 2 & 4 & 8 & 16 & 32 & 64 & 128 & 256 & 512 & 1024 & 2048 & 4096 \\
0 & 3 & 1 & 1 & 1 & 2 & 5 & 13 & 35 & 96 & 267 & 750 & 2123 & 6046 & 17303 & 49721 & 143365 \\
0 & 4 & 1 & 1 & 1 & 2 & 5 & 14 & 41 & 124 & 384 & 1210 & 3865 & 12482 & 40677 & 133572 & 441468 \\
0 & 5 & 1 &  1 & 1 & 2 & 5 & 14 & 42 & 131 & 420 & 1375 & 4576 & 15431 & 52603 & 180957 & 627340 \\
0 & 1 & 2 & 1 & 1 & 2 & 4 & 8 & 16 & 32 & 64 & 128 & 256 & 512 & 1024 & 2048 & 4096 \\
1 & 1 & 2 & 1 & 1 & 2 & 5 & 10 & 21 & 42 & 85 & 170 & 341 & 682 & 1365 & 2730 & 5461 \\
0 & 2 & 2 & 1 & 1 & 2 & 5 & 13 & 35 & 96 & 267 & 750 & 2123 & 6046 & 17303 & 49721 & 143365 \\
1 & 2 & 2 & 1 & 1 & 2 & 5 & 14 & 42 & 131 & 420 & 1374 & 4561 & 15306 & 51793 & 176404 & 603990 \\
0 & 3 & 2 & 1 & 1 & 2 & 5 & 14 & 41 & 124 & 384 & 1210 & 3865 & 12482 & 40677 & 133572 & 441468 \\
1 & 3 & 2 & 1 & 1 & 2 & 5 & 14 & 42 & 132 & 429 & 1429 & 4851 & 16718 & 58331 & 205631 & 731257 \\
0 & 4 & 2 & 1 & 1 & 2 & 5 & 14 & 42 & 131 & 420 & 1375 & 4576 & 15431 & 52603 & 180957 & 627340 \\
1 & 4 & 2 & 1 & 1 & 2 & 5 & 14 & 42 & 132 & 429 & 1430 & 4862 & 16795 & 58773 & 207907 & 742220 \\
0 & 5 & 2 & 1 & 1 & 2 & 5 & 14 & 42 & 132 & 428 & 1420 & 4796 & 16432 & 56966 & 199444 & 704146 \\
1 & 5 & 2 & 1 & 1 & 2 & 5 & 14 & 42 & 132 & 429 & 1430 & 4862 & 16796 & 58786 & 208011 & 742885 \\
0 & 1 & 3 & 1 & 1 & 2 & 5 & 13 & 35 & 96 & 267 & 750 & 2123 & 6046 & 17303 & 49721 & 143365 \\
1 & 1 & 3 & 1 & 1 & 2 & 5 & 14 & 42 & 128 & 390 & 1185 & 3586 & 10862 & 32929 & 99883 & 303000 \\
2 & 1 & 3 & 1 & 1 & 2 & 5 & 14 & 42 & 129 & 398 & 1223 & 3752 & 11510 & 35305 & 108217 & 331434 \\
0 & 2 & 3 & 1 & 1 & 2 & 5 & 14 & 41 & 124 & 384 & 1210 & 3865 & 12482 & 40677 & 133572 & 441468 \\
1 & 2 & 3 & 1 & 1 & 2 & 5 & 14 & 42 & 132 & 429 & 1430 & 4861 & 16784 & 58695 & 207452 & 739839 \\
2 & 2 & 3 & 1 & 1 & 2 & 5 & 14 & 42 & 132 & 429 & 1430 & 4861 & 16784 & 58695 & 207450 & 739810 \\
0 & 3 & 3 & 1 & 1 & 2 & 5 & 14 & 42 & 131 & 420 & 1375 & 4576 & 15431 & 52603 & 180957 & 627340 \\
1 & 3 & 3 & 1 & 1 & 2 & 5 & 14 & 42 & 132 & 429 & 1430 & 4862 & 16796 & 58786 & 208011 & 742885 \\
2 & 3 & 3 & 1 & 1 & 2 & 5 & 14 & 42 & 132 & 429 & 1430 & 4862 & 16796 & 58786 & 208011 & 742885 \\
0 & 4 & 3 & 1 & 1 & 2 & 5 & 14 & 42 & 132 & 428 & 1420 & 4796 & 16432 & 56966 & 199444 & 704146 \\
\midrule
\multicolumn{17}{r}{\textit{Continued on the next page.}} \\
\bottomrule
\end{tabular}
\end{center}
\caption{$T_{\Lambda,n}$ with $\Lambda = \ZZ(u,v) + \ZZ(w,0)$.}
\label{tab:grid-eq}
\end{sidewaystable}

\begin{sidewaystable}[p]
\ContinuedFloat
\setlength{\tabcolsep}{3pt}
\begin{center}
\small
\begin{tabular}{rrr*{14}{r}}
\toprule
\multicolumn{17}{l}{\textit{Continued from the previous page.}} \\
\midrule
$u$ & $v$ & $w$ & $n = {}$1 & 2 & 3 & 4 & 5 & 6 & 7 & 8 & 9 & 10 & 11 & 12 & 13 & 14 \\
\midrule
1 & 4 & 3 & 1 & 1 & 2 & 5 & 14 & 42 & 132 & 429 & 1430 & 4862 & 16796 & 58786 & 208012 & 742900 \\
2 & 4 & 3 & 1 & 1 & 2 & 5 & 14 & 42 & 132 & 429 & 1430 & 4862 & 16796 & 58786 & 208012 & 742900 \\
0 & 5 & 3 & 1 & 1 & 2 & 5 & 14 & 42 & 132 & 429 & 1429 & 4851 & 16718 & 58331 & 205632 & 731272 \\
1 & 5 & 3 & 1 & 1 & 2 & 5 & 14 & 42 & 132 & 429 & 1430 & 4862 & 16796 & 58786 & 208012 & 742900 \\
2 & 5 & 3 & 1 & 1 & 2 & 5 & 14 & 42 & 132 & 429 & 1430 & 4862 & 16796 & 58786 & 208012 & 742900 \\
0 & 1 & 4 & 1 & 1 & 2 & 5 & 14 & 41 & 124 & 384 & 1210 & 3865 & 12482 & 40677 & 133572 & 441468 \\
1 & 1 & 4 & 1 & 1 & 2 & 5 & 14 & 42 & 132 & 429 & 1425 & 4807 & 16402 & 56472 & 195860 & 683420 \\
2 & 1 & 4 & 1 & 1 & 2 & 5 & 14 & 42 & 131 & 420 & 1374 & 4561 & 15306 & 51793 & 176404 & 603990 \\
3 & 1 & 4 & 1 & 1 & 2 & 5 & 14 & 42 & 132 & 429 & 1429 & 4849 & 16689 & 58074 & 203839 & 720429 \\
0 & 2 & 4 & 1 & 1 & 2 & 5 & 14 & 42 & 131 & 420 & 1375 & 4576 & 15431 & 52603 & 180957 & 627340 \\
1 & 2 & 4 & 1 & 1 & 2 & 5 & 14 & 42 & 132 & 429 & 1430 & 4862 & 16796 & 58786 & 208011 & 742885 \\
2 & 2 & 4 & 1 & 1 & 2 & 5 & 14 & 42 & 132 & 429 & 1429 & 4851 & 16718 & 58331 & 205631 & 731257 \\
3 & 2 & 4 & 1 & 1 & 2 & 5 & 14 & 42 & 132 & 429 & 1430 & 4862 & 16796 & 58786 & 208011 & 742885 \\
0 & 3 & 4 & 1 & 1 & 2 & 5 & 14 & 42 & 132 & 428 & 1420 & 4796 & 16432 & 56966 & 199444 & 704146 \\
1 & 3 & 4 & 1 & 1 & 2 & 5 & 14 & 42 & 132 & 429 & 1430 & 4862 & 16796 & 58786 & 208012 & 742900 \\
2 & 3 & 4 & 1 & 1 & 2 & 5 & 14 & 42 & 132 & 429 & 1430 & 4862 & 16795 & 58773 & 207907 & 742220 \\
3 & 3 & 4 & 1 & 1 & 2 & 5 & 14 & 42 & 132 & 429 & 1430 & 4862 & 16796 & 58786 & 208012 & 742900 \\
0 & 4 & 4 & 1 & 1 & 2 & 5 & 14 & 42 & 132 & 429 & 1429 & 4851 & 16718 & 58331 & 205632 & 731272 \\
1 & 4 & 4 & 1 & 1 & 2 & 5 & 14 & 42 & 132 & 429 & 1430 & 4862 & 16796 & 58786 & 208012 & 742900 \\
2 & 4 & 4 & 1 & 1 & 2 & 5 & 14 & 42 & 132 & 429 & 1430 & 4862 & 16796 & 58786 & 208011 & 742885 \\
3 & 4 & 4 & 1 & 1 & 2 & 5 & 14 & 42 & 132 & 429 & 1430 & 4862 & 16796 & 58786 & 208012 & 742900 \\
0 & 5 & 4 & 1 & 1 & 2 & 5 & 14 & 42 & 132 & 429 & 1430 & 4861 & 16784 & 58695 & 207452 & 739840 \\
1 & 5 & 4 & 1 & 1 & 2 & 5 & 14 & 42 & 132 & 429 & 1430 & 4862 & 16796 & 58786 & 208012 & 742900 \\
2 & 5 & 4 & 1 & 1 & 2 & 5 & 14 & 42 & 132 & 429 & 1430 & 4862 & 16796 & 58786 & 208012 & 742900 \\
3 & 5 & 4 & 1 & 1 & 2 & 5 & 14 & 42 & 132 & 429 & 1430 & 4862 & 16796 & 58786 & 208012 & 742900 \\
\midrule
\multicolumn{17}{r}{\textit{Continued on the next page.}} \\
\bottomrule
\end{tabular}
\end{center}
\caption{$T_{\Lambda,n}$ with $\Lambda = \ZZ(u,v) + \ZZ(w,0)$.}
\end{sidewaystable}

\begin{sidewaystable}[p]
\setlength{\tabcolsep}{3pt}
\ContinuedFloat
\begin{center}
\small
\begin{tabular}{rrr*{14}{r}}
\toprule
\multicolumn{17}{l}{\textit{Continued from the previous page.}} \\
\midrule
$u$ & $v$ & $w$ & $n = {}$1 & 2 & 3 & 4 & 5 & 6 & 7 & 8 & 9 & 10 & 11 & 12 & 13 & 14 \\
\midrule
0 & 1 & 5 & 1 & 1 & 2 & 5 & 14 & 42 & 131 & 420 & 1375 & 4576 & 15431 & 52603 & 180957 & 627340 \\
1 & 1 & 5 & 1 & 1 & 2 & 5 & 14 & 42 & 132 & 429 & 1430 & 4862 & 16790 & 58708 & 207382 & 738815 \\
2 & 1 & 5 & 1 & 1 & 2 & 5 & 14 & 42 & 132 & 429 & 1430 & 4862 & 16795 & 58773 & 207907 & 742219 \\
3 & 1 & 5 & 1 & 1 & 2 & 5 & 14 & 42 & 132 & 429 & 1430 & 4862 & 16795 & 58773 & 207907 & 742219 \\
4 & 1 & 5 & 1 & 1 & 2 & 5 & 14 & 42 & 132 & 429 & 1430 & 4862 & 16795 & 58773 & 207906 & 742203 \\
0 & 2 & 5 & 1 & 1 & 2 & 5 & 14 & 42 & 132 & 428 & 1420 & 4796 & 16432 & 56966 & 199444 & 704146 \\
1 & 2 & 5 & 1 & 1 & 2 & 5 & 14 & 42 & 132 & 429 & 1430 & 4862 & 16796 & 58786 & 208012 & 742900 \\
2 & 2 & 5 & 1 & 1 & 2 & 5 & 14 & 42 & 132 & 429 & 1430 & 4862 & 16796 & 58786 & 208012 & 742900 \\
3 & 2 & 5 & 1 & 1 & 2 & 5 & 14 & 42 & 132 & 429 & 1430 & 4862 & 16796 & 58786 & 208012 & 742900 \\
4 & 2 & 5 & 1 & 1 & 2 & 5 & 14 & 42 & 132 & 429 & 1430 & 4862 & 16796 & 58786 & 208012 & 742900 \\
0 & 3 & 5 & 1 & 1 & 2 & 5 & 14 & 42 & 132 & 429 & 1429 & 4851 & 16718 & 58331 & 205632 & 731272 \\
1 & 3 & 5 & 1 & 1 & 2 & 5 & 14 & 42 & 132 & 429 & 1430 & 4862 & 16796 & 58786 & 208012 & 742900 \\
2 & 3 & 5 & 1 & 1 & 2 & 5 & 14 & 42 & 132 & 429 & 1430 & 4862 & 16796 & 58786 & 208012 & 742900 \\
3 & 3 & 5 & 1 & 1 & 2 & 5 & 14 & 42 & 132 & 429 & 1430 & 4862 & 16796 & 58786 & 208012 & 742900 \\
4 & 3 & 5 & 1 & 1 & 2 & 5 & 14 & 42 & 132 & 429 & 1430 & 4862 & 16796 & 58786 & 208012 & 742900 \\
0 & 4 & 5 & 1 & 1 & 2 & 5 & 14 & 42 & 132 & 429 & 1430 & 4861 & 16784 & 58695 & 207452 & 739840 \\
1 & 4 & 5 & 1 & 1 & 2 & 5 & 14 & 42 & 132 & 429 & 1430 & 4862 & 16796 & 58786 & 208012 & 742900 \\
2 & 4 & 5 & 1 & 1 & 2 & 5 & 14 & 42 & 132 & 429 & 1430 & 4862 & 16796 & 58786 & 208012 & 742900 \\
3 & 4 & 5 & 1 & 1 & 2 & 5 & 14 & 42 & 132 & 429 & 1430 & 4862 & 16796 & 58786 & 208012 & 742900 \\
4 & 4 & 5 & 1 & 1 & 2 & 5 & 14 & 42 & 132 & 429 & 1430 & 4862 & 16796 & 58786 & 208012 & 742900 \\
0 & 5 & 5 & 1 & 1 & 2 & 5 & 14 & 42 & 132 & 429 & 1430 & 4862 & 16795 & 58773 & 207907 & 742220 \\
1 & 5 & 5 & 1 & 1 & 2 & 5 & 14 & 42 & 132 & 429 & 1430 & 4862 & 16796 & 58786 & 208012 & 742900 \\
2 & 5 & 5 & 1 & 1 & 2 & 5 & 14 & 42 & 132 & 429 & 1430 & 4862 & 16796 & 58786 & 208012 & 742900 \\
3 & 5 & 5 & 1 & 1 & 2 & 5 & 14 & 42 & 132 & 429 & 1430 & 4862 & 16796 & 58786 & 208012 & 742900 \\
4 & 5 & 5 & 1 & 1 & 2 & 5 & 14 & 42 & 132 & 429 & 1430 & 4862 & 16796 & 58786 & 208012 & 742900 \\
\bottomrule
\end{tabular}
\end{center}
\caption{$T_{\Lambda,n}$ with $\Lambda = \ZZ(u,v) + \ZZ(w,0)$.}
\end{sidewaystable}


\section*{Declarations}

\noindent
\textbf{Artificial intelligence (AI) use declaration}

\smallskip\noindent
The authors used ChatGPT (OpenAI) to assist in the mathematical and editorial review of the manuscript, literature and reference searches, \LaTeX{} revisions, and computational checks.
The tool also suggested potential mathematical additions that were subsequently evaluated by the authors; material derived from these suggestions appears in Remark~\ref{rem:m}, Example~\ref{ex:finite}, and Remark~\ref{rem:single-identity}.
All AI\hyp{}generated suggestions, analyses, and text were independently verified by the authors before incorporation into the manuscript.
The authors take full responsibility for the selection, accuracy, and presentation of all material contained in the paper.

\end{document}